\documentclass{amsart}

\usepackage{amsmath, amscd, amsfonts, amssymb, amsthm, bbm, tikz, times, dsfont, enumerate, graphicx, hyperref, mathtools}
\usepackage{fullpage}
\usepackage{appendix}
\newtheorem{theorem}{Theorem}[section]
\newtheorem{lemma}[theorem]{Lemma}

\newtheorem{proposition}[theorem]{Proposition}

\newtheorem{claim}[theorem]{Claim}

\theoremstyle{definition}
\newtheorem{definition}[theorem]{Definition}
\newtheorem{assumption}[theorem]{Assumption}

\theoremstyle{remark}
\newtheorem{remark}[theorem]{Remark}

\numberwithin{equation}{section}

\newcommand{\prob}{\mathbb{P}}

\newcommand{\R}{\mathbb{R}}
\newcommand{\C}{\mathbb{C}}
\newcommand{\E}{\mathbb{E}}

\newcommand{\eps}{\epsilon}
\newcommand{\indic}{\mathbbm{1}}

\newcommand{\cumul}{{\scriptstyle\mathcal{K}}}
\newcommand{\const}{\mathcal{C}}

\newcommand{\asedge}{\widetilde{\mathcal{L}}}

\newcommand{\vv}{\mathbf{v}}

\newcommand{\tcb}{\textcolor{blue}}

\newcommand{\deri}{\partial_{ij}}
\newcommand{\tz}{\tilde{z}}
\newcommand{\bigO}{\mathcal{O}}

\newcommand{\expec}[1]{{\mathbb E}\left[ #1 \right]}
\newcommand{\size}[1]{\left| #1 \right|}
\newcommand{\paren}[1]{\left( #1 \right)}
\newcommand{\matnorm}[1]{\lVert #1 \rVert}
\newcommand{\combi}[2]{\begin{pmatrix}
		#1 \\ #2
\end{pmatrix}}

\DeclareMathOperator{\Mod}{mod}

\begin{document}

\title{Higher order fluctuations of extremal eigenvalues of sparse random matrices}

\author{Jaehun~Lee}
\address{Department of Mathematics, HKUST, Clear Water Bay, Hong Kong}
%\curraddr{-}
\email{jaehun.lee@ust.hk}

\thanks{\tcb{\emph{Note.}~The first draft of this paper was done while the author was at KAIST, South Korea.}}



%\date{\today}

\begin{abstract}
We consider extremal eigenvalues of sparse random matrices, a class of random matrices including the adjacency matrices of Erd\H{o}s-R\'{e}nyi graphs $\mathcal{G}(N,p)$. Recently, it was shown that the leading order fluctuations of extremal eigenvalues are given by a single random variable associated with the total degree of the graph (\textit{Ann. Probab.}, 48(2):916--962, 2020; \textit{Probab. Theory Related Fields}, 180:985--1056, 2021). We construct a sequence of random correction terms to capture higher (sub-leading) order fluctuations of extremal eigenvalues in the regime $N^{\epsilon} < pN < N^{1/3-\epsilon}$. Using these random correction terms, we prove a local law up to a shifted edge and recover the rigidity of extremal eigenvalues under some corrections for $pN>N^{\epsilon}$.
\end{abstract}

\maketitle

\section{Introduction}
An $N\times N$ random matrix is said to be \emph{sparse} if the number of nonzero entries per row is much less than $N$ on average as $N\to\infty$. A motivating example is the adjacency matrix of the sparse \emph{Erd\H{o}s-R\'{e}nyi graph} $\mathcal{G}(N,p)$, a random graph on $N$ vertices in which each edge independently exists with probability $p\ll 1$. As the canonical model of random graphs, Erd\H{o}s-R\'{e}nyi graphs have been used in many areas such as combinatorics, network theory and mathematical physics \cite{Bollobas01, Chung97, JLR00}. It is important to study the spectral statistics of Erd\H{o}s-R\'{e}nyi graphs since their eigenvalues and eigenvectors contain fundamental information, and
accordingly, have many applications: combinatorial optimization, spectral partitioning, community detection, etc \cite{BS16,Chung97,MP91,PSL90}.

The Erd\H{o}s-R\'{e}nyi graph $\mathcal{G}(N,p)$ exhibits some threshold phenomena with regard to the probability $p$. In the seminal works \cite{ER59, ER60}, a connectivity transition was shown around $pN=\log{N}$:
\begin{enumerate}
	\item If $pN>(1+\eps)\log{N}$, the Erd\H{o}s-R\'{e}nyi graph $\mathcal{G}(N,p)$ is almost surely connected.  
	\item If $pN<(1-\eps)\log{N}$, the Erd\H{o}s-R\'{e}nyi graph $\mathcal{G}(N,p)$ almost surely contains isolated vertices, and thus it is disconnected.
\end{enumerate}
In this paper, we consider the regime $pN>N^{\eps}$, which is included in the super-critical regime, $pN>(1+\eps)\log{N}$. A pioneering work for this case is a series of papers by Erd\H{o}s, Knowles, Yau and Yin \cite{EKYY12,EKYY13}. They established many outstanding results such as a local law, eigenvalue rigidity and universality for sparse random matrices.

A so-called local law is an essential ingredient to understand the local behavior of eigenvalues. A three-step strategy based on local laws was developed to solve the well-known \emph{Wigner-Dyson-Mehta} universality conjecture for Wigner matrices \cite{ESY11,EYY12-1,EYY12-2}, an open problem for nearly 50 years. A similar approach is believed to be valid for sparse matrices but sparsity introduces some additional technical challenges to overcome, especially near the spectral edge. Recently, bulk universality was established when $pN > N^{\eps}$ \cite{He20, HLY15} but progress on edge universality has been slower since extreme eigenvalues fluctuate in a more pronounced way for sparse matrices. In \cite{EKYY12,EKYY13}, both rigidity and universality of the extremal eigenvalues were proved only for $pN>N^{2/3+\eps}$. One reason for the restriction on the regime of $pN$, is that large deviation estimates for sparse matrices are not powerful enough: they contain a fixed power of $(pN)^{-1/2}$. (See \cite[Lemma 3.8]{EKYY13} for more details.) 

Later Lee and Schnelli discovered that local law estimates at the edge can be strengthened by introducing a deterministic correction of the semicircle law \cite{LS18}. Using their improved local law, they established edge rigidity and universality for $pN>N^{1/3+\eps}$ with respect to a deterministically shifted edge $\mathcal{L}=2+O((pN)^{-1})$ where we note that $+2$ is the right edge of the standard semicircle law. (See \cite[Theorem 2.9, Theorem 2.10]{LS18} for precise statements.) They also expected that the order of the fluctuations of the extremal eigenvalues would start to exceed $N^{-2/3}$, the typical order of Tracy-Widom fluctuations, when $pN < N^{1/3-\eps}$.

In recent work by Huang, Landon and Yau \cite{HLY20}, the authors indeed confirmed a transition from Tracy-Widom to Gaussian fluctuations for extreme eigenvalues around $pN\sim N^{1/3}$. The main observation is that the random correction term $\mathcal{X}$ defined by
\begin{align}\label{eq: mathcal X}
	\mathcal{X} \coloneqq \frac{1}{N}\sum_{i,j}\paren{h_{ij}^{2}-\frac{1}{N}},
\end{align}
where the entries of an $N\times N$ sparse random matrix are denoted by $(h_{ij})$, becomes the leading order term of the fluctuations as $pN < N^{1/3-\eps}$. In addition, they recovered edge rigidity and Tracy-Widom fluctuations with respect to the randomly shifted edge $L+\mathcal{X}$ under the condition that $N^{2/9+\eps} < pN < N^{1/3-\eps}$. One of their novel ideas involved subtracting the main source of Gaussian fluctuations $\mathcal{X}$. However the results in \cite{HLY20} are valid only when $pN > N^{2/9+\eps}$, due to technical constraints.

Finally, He and Knowles resolved some subtle technical issues and showed that the extremal eigenvalues have Gaussian fluctuations in the regime $N^{\eps} < pN < N^{1/3-\eps}$ \cite{HK21}. Thus, the random quantity $\mathcal{X}$ is the leading order term of fluctuations when $pN < N^{1/3-\eps}$. One can also ask about higher (sub-leading) order fluctuations of extremal eigenvalues. One might anticipate higher order fluctuations to exist and that new random correction terms should capture such fluctuations. Indeed in \cite{HLY20}, the authors suspected higher order fluctuations would depend on subgraph counts. In \cite{HK21}, the authors claimed, for new random corrections terms, that there is an infinite hierarchy of random variables which is strongly correlated and asymptotically Gaussian, and the random variable $\mathcal{X}$ would be only the leading order fluctuations of extremal eigenvalues.

In this paper, we investigate such higher order fluctuations of extremal eigenvalues of sparse random matrices in the regime $pN > N^{\eps}$. We construct a series of random correction terms $(\mathcal{Z}_{n})$ to control higher order fluctuations of extremal eigenvalues. (See \eqref{eq: form of correction terms} for more detail.) Using these random correction terms $(\mathcal{Z}_{n})$, we derive a new (randomly) shifted edge $\mathcal{L}\equiv\mathcal{L}(\mathcal{Z}_{1}, \mathcal{Z}_{2}, \cdots, \mathcal{Z}_{\ell})$ which is a polynomial in the variables $(\mathcal{Z}_{n})_{1\le n\le \ell}$, and also obtain a local law near this shifted edge $\mathcal{L}$ (Theorem \ref{thm: local law}). As a result, we prove that the extremal eigenvalues are concentrated at the shifted edge $\mathcal{L}$ with scale $N^{-2/3}$, which is precisely the eigenvalue rigidity estimate near the edge (Theorem \ref{thm: rigidity}).

One of the main ingredients is the so-called recursive moment estimate (RME), Proposition \ref{prop: RME}. The RME is a moment estimate of a self-consistent polynomial. In \cite{HK21,HLY20,LS18}, an approach using the RME was developed and a local law was proved as a consequence of the RME and stability analysis. The RME in \cite{HLY20,LS18} has a bound containing powers of $(pN)^{-1/2}$ so the resulting local law bound has powers of $(pN)^{-1/2}$ too. Inspired by \cite{HK21}, we tried to find a way to avoid such powers of $(pN)^{-1/2}$ in the RME bound. As it turns out, by including sufficiently many random correction terms, we are able to obtain the RME without any powers of $(pN)^{-1/2}$, which consequently resulted in a nearly optimal local law bound. From standard arguments using the local law \cite{ERSY10}, we can then establish the desired rigidity estimate.
% We would like to emphasize that this result gives strong evidence supporting the notion that Tracy-Widom fluctuations may be recovered by subtracting higher order fluctuations.

\section{Definitions and main results}
\subsection{Basic notions and conventions}
To formulate the main results, we introduce some definitions and the notation as preliminaries. Let $N$ be the fundamental positive-integer parameter. % A fundamental parameter, usually corresponding to the dimension of a given random matrix.
Dealing with $N$-dependent quantities, we almost always omit the $N$-dependence to ease notation, unless otherwise stated.
We use $c>0$ for a small universal constant while we denote by $C>0$ a large universal constant throughout this paper. Their values may change between occurrences.
\begin{definition}[High probability event]
	Let $E\equiv E_{N}$ be an event parametrized by $N$. We say that $E$ holds \emph{with very high probability} if for any (large) $D>0$ there exists a constant $C$ such that $\prob(\Omega^{c})\le CN^{-D}$ for all $N$.
\end{definition}
\begin{definition}[Stochastic domination]
	Let $X\equiv X_{N}$ and $Y\equiv Y_{N}$ be random (or deterministic) variables depending on $N$. We say that $X$ is \emph{stochastically dominated} by $Y$ if for any (small) $\eps>0$ and (large) $D>0$ there exists a constant $C$ such that $\prob[|X|>N^{\eps}Y]\le CN^{-D}$ for all $N$. If $Y$ stochastically dominates $X$, we write $X\prec Y$ or $X=\bigO_{\prec}(Y)$.
\end{definition}
We have some notational conventions for the asymptotics of the limit $N \to \infty$. The symbols $\bigO(\cdot)$ and $o(\cdot)$ are used for the standard big-O and little-o notation. Let $x\equiv x_{N}$ and $y\equiv y_{N}$ be nonnegative $N$-dependent quantities. We write $x\lesssim y$ if there exist a constant $C>0$ such that $x\le Cy$ for all $N$. We use the notation $x\asymp y$ if $x\lesssim y$ and $y\lesssim x$. In this paper, we use the symbol $\ll$ in a less standard way. We write $x\ll y$ if there exist a constant $c>0$ satisfying $N^{c}x\lesssim y$.

For a complex number $w\in\C$, we denote its real part by $\Re{w}$ and its imaginary part by $\Im{w}$ throughout this paper.

\subsection{Main results}
In this paper, we consider the class of sparse random matrices introduced in \cite{EKYY12,EKYY13,HLY20}.
\begin{definition}[Sparse random matrices]\label{def: sparse RM}
	Let $H=(h_{ij})_{1\le i,j \le N}$ be an $N\times N$ real symmetric random matrix. The entries $(h_{ij})$ are independent up to the symmetry constraint $h_{ij}=h_{ji}$ and have the same moments. We assume that the entries $(h_{ij})$ have zero mean and $(1/N)$-variance, i.e.
	\begin{align*}
		\expec{h_{ij}} = 0, \quad \expec{h_{ij}^{2}} = \frac{1}{N}.
	\end{align*}
	For $p\ge 2$, the $p$-th cumulant $\cumul_{p}$ of $h_{ij}$ is given by
	\begin{align*}
		\cumul_{p} = \frac{(p-1)!\const_{p}}{Nq^{p-2}},
	\end{align*}
	where $q=N^{b}$ is the sparsity parameter with fixed (auxiliary parameter) $b$ satisfying $0<b<\frac{1}{2}$. We further assume
	\begin{align*}
		|\const_{p}|\le C_{p}
	\end{align*}
	for some constant $C_{p}>0$, and
	\begin{align*}
		\const_{4}\gtrsim 1.
	\end{align*}
\end{definition}

%Throughout the remainder of this paper, 
Throughout this paper, we denote by $H$ the sparse random matrix as in Definition \ref{def: sparse RM}. We define the Green function of $H$ by 
\begin{align*}
	G(z)\coloneqq(H-z)^{-1}, \qquad z\in\C_{+},
\end{align*}
and use the notation
\begin{align*}
	m(z)\coloneqq\frac{1}{N}\text{Tr}G(z), \qquad z\in\C_{+},
\end{align*}
for the normalized trace of the Green function. Let $\mu$ be the empirical eigenvalue distribution of $H$ given by
\begin{align}\label{eq: ESD}
	\mu \coloneqq \frac{1}{N}\sum_{i=1}^{N}\delta_{\lambda_{i}},
\end{align}
where we denote by $\lambda_{1}\ge\lambda_{2}\ge\cdots\ge\lambda_{N}$ the ordered eigenvalues of $H$. Note that the normalized trace $m(z)$ is the Stieltjes transform of the empirical eigenvalue distribution $\mu$. We introduce the random polynomial $P(z,w)$ defined through
\begin{align*}
	P(z,w) \coloneqq 1 + zw + Q(w), \qquad z,w\in\C_{+},
\end{align*}
where the random polynomial $Q(w)$ is given by
\begin{align*}
	Q(w) = \sum_{n=1}^{\ell} \mathcal{Z}_{n}w^{2n},
\end{align*}
for a large integer $\ell\ge 1$ to be chosen later, and each random coefficient $\mathcal{Z}_{n}$ is a polynomial in the variables $(h_{ij})_{1\le i,j\le N}$. There are two important assumptions for the random coefficients $(\mathcal{Z}_{n})_{1\le n\le \ell}$.
\begin{assumption}\label{assump: random correction term 1}
	For the sparsity parameter $q$ as in Definition \ref{def: sparse RM}, we have
	\begin{align}\label{eq: asymp 1}
	\mathcal{Z}_{1}-1 \prec \frac{1}{\sqrt{N}q},
	\end{align}
	and
	\begin{align}\label{eq: asymp 2}
	\mathcal{Z}_{n} \prec \frac{1}{q^{n-1}}, \quad 1\le n\le\ell.
	\end{align}
\end{assumption}
\begin{assumption}\label{assump: random correction term 2}
	For all $i,j\in\{1,2,\cdots,N\}$, the following holds:
	\begin{align}\label{eq: asymp 3}
	\frac{\partial \mathcal{Z}_{n}}{\partial h_{ij}} \prec \frac{1}{N}, \quad 1\le n\le\ell.
	\end{align}
\end{assumption}

If the above two assumptions hold, then we call the random polynomial $P(z,w)$ and the associated equation $P(z,\widetilde{m}(z))=0$ \emph{self-consistent polynomial} and \emph{self-consistent equation} respectively. Some important properties of a solution $\widetilde{m}$ of the self-consistent equation are described in the following lemma.

\begin{lemma}\label{lem: property of tilde m}
	Suppose Assumption \ref{assump: random correction term 1} holds. Then, there exists an algebraic function $\widetilde{m}: \C_{+}\to\C_{+}$, such that the following properties hold:
	\begin{enumerate}
		\item The function $\widetilde{m}$ is a solution of the self-consistent equation $P(z,\widetilde{m}(z))=0$.
		\item The function $\widetilde{m}$ is the Stieltjes transform of a random probability measure $\widetilde{\rho}$ and the support of $\widetilde{\rho}$ is $[-\widetilde{L},\widetilde{L}]$, where $\widetilde{L}$ and $-\widetilde{L}$ are called (spectral) edges.
		\item The probability measure $\widetilde{\rho}$ has strictly positive density on $(-\widetilde{L},\widetilde{L})$ and square root behavior at the edges.
		\item For any (large) $C>0$, there exists a polynomial $\asedge\equiv\asedge(\mathcal{Z}_{1}, \mathcal{Z}_{2}, \cdots, \mathcal{Z}_{\ell})$ such that 
		\begin{align*}%\label{eq: asymptotic edge}
			\widetilde{L} = \asedge + \bigO_{\prec}(N^{-C}).
		\end{align*}
		Thus, an approximate location of the edge is denoted by $\asedge$. 
		\item Consider $z\in\C_{+}$ and set $\kappa \coloneqq \text{dist}(\Re(z),\{-\widetilde{L},\widetilde{L}\})$. In a neighborhood of the support $[-\widetilde{L},\widetilde{L}]$, we have
		\begin{align*}
			\Im \widetilde{m}(z) \asymp \begin{cases}
				\sqrt{\kappa + \eta}, & \Re{z}\in[-\widetilde{L},\widetilde{L}], \\
				\frac{\eta}{\sqrt{\kappa + \eta}}, & \Re{z}\notin[-\widetilde{L},\widetilde{L}],
			\end{cases}
		\end{align*}
		\begin{align*}
			|\partial_{2}P(z,\widetilde{m}(z))|\asymp\sqrt{\kappa+\eta},
		\end{align*}
		\begin{align*}
			\partial_{2}^{2}P(z,\widetilde{m}(z)) = 2+\bigO(q^{-1}),
		\end{align*}
		where we write $\partial_{2}P(x,y)\equiv\partial_{y}P(x,y)/\partial y$.
	\end{enumerate}
\end{lemma}

The proof of Lemma \ref{lem: property of tilde m} is essentially same with that of \cite[Proposition 2.5, Proposition 2.6]{HLY20} and \cite[Lemma 4.1]{LS18} but we write it up in Appendix \ref{sec: properties of tilde m} for completeness. Our first main result is the local law up to the edge.  

\begin{theorem}[Local law near the edge]\label{thm: local law}
	Let $H$ be as in Definition \ref{def: sparse RM} with any small $b>0$. For $\ell\ge 1$ large enough, we can construct the random coefficients $(\mathcal{Z}_{n})_{1\le n\le \ell}$ such that the following statements hold:
	\begin{enumerate}
		\item Assumption \ref{assump: random correction term 1} and Assumption \ref{assump: random correction term 2} are satisfied.
		\item Let $\widetilde{m}$ and $\asedge$ be as in Lemma \ref{lem: property of tilde m}. We define $\tz$ by setting $\tz\coloneqq\asedge+E+i\eta$, and assume $|E|\le 1$ and $N^{-1}\ll\eta\le 1$. Then, we have
		\begin{multline}\label{eq: local law}
		\size{m(\tz)-\widetilde{m}(\tz)} \prec \paren{\frac{\boldsymbol{\phi}}{N\eta}}^{1/2} + (\sqrt{\kappa+\eta})^{1/4}\paren{\frac{\boldsymbol{\phi}}{N\eta}}^{3/8} + \frac{1}{N^{1/4}}\paren{\frac{\boldsymbol{\phi}}{N\eta}}^{1/8} + (\sqrt{\kappa+\eta})^{1/4}\paren{\frac{\boldsymbol{\phi}}{N^{2}\eta^{2}}}^{1/4} \\
		+ \frac{1}{N^{1/2}\eta^{1/4}} + \frac{1}{N\eta} + \frac{(\sqrt{\kappa+\eta})^{2/5}}{(N\eta)^{3/5}} + \frac{1}{N^{2/7}(N\eta)^{1/7}} + \frac{(\sqrt{\kappa+\eta})^{1/3}}{(N\eta)^{2/3}},
		\end{multline}
		where $\kappa \coloneqq \text{dist}(\Re(\tz),\{-\widetilde{L},\widetilde{L}\})$ and the parameter $\boldsymbol{\phi}$ is given by
		\begin{align}\label{eq: bold phi}
		\boldsymbol{\phi}\equiv\boldsymbol{\phi}(\tz)\coloneqq \begin{cases}
		\sqrt{\kappa + \eta}, & \Re{\tz}\in[-\widetilde{L},\widetilde{L}], \\
		\frac{\eta}{\sqrt{\kappa + \eta}}, & \Re{\tz}\notin[-\widetilde{L},\widetilde{L}].
		\end{cases}
		\end{align}
	\end{enumerate}
    We call these constructed random coefficients $\{\mathcal{Z}_{n}\}_{n=1}^{\ell}$ random correction terms.
\end{theorem}
\begin{remark}
	In contrast to the previous results (e.g.~\cite[Theorem 2.8]{EKYY13}, \cite[Theorem 2.4]{LS18}, and \cite[Theorem 2.1]{HLY20}), the local law estimate \eqref{eq: local law} does not contain any power of $q^{-1}$ since we implicitly assume that $\ell$ is large enough. The condition that $\ell > 2/(3b)$ is enough to guarantee an optimal local law in the vicinity of the edge. Compared with \cite[Equations (6.2), (6.4), (9.1)-(9.2)]{HK21}, our results cover a larger spectral domain.
\end{remark}
\begin{remark}
	The estimate \eqref{eq: local law} is perhaps daunting at first glance. The main thing to keep in mind is that the order of the bound is strong enough to be used for spectral analysis near the edge. For example, if $\kappa=\eta=N^{-2/3}$, then by \eqref{eq: local law} it is straightforward  to see that
	\begin{align*}
		\size{m(\tz)-\widetilde{m}(\tz)} \prec N^{-1/3}.
	\end{align*}
	Moreover, we can improve the local law estimate \eqref{eq: local law} outside the spectrum. See \eqref{eq: imporved local law outside the spectrum} for more details.
\end{remark}
\begin{remark}
	We can write a random correction term $\mathcal{Z}_{n}$ explicitly. For example,
	\begin{align*}
		\mathcal{Z}_{1} = \frac{1}{N}\sum_{i\neq j}h_{ij}^{2} = 1 + \mathcal{X} + \bigO_{\prec}\paren{\frac{1}{N}},
	\end{align*}
	(See \eqref{eq: mathcal X} for the definition of $\mathcal{X}$.) and
	\begin{align*}
		\mathcal{Z}_{2} &= \frac{1}{N}\sum_{i\neq j}h_{ij}^{4} - \frac{2}{N^{2}}\sum_{i\neq j\neq x\neq y} h_{ij}^{2}(h_{xy}^{2}-\expec{h_{xy}^{2}})\\
		&\qquad + \frac{1}{N}\sum_{i\neq j\neq y}h_{ij}^{2}(h_{iy}^{2}-\expec{h_{iy}^{2}})
		+ \frac{1}{N}\sum_{i\neq j \neq x}h_{ij}^{2}(h_{xj}^{2}-\expec{h_{xj}^{2}}),
	\end{align*}
	where we denote by $\sum_{i\neq j\neq x\neq y}$ the sum over all distinct indexes, similarly for $\sum_{i\neq j\neq y}$ and $\sum_{i\neq j \neq x}$.
	Note that $\mathcal{Z}_{1}$ contains the random correction term $\mathcal{X}$ introduced in \cite{HLY20} and the first term of $\mathcal{Z}_{2}$,
	\begin{align*}
		\frac{1}{N}\sum_{i\neq j}h_{ij}^{4},
	\end{align*}
	has a mean of the order $q^{-2}$ corresponding to the deterministic correction term introduced in \cite{LS18}. For $n\ge3$, we can still calculate $\mathcal{Z}_{n}$ but the explicit form is more complicated.
\end{remark}

As a result of Theorem \ref{thm: local law}, we can show that the extremal eigenvalues has fluctuations of order $N^{-2/3}$ with respect to the approximate location of the edge, $\asedge$. Recall that the ordered eigenvalues of $H$ is denoted by $$\lambda_{1}\ge\lambda_{2}\ge\cdots\ge\lambda_{N}.$$
\begin{theorem}[Recovery of edge rigidity]\label{thm: rigidity}
	Let $H$ be as in Definition \ref{def: sparse RM} with any small $b>0$. For $\ell\ge 1$ large enough, let $\{\mathcal{Z}_{n}\}_{n=1}^{\ell}$ be as in Theorem \ref{thm: local law}, and let $\asedge$ be as in Lemma \ref{lem: property of tilde m}. Fix an integer $k\ge 1$. We have for all $1\le i\le k$,
	\begin{align*}
		|\lambda_{i} - \asedge| \prec N^{-2/3}.
	\end{align*}
\end{theorem}
\begin{remark}
	This result extends \cite[Theorem 1.4]{HLY20} and reveals what exactly governs the higher order fluctuation of extremal eigenvalues of sparse random matrices when $1 \ll  q\ll N^{1/6}$.
\end{remark}
\begin{remark}
	As an application of Theorem \ref{thm: rigidity}, we consider the noise sensitivity problem of the top eigenvector, which is a unit eigenvector associated with the largest eigenvalue. Let $\vv$ be the top eigenvector of the random matrix $H$. We resample $k$ randomly chosen entries of the matrix $H$ and obtain another realization of the random matrix with top eigenvector $\vv^{[k]}$. According to \cite{BL21+}, if $q\gg N^{1/9}$ and $k\gg N^{5/3}$, the top eigenvectors, $\vv$ and $\vv^{[k]}$, are ``almost orthogonal'', i.e.
	\begin{align}\label{eq: noise sensitivity}
		\E\left\lvert \left\langle \vv,\vv^{[k]} \right\rangle \right\rvert = o(1).
	\end{align}
	Applying Theorem \ref{thm: rigidity}, it might be feasible to extend the noise sensitivity result \eqref{eq: noise sensitivity} to a larger regime, such as $q\gg 1$, by modifying the argument in \cite{BL21+}.
\end{remark}
\begin{remark}
	Very recently, Huang and Yau improved Theorem \ref{thm: rigidity} by including bulk eigenvalues \cite{HY22+}. Let $\widetilde{\rho}$ and $\widetilde{L}$ as in Lemma \ref{lem: property of tilde m}. According to \cite[Theorem 1.6]{HY22+}, for the classical eigenvalue locations $\gamma_{1}>\gamma_{2}>\cdots>\gamma_{N}$ defined by
	\begin{equation*}
	\frac{i-1/2}{N} = \int_{\gamma_{i}}^{\widetilde{L}} \widetilde{\rho}(x) dx,\quad 1\le i\le N,
	\end{equation*}
	we have
	\begin{equation*}
	|\lambda_{i} - \gamma_{i}| \prec N^{-2/3} \min(i,N-i+1)^{-1/3}, \quad 1\le i\le N.
	\end{equation*}
	In order to get \cite[Theorem 1.6]{HY22+}, they extended the main argument of this paper, which is the construction of higher order random correction terms $(\mathcal{Z}_{n})_{1\le n\le \ell}$.
\end{remark}
\begin{remark}
	By \cite[Theorem 1.2]{HK21}, in the regime $1 \ll q \ll N^{1/6}$, all nontrivial eigenvalues away from $0$ have Gaussian fluctuations, and the dominating term of the fluctuations is $\mathcal{X}=\mathcal{Z}_{1} - \expec{\mathcal{Z}_{1}}+\bigO_{\prec}(N^{-1})$. It might be possible to give another proof of \cite[Theorem 1.2]{HK21} by showing that
	\begin{equation*}
	\size{ \mathcal{Z}_{1} - \expec{\mathcal{Z}_{1}} } \gg \size{ \mathcal{Z}_{n} - \expec{\mathcal{Z}_{n}} },\quad n\ge 2,
	\end{equation*}
	and using some results in \cite{HY22+}.
\end{remark}

Since the typical order of Tracy-Widom fluctuation is $N^{-2/3}$, in the previous version of this paper, we conjectured that edge universality can be also recovered for $q \gg 1$. It was indeed confirmed by Huang and Yau very recently.
\begin{theorem}[Recovery of edge universality, {\cite[Theorem 1.7]{HY22+}}]
	Let $H$ be as in Definition \ref{def: sparse RM} with any small $b>0$. Let the random correction terms $\mathcal{Z}_{1},\mathcal{Z}_{2},\cdots,\mathcal{Z}_{\ell-1},\mathcal{Z}_{\ell}$ be as in Theorem \ref{thm: local law}, and let $\asedge$ be as in Lemma \ref{lem: property of tilde m}. Fix an integer $k\ge 1$. Let $F:\R^{k}\to\R$ be a bounded test function with bounded derivatives. Then, there exist a constant $c>0$ such that
	\begin{align*}
		\E\left[ F\big( N^{2/3}(\lambda_{1}-\asedge), \cdots , N^{2/3}(\lambda_{k}-\asedge) \big) \right] 
		= \E_{\text{GOE}}\left[ F\big( N^{2/3}(\mu_{1}-2), \cdots, N^{2/3}(\mu_{k}-2) \big) \right] + \bigO(N^{-c}),
	\end{align*}
	where the second expectation is with respect to a GOE matrix with eigenvalues $\mu_{1},\mu_{2},\cdots,\mu_{N}$.
	%	Then, the rescaled extremal eigenvalues,
	%	\begin{align*}
	%		N^{2/3}(\lambda_{1}-\asedge),\; \cdots ,\; N^{2/3}(\lambda_{k}-\asedge),
	%	\end{align*}
	%	converge to the Tracy-Widom distribution as $N\to\infty$.
\end{theorem}

\section{Outline of the proof} %Strategy and outline of proofs
The following proposition is the essential tool to prove the local law. 
\begin{proposition}[Recursive moment estimate]\label{prop: RME}
	Let $H$ be as in Definition \ref{def: sparse RM} with any small $b>0$. For an integer $\ell\ge 1$ large enough, we can construct the random coefficients $(\mathcal{Z}_{n})_{1\le n\le \ell}$ such that the following statements hold:
	\begin{enumerate}
		\item Assumption \ref{assump: random correction term 1} and Assumption \ref{assump: random correction term 2} are satisfied.
		\item Let $\asedge$ be as in Lemma \ref{lem: property of tilde m}. Consider $\tz=\asedge+E+i\eta$ such that $|E|\le 1$ and $N^{-1}\ll\eta\le1$.
		We define the control parameter $\Phi_{r}$ by
		\begin{multline}\label{eq: control parameter}
		\Phi_{r} \coloneqq \expec{\paren{\frac{\Im m(\tz)}{N\eta}+\frac{1}{N}}|P|^{2r-1}} \\
		+ \max_{1\le s\le 2r-1}\expec{\paren{\frac{\Im m(\tz)}{N\eta}+\sqrt{\frac{\Im m(\tz)}{N\eta}}+\frac{1}{N\eta}}\paren{\frac{|\partial_{2}P|\Im m(\tz)}{N\eta}+\paren{\frac{\Im m(\tz)}{N\eta}}^{2}+\frac{1}{N}}^{s}|P|^{2r-s-1} },
		\end{multline}
		where $\partial_{2}P(x,y)\equiv\partial_{y}P(x,y)/\partial y$. Then, we have
		\begin{align}\label{eq: rme}
		\E\left[ |P(\tz,m(\tz))|^{2r} \right]
		\prec \Phi_{r}.
		\end{align}
	\end{enumerate}
\end{proposition}
\begin{remark}
	In \cite[Proposition 2.6]{HY22+}, Huang and Yau improved Proposition \ref{prop: RME} by removing the term
	\begin{equation*}
	\max_{1\le s\le 2r-1}\expec{\sqrt{\frac{\Im m(\tz)}{N\eta}}\paren{\frac{|\partial_{2}P|\Im m(\tz)}{N\eta}+\frac{1}{N}}^{s}|P|^{2r-s-1} },
	\end{equation*}
	in the control parameter $\Phi_{r}$ so that they can get the rigidity estimates for the bulk eigenvalues in \cite[Theorem 1.6]{HY22+}.
\end{remark}

Let $\boldsymbol{\phi}$ be as in \eqref{eq: bold phi} and we set $\Lambda:=|m(\tz)-\widetilde{m}(\tz)|$. If \eqref{eq: rme} is given, using Young's inequality and Lemma \ref{lem: property of tilde m}, we can show
\begin{multline*}
	\expec{|P(\tz,m(\tz))|^{2r}}
	\prec \E\bigg[ \paren{\frac{\boldsymbol{\phi}}{N\eta}}^{2r} + \frac{\Lambda^{2r}}{(N\eta)^{2r}} + \frac{1}{N^{2r}} + (\sqrt{\kappa+\eta})^{r}\paren{\frac{\boldsymbol{\phi}}{N\eta}}^{3r/2} + (\sqrt{\kappa+\eta})^{r}\frac{\Lambda^{3r/2}}{(N\eta)^{3r/2}}
	\\ 
	+ \frac{1}{N^{r}}\paren{\frac{\boldsymbol{\phi}}{N\eta}}^{r/2} + \frac{\Lambda^{r/2}}{N^{r}(N\eta)^{r/2}}
	+ (\sqrt{\kappa+\eta})^{r}\paren{\frac{\boldsymbol{\phi}}{N^{2}\eta^{2}}}^{r} + (\sqrt{\kappa+\eta})^{r}\frac{\Lambda^{r}}{(N\eta)^{2r}} + \frac{1}{N^{2r}\eta^{r}} \bigg].
\end{multline*}
Taking Taylor expansion of the self-consistent polynomial $P(\tz,m(\tz))$ at $m=\widetilde{m}$, we can observe that
\begin{align*}
	\Lambda \prec \sqrt{|P(\tz,m(\tz))|}.
\end{align*}
Combining the above two estimates, we can show the local law, Theorem \ref{thm: local law}. The rigidity of extremal eigenvalues, Theorem \ref{thm: rigidity}, will follows from the standard argument using the improved local law outside the spectrum and Helffer-Sj\H{o}strand calculus. We remark that all omitted details can be found in Section \ref{sec: proof of local law and rigidity}.

The most technical part is to show the estimate \eqref{eq: rme}. In order to do that, we shall use the cumulant expansion also known as generalized Stein lemma. The statement is as follows.
\begin{lemma}[Cumulant expansion, {\cite[Lemma 2.1]{HK21}} and {\cite[Lemma 3.2]{LS18}}]\label{lem: cumulant expansion}
	Let $h$ be a centered random variable with finite moments of all order. We denote by $\cumul_{p}(h)$ the $p$-th cumulant of $h$. Let $f:\R\to\C$ be a smooth function. Then, for every positive integer $\ell$, we have
	\begin{align*}
		\expec{hf(h)} = \sum_{p=1}^{\ell} \frac{\cumul_{p+1}(h)}{p!}\expec{\frac{d^{p}f(h)}{dh^{p}}} + \mathcal{R}_{\ell+1},
	\end{align*}
	assuming that all expectations in the above equation exist, where $\mathcal{R}_{\ell+1}$ is a remainder term such that for any $t>0$,
	\begin{align*}
		\mathcal{R}_{\ell+1} = \bigO(1)\cdot\paren{\expec{\sup_{|x|\le|h|}\size{\frac{d^{\ell+1}f(h)}{dh^{\ell+1}}}^{2}}\cdot\expec{\size{h^{2\ell+4}\indic(|h|>t)}}} + \bigO(1)\cdot\expec{|h|^{\ell+2}}\cdot\sup_{|x|\le t}\size{\frac{d^{\ell+1}f(h)}{dh^{\ell+1}}(x)}.
	\end{align*}
\end{lemma}
%We write
%\begin{align*}
%	\E\left[ |P(\tilde{z},m(\tilde{z}))|^{2r} \right] = \expec{\big(1+\tz m + Q(m)\big)P^{r-1}\bar{P}^{r}}.
%\end{align*}
Since we have the resolvent identity
\begin{align*}
	\sum_{k}h_{ik}G_{kj} - z G_{ij} = \delta_{ij},
\end{align*}
it follows that
\begin{align*}
	1 + z m = \frac{1}{N}\sum_{i,j}h_{ij}G_{ij},
\end{align*}
and
\begin{align}\label{eq: for RME 1}
	\E\left[ |P(\tilde{z},m(\tilde{z}))|^{2r} \right]
	= \frac{1}{N}\sum_{i,j}\E\left[ h_{ij}G_{ij}P^{r-1}\bar{P}^{r} \right] + \E\left[ Q(m)P^{r-1}\bar{P}^{r} \right].
\end{align}
By the cumulant expansion, Lemma \ref{lem: cumulant expansion}, we have
\begin{align}\label{eq: for RME 2}
	\frac{1}{N}\sum_{i,j}\E\left[ h_{ij}G_{ij}P^{r-1}\bar{P}^{r} \right] = \frac{1}{N}\sum_{i,j}\sum_{p=1}^{\ell}\frac{\mathcal{C}_{p+1}}{Nq^{p-1}}\expec{\deri^{p}(G_{ij}P^{r-1}\bar{P}^{r})} + \bigO_{\prec}(q^{-\ell}),
\end{align}
where we use $\deri$ to denote the partial derivative with respect to $h_{ij}$, i.e.~$\deri\coloneqq\partial/\partial h_{ij}$.
%\tcb{We refer to \cite[Corollary 6.2]{LS18} for the justification of \eqref{eq: for RME 2}.}
Thus the problem boils down to showing
\begin{align*}
	\sum_{p=1}^{\ell}\frac{\mathcal{C}_{p+1}}{N^{2}q^{p-1}}\sum_{i,j}\expec{\deri^{p}(G_{ij}P^{r-1}\bar{P}^{r})} + \E\left[ Q(m)P^{r-1}\bar{P}^{r} \right] \prec \Phi_{r}.
\end{align*}
Therefore we should find the way to construct $Q(m)$ so that there is a nice cancellation with leading order terms of
\begin{align}\label{eq: main terms}
	\sum_{p=1}^{\ell}\frac{\mathcal{C}_{p+1}}{N^{2}q^{p-1}}\sum_{i,j}\expec{\deri^{p}(G_{ij}P^{r-1}\bar{P}^{r})}.
\end{align}
In Section \ref{sec: RME}, we shall describe how to construct the random correction terms $\{\mathcal{Z}_{n}\}_{n=1}^{\ell}$ in order to obtain the desired estimates of \eqref{eq: main terms}.

\section{Recursive moment estimate}\label{sec: RME}
This section is devoted to the proof of Proposition \ref{prop: RME}. For brevity, instead of \eqref{eq: main terms}, we consider
\begin{align}\label{eq: main terms alt}
\sum_{p=1}^{\ell}\frac{\mathcal{C}_{p+1}}{N^{2}q^{p-1}}\sum_{i,j}\expec{\deri^{p}(G_{ij}P^{2r-1})}.
\end{align}
The idea is roughly as follows:
\begin{enumerate}
	\item[Step 1.] Assume there exists a set of random coefficients $\{\mathcal{Z}_{n}\}_{n=1}^{\ell}$ satisfying Assumption \ref{assump: random correction term 1} and Assumption \ref{assump: random correction term 2}.
	\item[Step 2.] Observe that the main contribution of \eqref{eq: main terms alt} comes from the terms like
	\begin{equation}\label{eq: step 2}
	\expec{G_{ii}^{\frac{s+1}{2}}G_{jj}^{\frac{s+1}{2}}(\deri^{p-s}P^{2r-1})},
	\end{equation}
	where $s$ is a positive odd integer. (Proposition \ref{prop: moment estimates})
	\item[Step 3.] Replace all diagonal entries, $G_{ii}$ and $G_{jj}$, in \eqref{eq: step 2} with the normalized trace $m$ to get
	\begin{equation}\label{eq: step 3}
	\expec{m^{s+1}(\deri^{p-s}P^{2r-1})},
	\end{equation}
	and keep track of the remaining terms after the replacement. (Proposition \ref{prop: replacing all diagonals})
	\item[Step 4.] Repeat Step 2 and Step 3 for the remaining terms until the next remaining terms are negligible. Then, ignoring some negligible errors, we can write \eqref{eq: main terms alt} as a linear combination of the sums of the terms like \eqref{eq: step 3}. (Proposition \ref{prop: moment final})
	\item[Step 5.] Show that a sum of the terms like \eqref{eq: step 3} can be written as a sum of terms in the following form
	\begin{equation*}
	\expec{F(h_{ij})m^{s+1}P^{2r-1}},
	\end{equation*}
	where $F(h_{ij})$ is a polynomial in the variable $h_{ij}$. (Proposition \ref{prop: correction term mathing})\\
	Using this, we can construct $\{\mathcal{Z}_{n}\}_{n=1}^{\ell}$ so that the recursive moment estimate holds. Check Assumption \ref{assump: random correction term 1} and Assumption \ref{assump: random correction term 2} for the constructed set $\{\mathcal{Z}_{n}\}_{n=1}^{\ell}$.
\end{enumerate}

\subsection{Step 1.~Assuming the existence of the desired random coefficients}\label{subsec: step1}
We suppose there exists a set of random coefficients $\{\mathcal{Z}_{n}\}_{n=1}^{\ell}$ satisfying Assumption \ref{assump: random correction term 1} and Assumption \ref{assump: random correction term 2}. Then, we can define $\asedge=\asedge(\mathcal{Z}_{1},\cdots,\mathcal{Z}_{\ell})$ as in Lemma \ref{lem: property of tilde m}. In this subsection, we shall derive some useful estimates by supposing Assumption \ref{assump: random correction term 1} and Assumption \ref{assump: random correction term 2}. Recall that $\tz=\asedge+E+i\eta$ satisfying $|E|\le 1$ and $N^{-1}\ll\eta\le1$. Let $\deri$ be the partial derivative with respect to $h_{ij}$. Since $\asedge$ is a polynomial in the variables $\{\mathcal{Z}_{n}\}_{n=1}^{\ell}$, it follows from Assumption \ref{assump: random correction term 2} and the product rule that
\begin{align*}
\deri\tz \prec \frac{1}{N}. 
\end{align*}

\begin{remark}
	From Section \ref{subsec: step1} to Section \ref{subsec: step5}, we always assume that there exist the random coefficients $(\mathcal{Z}_{n})_{1\le n\le \ell}$ satisfies Assumption \ref{assump: random correction term 1} and Assumption \ref{assump: random correction term 2}.
\end{remark}

\begin{proposition}\label{prop: bounds for derivatives}
	Suppose Assumption \ref{assump: random correction term 1} and Assumption \ref{assump: random correction term 2} hold. For every $p\ge1$, we define $D_{ij}^{p}G \coloneqq \paren{\deri^{p}G}(\tz)$. The following estimates holds:
	\begin{align}\label{eq: derivative G_ij}
	\deri^{p}G_{ij}(\tz) = D_{ij}^{p}G_{ij} + \bigO_{\prec}\paren{\frac{\Im m(\tz)}{N\eta}},
	\end{align}
	\begin{align}\label{eq: derivative m}
	\deri^{p} m(\tz) = \bigO_{\prec}\paren{\frac{\Im m(\tz)}{N\eta}},
	\end{align}
	\begin{align}\label{eq: derivative P}
	\deri^{p} P(\tz, m(\tz)) = \begin{cases}
	\bigO_{\prec}\paren{\frac{|\partial_{2}P|\Im m(\tz)}{N\eta}+\frac{1}{N}} & p=1, \\
	\bigO_{\prec}\paren{\frac{|\partial_{2}P|\Im m(\tz)}{N\eta}+\paren{\frac{\Im m(\tz)}{N\eta}}^{2}+\frac{1}{N}} & p\ge2.
	\end{cases}
	\end{align}
	\begin{proof}
		We shall prove \eqref{eq: derivative G_ij} first. Applying the Ward identity
		\begin{align}\label{eq: ward id}
		\sum_{l}|G_{il}(\tz)|^{2} = \frac{\Im G_{ii}(\tz)}{\eta},
		\end{align}
		and \cite[Proposition A.1]{HLY20}, we get
		\begin{multline*}
		\deri G_{ij}(\tz) = D_{ij}G_{ij}(\tz) + (\partial_{z}G_{ij})(\tz) \deri(\tz)
		= D_{ij}G_{ij}(\tz) + \deri(\tz) \sum_{l}G_{il}(\tz)G_{lj}(\tz) \\
		= D_{ij}G_{ij}(\tz) + \bigO_{\prec}\paren{\frac{\Im m(\tz)}{N\eta}}.
		\end{multline*}
		Next, we will compute $\deri^{p} G_{ij}(\tz)$. Since $\deri$ (or $D_{ij}$) and $\partial_{z}$ commute, the Ward identity \eqref{eq: ward id} implies
		\begin{multline*}
		\deri D_{ij}G_{ij}(\tz) = D_{ij}^{2}G_{ij}(\tz) + \paren{\partial_{z}D_{ij}G_{ij}}(\tz) \deri(\tz) = D_{ij}^{2}G_{ij}(\tz) + \paren{D_{ij}\partial_{z}G_{ij}}(\tz) \deri(\tz) \\
		= D_{ij}^{2}G_{ij}(\tz) + \bigO_{\prec}\paren{\frac{\Im m(\tz)}{N\eta}}.
		\end{multline*}
		We also observe
		\begin{align*}
		\deri\paren{(\partial_{z}G_{ij})(\tz) \deri(\tz)} = \deri(\partial_{z}G_{ij})(\tz) \deri(\tz) + (\partial_{z}G_{ij})(\tz) \deri^{2}(\tz) = \bigO_{\prec}\paren{\frac{\Im m(\tz)}{N\eta}}.
		\end{align*}
		Repeating the above argument, we conclude \eqref{eq: derivative G_ij}.
		
		Now we take derivative for $m(\tz)$ to show \eqref{eq: derivative m}. Since $m = \frac{1}{N}\sum_{l}G_{ll}$ and $\deri G_{ll}(\tz) = D_{ij}G_{ll}(\tz) + \bigO_{\prec}\paren{\frac{\Im m(\tz)}{N\eta}}$ by the above argument, it follows from Ward identity \eqref{eq: ward id} that
		\begin{align*}
		\deri m(\tz) = %\frac{1}{N}\sum_{l}\deri G_{ll} = \frac{1}{N}\sum_{l}D_{ij} G_{ll} + \bigO_{\prec}\paren{\frac{\Im m(\tz)}{N\eta}}  = 
		\bigO_{\prec}\paren{\frac{\Im m(\tz)}{N\eta}}.
		\end{align*}
		We can prove the estimate \eqref{eq: derivative m} for general $p\ge1$ by the same reasoning.
		
		In order to get \eqref{eq: derivative P}, consider the polynomial
		\begin{align*}
		P\big(\tz,m(\tz)\big) = 1 + \tz m(\tz) + \sum_{n=1}^{\ell} \mathcal{Z}_{n}m^{2n}(\tz),
		\end{align*}
		and take derivative; then we have
		\begin{align*}
		\deri P\big(\tz,m(\tz)\big) = \paren{\partial_{2}P} \deri m(\tz) + \paren{ (\deri \tz)m(\tz) + \sum_{n=1}^{\ell} (\deri \mathcal{Z}_{n})m^{2n}(\tz) } = \bigO_{\prec}\paren{\frac{|\partial_{2}P|\Im m(\tz)}{N\eta}+\frac{1}{N}},
		\end{align*}
		where we denote by $\partial_{2}P(x,y)$ the partial derivative $\partial_{y}P(x,y)/\partial y$. Using $\deri\mathcal{Z}_{n} \prec \frac{1}{N}$ and the estimate \eqref{eq: derivative m}, the remaining part of \eqref{eq: derivative P} immediately follows. % We also can refer to \cite[Equation (2.48)-(2.49)]{HLY20}.
	\end{proof}
\end{proposition}

\subsection{Step 2.~Extracting the main contribution}
In this subsection, we shall calculate the main contribution of \eqref{eq: main terms alt}.

\begin{proposition}\label{prop: moment estimates}
	Suppose Assumption \ref{assump: random correction term 1} and Assumption \ref{assump: random correction term 2} hold. Let $\Phi_{r}$ be as in \eqref{eq: control parameter}. If $\ell$ is large enough, we have
%	\begin{align*}
%		\sum_{p=1}^{\ell}\frac{\mathcal{C}_{p+1}}{N^{2}q^{p-1}}&\sum_{i,j}\expec{\deri^{p}(G_{ij}P^{2r-1}) }\\
%		& = \sum_{p=1}^{\ell}\sum_{\substack{0\le s\le p \\ s\equiv1(\Mod2) }} \combi{p}{s}\frac{\mathcal{C}_{p+1}}{N^{2}q^{p-1}}
%		\sum_{i,j}\expec{(D_{ij}^{s}G_{ij})(\deri^{p-s}P^{2r-1})} + \bigO_{\prec}\paren{\Phi_{r}} \\
%		&= - \sum_{p=1}^{\ell}\sum_{\substack{0\le s\le p \\ s\equiv1(\Mod2) }} \combi{p}{s}\frac{s!\mathcal{C}_{p+1}}{N^{2}q^{p-1}}
%		\sum_{i,j}\expec{G_{ii}^{\frac{s+1}{2}}G_{jj}^{\frac{s+1}{2}}(\deri^{p-s}P^{2r-1})} + \bigO_{\prec}\paren{\Phi_{r}}.
%	\end{align*}
	\begin{multline*}
	\sum_{p=1}^{\ell}\frac{\mathcal{C}_{p+1}}{N^{2}q^{p-1}}\sum_{i,j}\expec{\deri^{p}(G_{ij}P^{2r-1}) } \\
	= - \sum_{p=1}^{\ell}\sum_{\substack{0\le s\le p \\ s\equiv1(\Mod2) }} \combi{p}{s}\frac{s!\mathcal{C}_{p+1}}{N^{2}q^{p-1}}
	\sum_{i\neq j}\expec{G_{ii}^{\frac{s+1}{2}}G_{jj}^{\frac{s+1}{2}}(\deri^{p-s}P^{2r-1})} + \bigO_{\prec}\paren{\Phi_{r}}.
	\end{multline*}
\end{proposition}

\begin{proof}
	Due to Proposition \ref{prop: bounds for derivatives} and \cite[Proposition A.1]{HLY20}, we have
	\begin{equation*}
	\frac{1}{N^{2}}\sum_{i}\expec{\partial_{ii}^{p}(G_{ii}P^{2r-1}) } = \bigO_{\prec}\paren{\Phi_{r}},
	\end{equation*}
	which implies
	\begin{equation*}
	\frac{1}{N^{2}}\sum_{i,j}\expec{\deri^{p}(G_{ij}P^{2r-1}) }
	= \frac{1}{N^{2}}\sum_{i\neq j}\expec{\deri^{p}(G_{ij}P^{2r-1}) } + \bigO_{\prec}\paren{\Phi_{r}}.
	\end{equation*}
	Then, the proof is split into two parts. One is to show that we have for odd integer $s$ with $0\le s\le p$,
	\begin{align*}
		\frac{1}{N^{2}}\sum_{i\neq j}\expec{(\deri^{s}G_{ij})(\deri^{p-s}P^{2r-1})} = -\frac{(s!)}{N^{2}}\sum_{i\neq j}\expec{G_{ii}^{\frac{s+1}{2}}G_{jj}^{\frac{s+1}{2}}(\deri^{p-s}P^{2r-1})} + \bigO_{\prec}\paren{\Phi_{r}}.
	\end{align*}
	The other is to prove for even integer $s$ satisfying $0\le s\le p$,
	\begin{align*}
		\frac{1}{N^{2}}\sum_{i\neq j}\expec{(\deri^{s}G_{ij})(\deri^{p-s}P^{2r-1})} = \bigO_{\prec}\paren{\Phi_{r}}.
	\end{align*}
	
	We first consider the case $p=1$:
	\begin{align*}
		\frac{1}{N^{2}}\sum_{i\neq j}\expec{\deri(G_{ij}P^{2r-1})} = \frac{1}{N^{2}}\sum_{i\neq j}\E\left[ (\deri G_{ij})P^{2r-1} \right] + \frac{(2r-1)}{N^{2}}\sum_{i\neq j}\E\left[ G_{ij}(\deri P) P^{2r-2} \right].
	\end{align*}
	Applying \eqref{eq: derivative G_ij}, we get
	\begin{align*}
		\frac{1}{N^{2}}\sum_{i\neq j}\E\left[ (\deri G_{ij})P^{2r-1} \right]
		&= \frac{1}{N^{2}}\sum_{i\neq j}\E\left[D_{ij}G_{ij}P^{2r-1}\right] + \bigO\left(\E\left[\frac{\Im m(\tz)}{N\eta}|P|^{2r-1}\right]\right) \\
		&= -\frac{1}{N^{2}}\sum_{i\neq j}\E\left[G_{ii}G_{jj}P^{2r-1}\right] + \bigO\left(\E\left[\frac{\Im m(\tz)}{N\eta}|P|^{2r-1}\right]\right)
	\end{align*}
	We observe that
	\begin{align*}
		\size{\sum_{i\neq j} G_{ij}(\deri P)} \le \paren{\sum_{i,j}|G_{ij}|^{2}}^{1/2}\paren{\sum_{i,j}|\deri P|^{2}}^{1/2} \prec \sqrt{N}\paren{\frac{\Im m(\tz)}{\eta}}^{1/2}\paren{\frac{\Im m(\tz)|\partial_{2}P|}{\eta}+1}.
	\end{align*}
	%where we use Cauchy-Schwarz inequality %\cite[Proposition 2.1.4]{PS11}.
	Then, we obtain
	\begin{align*}
		\frac{1}{N^{2}}\sum_{i\neq j}\E\left[ G_{ij}(\deri P)P^{2r-2} \right] = \bigO_{\prec}\paren{\expec{\paren{\frac{\Im m(\tz)}{N\eta}}^{1/2}\paren{\frac{\Im m(\tz)|\partial_{2}P|}{N\eta}+\frac{1}{N}}}|P|^{2r-2}}.
	\end{align*}
	We conclude
	\begin{multline*}
		\frac{1}{N^{2}}\sum_{i\neq j}\E\left[ \deri(G_{ij}P^{2r-1})) \right] = -\frac{1}{N^{2}}\sum_{i\neq j}\E\left[G_{ii}G_{jj}P^{2r-1}\right] \\
		+ \bigO_{\prec}\left(\expec{\frac{\Im m(\tz)}{N\eta}|P|^{2r-1}}+ \expec{\paren{\frac{\Im m(\tz)}{N\eta}}^{1/2}\paren{\frac{\Im m(\tz)|\partial_{2}P|}{N\eta}+\frac{1}{N}}|P|^{2r-2}} \right).
	\end{multline*}
	
	Next, we consider the case $p=2$:
	\begin{align*}
		\frac{1}{N^{2}}\sum_{i\neq j}\expec{\deri^{2}(G_{ij}P^{2r-1})}.
	\end{align*}
	The term $\deri^{2}(G_{ij}P^{2r-1})$ is split into $(\deri^{2}G_{ij})P^{2r-1}$, $\deri G_{ij}\deri P^{2r-1}$, and $G_{ij}\deri^{2}P^{2r-1}$. Using \eqref{eq: derivative G_ij}, it follows that
	\begin{align*}
		\frac{1}{N^{2}}\sum_{i\neq j}\expec{(\deri^{2}G_{ij})P^{2r-1}} = \frac{1}{N^{2}}\sum_{i\neq j}\expec{(D_{ij}^{2}G_{ij})P^{2r-1}} + \bigO_{\prec}\paren{\expec{\frac{\Im m(\tz)}{N\eta}|P|^{2r-1}}}.
	\end{align*}
	When we estimate
	\begin{align*}
		\frac{1}{N^{2}}\sum_{i\neq j}\expec{(D_{ij}^{2}G_{ij})P^{2r-1}},
	\end{align*}
	any term containing at least two off-diagonal Green function entries can be bounded by
	\begin{align*}
		\bigO_{\prec}\paren{\expec{\frac{\Im m(\tz)}{N\eta}|P|^{2r-1}}},
	\end{align*}
	using the Ward identity \eqref{eq: ward id}. Thus, the tricky term is
	\begin{align*}
		\frac{1}{N^{2}}\sum_{i\neq j}\expec{G_{ij}G_{ii}G_{jj}P^{2r-1}}.
	\end{align*}
    Let $H^{(i)}$ be the $N\times N$ matrix defined by
    \begin{align*}
    	\paren{H^{(i)}}_{kj} \coloneqq \indic(k\neq i)\indic(j\neq i) h_{kj}.
    \end{align*}
    We denote by $G^{(i)}$ the Green function of $H^{(i)}$. When $i\neq j$, we have the following resolvent identity: 
	\begin{align}\label{eq: resolvent entry identity 1} % Erdos-Yau book Lem 8.3
		G_{ij} = -G_{ii}\sum_{k}^{(i)}h_{ik}G_{kj}^{(i)},
	\end{align}
	where we use the notation 
	\begin{align}\label{eq: sum without an index}
		\sum_{k}^{(j)}\coloneqq\sum_{\substack{1\le k \le N \\ k\neq j}}.
	\end{align}
	By the above resolvent identity \eqref{eq: resolvent entry identity 1}, we get
	\begin{align*}
		\frac{1}{N^{2}}\sum_{i\neq j}\expec{ G_{ij}G_{ii}G_{jj}P^{2r-1}}
		= - \frac{1}{N^{2}}\sum_{i\neq j}\sum_{k}^{(i)}\expec{ h_{ik}G_{kj}^{(i)}G_{ii}^{2}G_{jj}P^{2r-1} } .
	\end{align*}
	We apply the cumulant expansion again and obtain
	\begin{align*}
		- \sum_{p'=1}^{\ell}\sum_{i\neq j}\sum_{k}^{(i)} \frac{\mathcal{C}_{p'+1}}{N^{3}q^{p'-1}}
		\expec{ G_{kj}^{(i)}\partial_{ik}^{p'}(G_{ii}^{2}G_{jj}P^{2r-1}) }.
	\end{align*}
	For $i\notin \{k,j\}$, the following identity holds:
	\begin{align}\label{eq: resolvent entry identity 2}
		G_{kj}^{(i)}=G_{kj}-\frac{G_{ki}G_{ij}}{G_{ii}}.
	\end{align}
	Using this identity \eqref{eq: resolvent entry identity 2}, we have for each $p'\ge 1$,
	\begin{multline}\label{eq: moment estimate p=2 one off-diagonal}
		\frac{1}{N^{3}q^{p'-1}}\sum_{i\neq j}\sum_{k}^{(i)}\expec{ G_{kj}^{(i)}\partial_{ik}^{p'}(G_{ii}^{2}G_{jj}P^{2r-1}) }
		\\
		= \frac{1}{N^{3}q^{p'-1}}\sum_{i\neq j}\sum_{k}^{(i)}\expec{ G_{kj}\partial_{ik}^{p'}(G_{ii}^{2}G_{jj}P^{2r-1}) } -\frac{1}{N^{3}q^{p'-1}}\sum_{i\neq j}\sum_{k}^{(i)}\expec{\frac{G_{ki}G_{ij}}{G_{ii}}\partial_{ik}^{p'}(G_{ii}^{2}G_{jj}P^{2r-1}) }.
	\end{multline}
	The second term on the right side of \eqref{eq: moment estimate p=2 one off-diagonal} contains at least two off-diagonal Green function entries. We deduce from the Ward identity that
	\begin{multline*}
		\frac{1}{N^{3}q^{p'-1}}\sum_{i\neq j}\sum_{k}^{(i)}\expec{\frac{G_{ki}G_{ij}}{G_{ii}}\partial_{ik}^{p'}(G_{ii}^{2}G_{jj}P^{2r-1}) } \\
		= \bigO_{\prec}\paren{ \max_{0\le s\le 2r-1}\expec{\frac{\Im m(\tz)}{N\eta}\paren{\frac{\Im m(\tz)|\partial_{2}P|}{N\eta}+\paren{\frac{\Im m(\tz)}{N\eta}}^{2}+\frac{1}{N}}^{s}|P|^{2r-s-1}} }.
	\end{multline*}
	For the first term on the right side of  \eqref{eq: moment estimate p=2 one off-diagonal}, if the derivative $\partial_{ik}$ hits $G_{jj}$ at least once, there are at least three off-diagonal entries in the summand so we can gain a factor of $\Im m(\tz)/N\eta$ due to the Ward identity. Similarly, when the derivative $\partial_{ik}$ hits $G_{ii}$ odd times, we can find at least two off-diagonal entries and hence an additional factor of $\Im m(\tz)/N\eta$ follows from the Ward identity. Thus, if the derivative $\partial_{ik}$ hits $G_{ii}$, it should hit $G_{ii}$ even times. If the derivative $\partial_{ik}$ hits $P$ at least once, the resulting term is bounded by
	\begin{align*}
		\bigO_{\prec}\paren{\max_{1\le s\le 2r-1}\expec{ \frac{1}{N\eta}\left(\frac{|\partial_{2}P|\Im m(\tz)}{N\eta}+\paren{\frac{\Im m(\tz)}{N\eta}}^{2}+\frac{1}{N}\right)^{s}|P|^{2r-s-1} } },
	\end{align*}
	since we have \eqref{eq: derivative P} and
	%\begin{align*}
	%	\size{ \sum_{k,j}G_{kk}G_{kj}G_{jj}\partial_{ik}P } &= \size{ \sum_{k}G_{kk}\partial_{ik}P\sum_{j}G_{kj}G_{jj} } \le \paren{\sum_{k}|G_{kk}\partial_{ik}P|^{2}}^{1/2}\paren{\sum_{k}\paren{\sum_{j}G_{kj}G_{jj}}^{2}}^{1/2} \nn\\
	%	&\lesssim \sqrt{N}\paren{\frac{\Im m(\tz)|\partial_{2}P|}{N\eta}}\lVert G\mathbf{v} \rVert \lesssim N\paren{\frac{\Im m(\tz)|\partial_{2}P|}{N\eta}}\matnorm{G} \le \frac{N}{\eta}\paren{\frac{\Im m(\tz)|\partial_{2}P|}{N\eta}},
	%\end{align*}
	\begin{align*}
		\size{ \sum_{k,j}G_{kj}G_{kk}^{l}G_{jj}\partial_{ik}^{l'}P } &= \size{ \sum_{k}G_{kk}^{l}\partial_{ik}^{l'}P\sum_{j}G_{kj}G_{jj} } \le \paren{\sum_{k}|G_{kk}^{l}\partial_{ik}^{l'}P|^{2}}^{1/2}\paren{\sum_{k}\size{\sum_{j}G_{kj}G_{jj}}^{2}}^{1/2} \\
		&\prec \sqrt{N}\paren{\frac{|\partial_{2}P|\Im m(\tz)}{N\eta}+\paren{\frac{\Im m(\tz)}{N\eta}}^{2}+\frac{1}{N}}\lVert G\mathbf{v} \rVert \\
		&\lesssim N\paren{\frac{|\partial_{2}P|\Im m(\tz)}{N\eta}+\paren{\frac{\Im m(\tz)}{N\eta}}^{2}+\frac{1}{N}}\matnorm{G}\\
		&\le \frac{N}{\eta}\paren{\frac{|\partial_{2}P|\Im m(\tz)}{N\eta}+\paren{\frac{\Im m(\tz)}{N\eta}}^{2}+\frac{1}{N}},
	\end{align*}
	where $\mathbf{v}=(G_{11},\cdots,G_{NN})$. 
	If the derivative $\partial_{ik}$ hits $G_{ii}$ even times, we apply the identities \eqref{eq: resolvent entry identity 1}-\eqref{eq: resolvent entry identity 2} and the cumulant expansion again. For example, in the case that $p'=2$, we have the term
	\begin{align*}
		\frac{1}{N^{3}q}\sum_{i\neq j}\sum_{k}^{(i)}\expec{ G_{kj}G_{ii}^{3}G_{kk}G_{jj}P^{2r-1} }.
	\end{align*}
	Note that we gain at least a factor of $q^{-1}$ if the derivative $\partial_{ik}$ hits $G_{ii}$ even times. Using the identities and the cumulant expansion, we have
	\begin{align*}
		- \sum_{p''=1}^{\ell}\sum_{i\neq j}\sum_{k}^{(i)}\sum_{l}^{(k)}\frac{\mathcal{C}_{p''+1}}{N^{4}q^{p''}}\expec{G^{(k)}_{lj}\partial_{kl}^{p''}(G_{ii}^{3}G_{kk}^{2}G_{jj}P^{2r-1})}.
	\end{align*}
	Following the above argument similarly, a tricky term comes from the case that the derivative $\partial_{kl}$ hits $G_{kk}$ even times so at least a factor of $q^{-1}$ is obtained. Repeating such process until we get these factors of $q^{-1}$ enough, we conclude
	\begin{multline}\label{eq: estimate tricky term}
		\frac{1}{N^{2}}\sum_{i\neq j}\expec{ G_{ij}G_{ii}G_{jj}P^{2r-1}}
		\\
		= \bigO_{\prec} \Bigg( \max_{0\le s\le 2r-1}\expec{\frac{\Im m(\tz)}{N\eta}\paren{\frac{\Im m(\tz)|\partial_{2}P|}{N\eta}+\paren{\frac{\Im m(\tz)}{N\eta}}^{2}+\frac{1}{N}}^{s}|P|^{2r-s-1}}  \\
		+  \max_{1\le s\le 2r-1}\expec{ \frac{1}{N\eta}\left(\frac{|\partial_{2}P|\Im m(\tz)}{N\eta}+\paren{\frac{\Im m(\tz)}{N\eta}}^{2}+\frac{1}{N}\right)^{s}|P|^{2r-s-1} } \Bigg).
	\end{multline}
	In addition, we can see that
	\begin{multline*}
		\combi{2}{1}\frac{\mathcal{C}_{3}}{N^{2}q}\sum_{i\neq j}\expec{  (\deri G_{ij})\deri P^{2r-1}  } \\
		=\combi{2}{1}\frac{\mathcal{C}_{3}}{N^{2}q}\sum_{i\neq j}\expec{  (D_{ij}G_{ij})\deri P^{2r-1}  } + \bigO_{\prec}\paren{\expec{\frac{\Im m(\tz)}{N\eta}\paren{\frac{\Im m(\tz)|\partial_{2}P|}{N\eta}+\frac{1}{N}}|P|^{r-2}}} \\
		=-\combi{2}{1}\frac{\mathcal{C}_{3}}{N^{2}q}\sum_{i\neq j}\expec{  G_{ii}G_{jj}\deri(P^{2r-1})  } + \bigO_{\prec}\paren{\expec{\frac{\Im m(\tz)}{N\eta}\paren{\frac{\Im m(\tz)|\partial_{2}P|}{N\eta}+\frac{1}{N}}|P|^{r-2}}}.
	\end{multline*}
	Since we have
	\begin{align*}
		\size{\sum_{i\neq j} G_{ij}(\deri^{2} P^{2r-1})} %\le \paren{\sum_{i,j}|G_{ij}|^{2}}^{1/2}\paren{\sum_{i,j}|\deri^{2} P^{2r-1}|^{2}}^{1/2}
		\prec \sqrt{N}\paren{\frac{\Im m(\tz)}{\eta}}^{1/2}N\max_{s=1,2}\left[\paren{\frac{|\partial_{2}P|\Im m(\tz)}{N\eta}+\paren{\frac{\Im m(\tz)}{N\eta}}^{2}+\frac{1}{N}}^{s}|P|^{2r-s-1}\right],
	\end{align*}
	it follows that
	\begin{align*}
		\frac{1}{N^{2}}\sum_{i\neq j}\E\left[ G_{ij}(\deri^{2} P^{2r-1}) \right] = \bigO_{\prec}\paren{\max_{s=1,2}\expec{\paren{\frac{\Im m(\tz)}{N\eta}}^{1/2}\paren{\frac{|\partial_{2}P|\Im m(\tz)}{N\eta}+\paren{\frac{\Im m(\tz)}{N\eta}}^{2}+\frac{1}{N}}^{s}|P|^{2r-s-1}}}.
	\end{align*}
	Thus, we deduce that
	\begin{align*}
		\frac{\mathcal{C}_{3}}{N^{2}q}\sum_{i\neq j}\E\left[ \partial_{ij}^{2}(G_{ij}P^{r-1}\bar{P}^{r}) \right] 
		= -\combi{2}{1}\frac{\mathcal{C}_{3}}{N^{2}q}\sum_{i\neq j}\expec{  G_{ii}G_{jj}\deri(P^{r-1}\bar{P}^{r})  } + \bigO_{\prec}\paren{\Phi_{r}}.
	\end{align*}
	
	Now we consider the case $p\ge3$:
	\begin{align*}
		\frac{\mathcal{C}_{p+1}}{N^{2}q^{p-1}}\sum_{i\neq j}\expec{\deri^{p}(G_{ij}P^{2r-1}) } & 
		= \sum_{s=0}^{p}\combi{p}{s}\frac{\mathcal{C}_{p+1}}{N^{2}q^{p-1}}\sum_{i\neq j}\expec{(\deri^{s}G_{ij})(\deri^{p-s}P^{2r-1})} \\
		&= \sum_{s=0}^{p}\combi{p}{s}\frac{\mathcal{C}_{p+1}}{N^{2}q^{p-1}}\sum_{i\neq j}\expec{(D_{ij}^{s}G_{ij})(\deri^{p-s}P^{2r-1})} + \bigO_{\prec}\paren{\Phi_{r}}.
	\end{align*}
	If $s$ is odd, we can observe that
	\begin{equation*}
	D_{ij}^{s}G_{ij} = - (s!)G_{ii}^{\frac{s+1}{2}}G_{jj}^{\frac{s+1}{2}} + (\text{the terms having at least two off-diagonal Green function entries}),
	\end{equation*}
	which implies due to the Ward identity \eqref{eq: ward id}
	\begin{equation*}
	\frac{\mathcal{C}_{p+1}}{N^{2}q^{p-1}}\sum_{i\neq j}\expec{(D_{ij}^{s}G_{ij})(\deri^{p-s}P^{2r-1})}
	= - \frac{(s!)\mathcal{C}_{p+1}}{N^{2}q^{p-1}}\sum_{i\neq j}\expec{G_{ii}^{\frac{s+1}{2}}G_{jj}^{\frac{s+1}{2}}(\deri^{p-s}P^{2r-1})} + \bigO_{\prec}\paren{\Phi_{r}}, \quad s\equiv1(\Mod2).
	\end{equation*}
	
	Next we consider the case that $s$ is even. We can see that $D_{ij}^{s}G_{ij}$ has the following form:
	\begin{equation*}
	D_{ij}^{s}G_{ij} = t_{s}G_{ij}G_{ii}^{s/2}G_{jj}^{s/2} + (\text{the terms having at least three off-diagonal Green function entries}),
	\end{equation*}
	where $t_{s}$ is a constant depending on $s$. Since any term having at least two off-diagonal Green function entries is eventually absorbed into the desired bound $\bigO_{\prec}\paren{\Phi_{r}}$ because of the Ward identity \eqref{eq: ward id}, we focus on the terms
	\begin{align*}
	    \frac{1}{N^{2}}\sum_{i\neq j}\expec{G_{ij}G_{ii}^{s/2}G_{jj}^{s/2}(\deri^{p-s}P^{2r-1})}, \quad s\equiv 0(\Mod2),
	\end{align*}
	where each term has the only one off-diagonal Green function entry $G_{ij}$. If $s=p$ (in this case, $p$ is even), we should estimate
	\begin{align*}
		\frac{1}{N^{2}}\sum_{i\neq j}\expec{G_{ij}G_{ii}^{p/2}G_{jj}^{p/2}P^{2r-1}}.
	\end{align*}
	We shall apply the argument from \eqref{eq: resolvent entry identity 1} to \eqref{eq: estimate tricky term} similarly to get
	\begin{align}\label{eq: moment estimate s even s=p}
		\frac{1}{N^{2}}\sum_{i\neq j}\expec{G_{ij}G_{ii}^{p/2}G_{jj}^{p/2}P^{2r-1}} = \bigO_{\prec}\paren{\Phi_{r}}.
	\end{align}
	Using the identity \eqref{eq: resolvent entry identity 1}, we have
	\begin{equation}\label{eq: ss1}
	\frac{1}{N^{2}}\sum_{i\neq j}\expec{G_{ij}G_{ii}^{p/2}G_{jj}^{p/2}P^{2r-1}}
	= - \frac{1}{N^{2}}\sum_{i\neq j}\sum_{k}^{(i)}\expec{h_{ik}G_{kj}^{(i)}G_{ii}^{1+p/2}G_{jj}^{p/2}P^{2r-1}}.
	\end{equation}
	Applying the cumulant expansion (Lemma \ref{lem: cumulant expansion}), we get
	\begin{equation*}
	- \sum_{p'=1}^{\ell}\sum_{i\neq j}\sum_{k}^{(i)}\frac{\mathcal{C}_{p'+1}}{N^{3}q^{p'-1}}\expec{ G_{kj}^{(i)} \partial_{ik}^{p'}(G_{ii}^{1+p/2}G_{jj}^{p/2}P^{2r-1}) }.
	\end{equation*}
	Due to the identity \eqref{eq: resolvent entry identity 2}, we have for each $p'\ge 1$,
	\begin{multline*}
	\frac{1}{N^{3}q^{p'-1}}\sum_{i\neq j}\sum_{k}^{(i)}\expec{ G_{kj}^{(i)} \partial_{ik}^{p'}(G_{ii}^{1+p/2}G_{jj}^{p/2}P^{2r-1}) } \\
	= \frac{1}{N^{3}q^{p'-1}}\sum_{i\neq j}\sum_{k}^{(i)}\expec{ G_{kj} \partial_{ik}^{p'}(G_{ii}^{1+p/2}G_{jj}^{p/2}P^{2r-1}) }
	- \frac{1}{N^{3}q^{p'-1}}\sum_{i\neq j}\sum_{k}^{(i)}\expec{ \frac{G_{ki}G_{ij}}{G_{ii}} \partial_{ik}^{p'}(G_{ii}^{1+p/2}G_{jj}^{p/2}P^{2r-1}) }.
	\end{multline*}
	The second term on the right-hand side has at least two off-diagonal entries so it can be absorbed into $\bigO_{\prec}\paren{\Phi_{r}}$ because of the Ward identity. Next we consider the first term on the right-hand side. If $\partial_{ik}$ hits $G_{jj}$ at least once, then we can find at least three off-diagonal entries so that every term of this case is absorbed into $\bigO_{\prec}\paren{\Phi_{r}}$ by the Ward identity. If $\partial_{ik}$ hits $P$ at least once, the resulting term is also absorbed into $\bigO_{\prec}\paren{\Phi_{r}}$ because we have for $l\ge0$ and $l'\ge 1$
	\begin{align*}
	\size{ \sum_{k,j}G_{kj}G_{kk}^{l}G_{jj}^{p/2}\partial_{ik}^{l'}P } &= \size{ \sum_{k}G_{kk}^{l}\partial_{ik}^{l'}P\sum_{j}G_{kj}G_{jj}^{p/2} } \le \paren{\sum_{k}|G_{kk}^{l}\partial_{ik}^{l'}P|^{2}}^{1/2}\paren{\sum_{k}\size{\sum_{j}G_{kj}G_{jj}^{p/2}}^{2}}^{1/2} \\
	&\lesssim N\paren{\frac{|\partial_{2}P|\Im m(\tz)}{N\eta}+\paren{\frac{\Im m(\tz)}{N\eta}}^{2}+\frac{1}{N}}\matnorm{G}\\
	&\le \frac{N}{\eta}\paren{\frac{|\partial_{2}P|\Im m(\tz)}{N\eta}+\paren{\frac{\Im m(\tz)}{N\eta}}^{2}+\frac{1}{N}},
	\end{align*}
	where the estimate \eqref{eq: derivative P} is used. The only remaining case is that $\partial_{ik}$ hits $G_{ii}$ only:
	\begin{equation*}
	\frac{1}{N^{3}q^{p'-1}}\sum_{i\neq j}\sum_{k}^{(i)}\expec{ G_{kj} \partial_{ik}^{p'}(G_{ii}^{1+p/2})G_{jj}^{p/2}P^{2r-1} }.
	\end{equation*}
	If $p'$ is odd, then there are at least two off-diagonal entries so we are done. A tricky term comes from the case that $p'$ is even. Since any term containing at least two off-diagonal entries can be absorbed into $\bigO_{\prec}\paren{\Phi_{r}}$, it is enough to only consider
	\begin{equation*}
	\frac{1}{N^{3}q^{p'-1}}\sum_{i\neq j}\sum_{k}^{(i)}\expec{ G_{kj} G_{ii}^{1+p/2+p'/2}G_{kk}^{p'/2} G_{jj}^{p/2}P^{2r-1} }.
	\end{equation*}
	Since $p'$ is even, we note that $p'\ge2$ and, as a result, we gain at least a factor of $q^{-1}$ in this case. Let us repeat the similar argument as the above. Using the identities \eqref{eq: resolvent entry identity 1} and \eqref{eq: resolvent entry identity 2} with the cumulant expansion, we have
	\begin{equation*}
	\frac{1}{q^{p'-1}}\sum_{p''=1}^{\ell}\sum_{i\neq j}\sum_{k}^{(i)}\sum_{l}^{(k)}\frac{\mathcal{C}_{p''+1}}{N^{4}q^{p''-1}}\expec{G_{lj}\partial_{kl}^{p''}(G_{kk}^{1+p'/2}G_{ii}^{1+p/2+p'/2}G_{jj}^{p/2}P^{2r-1})} + \bigO_{\prec}\paren{\Phi_{r}}, \quad p'\ge 2.
	\end{equation*}
	By the same reasoning, the tricky case is that the derivative $\partial_{kl}$ only hits $G_{kk}$ and $p''$ is even, which give us an additional factor of $q^{-1}$. Repeating similar procedure many times enough, the desired estimate \eqref{eq: moment estimate s even s=p} follows.
	
	If $s$ is even and $0\le s < p$, a tricky part is to bound
	\begin{align*}
		\frac{1}{N^{2}}\sum_{i\neq j}\expec{G_{ij}G_{ii}^{s/2}G_{jj}^{s/2}(\deri^{p-s}P^{2r-1})},
	\end{align*}
	where each term has the only one off-diagonal entry. Due to \cite[Proposition A.1]{HLY20} and the Cauchy-Schwarz inequality, we have for $0\le s<p$,
	\begin{align*}
	\size{\sum_{i\neq j} G_{ij}G_{ii}^{s/2}G_{jj}^{s/2}(\deri^{p-s}P^{2r-1})} \prec \paren{\sum_{i,j}|G_{ij}|^{2}}^{1/2}\paren{\sum_{i,j}|\deri^{p-s}P^{2r-1}|^{2}}^{1/2}.
	\end{align*}
	Combining the above with the Ward identity and the estimate \eqref{eq: derivative P}, we get
	\begin{align}\label{eq: ss2}
	\frac{1}{N^{2}}\sum_{i\neq j}\expec{G_{ij}G_{ii}^{s/2}G_{jj}^{s/2}(\deri^{p-s}P^{2r-1})} = \bigO_{\prec}\paren{\Phi_{r}}.
	\end{align}
	We complete the proof.
\end{proof}

\subsection{Step 3.~Replacing every diagonal entry with the normalized trace}

Changing the order of summations, we can rewrite the conclusion of Proposition \ref{prop: moment estimates} as follows:
\begin{multline}\label{eq: for RME 3}
\sum_{p=1}^{\ell}\frac{\mathcal{C}_{p+1}}{N^{2}q^{p-1}}\sum_{i,j}\expec{\deri^{p}(G_{ij}P^{2r-1}) } \\
= - \sum_{\substack{0\le s \le \ell \\ s \equiv1(\Mod2)}} \sum_{p=s}^{\ell} \combi{p}{s}\frac{(s!)\mathcal{C}_{p+1}}{N^{2}q^{p-1}}
\sum_{i\neq j}\expec{G_{ii}^{\frac{s+1}{2}}G_{jj}^{\frac{s+1}{2}}(\deri^{p-s}P^{2r-1})} + \bigO_{\prec}\paren{\Phi_{r}}.
\end{multline}
In this subsection, we shall replace all diagonal Green function entries in the right-hand side of \eqref{eq: for RME 3} with the normalized trace $m$. In order to do that, we need a general version of Proposition \ref{prop: moment estimates}.

\begin{lemma}[General version of Proposition \ref{prop: moment estimates}]\label{lem: moment estimate general}
	Let $d\ge 0$ be a non-negative integer. Let $t\ge 1$ be a positive integer. Fix an integer $k\ge 1$. Let $\{u_{j}\}_{j=1}^{k}$ be a finite sequence of non-negative integers. Let $\{v_{j}\}_{j=1}^{k}$ be a finite sequence of positive integers. Let $D(P)$ be a $l$-th order derivative of $P^{2r-1}$ with $l\ge0$, and we use the convention that $D(P)=P^{2r-1}$ for $l=0$. Then, we have
	\begin{multline}\label{eq: t1}
	\sum_{p=1}^{\ell}\frac{\mathcal{C}_{p+1}}{N^{2}q^{p-1}}\sum_{x,y}\expec{ \partial_{xy}^{p}\paren{ m^{d}G_{yx}G_{ii}^{t} \bigg(\prod_{j=1}^{k} G_{v_{j}v_{j}}^{u_{j}}\bigg) D(P) } } \\
	= - \sum_{\substack{0\le s \le \ell \\ s \equiv1(\Mod2)}}\sum_{p=s}^{\ell} \combi{p}{s}\frac{(s!)\mathcal{C}_{p+1}}{N^{2}q^{p-1}} \sum_{x\neq y} \expec{ m^{d} G_{xx}^{\frac{s+1}{2}}G_{yy}^{\frac{s+1}{2}}G_{ii}^{t} \bigg(\prod_{j=1}^{k} G_{v_{j}v_{j}}^{u_{j}}\bigg) \bigg(\partial_{xy}^{p-s}D(P)\bigg) } + \bigO_{\prec}\paren{ \Phi_{r} },
	\end{multline}
	and
	\begin{multline}\label{eq: t2}
	\sum_{p=1}^{\ell}\frac{\mathcal{C}_{p+1}}{Nq^{p-1}}\sum_{x}\expec{ \partial_{ix}^{p}\paren{ m^{d+1}G_{xi}G_{ii}^{t-1} \bigg(\prod_{j=1}^{k} G_{v_{j}v_{j}}^{u_{j}}\bigg) D(P) } }  \\
	= - \sum_{\substack{0\le s \le \ell \\ s \equiv1(\Mod2)}}c_{s}\sum_{p=s}^{\ell} \combi{p}{s}\frac{(s!)\mathcal{C}_{p+1}}{Nq^{p-1}} \sum_{x} \expec{ m^{d+1} G_{xx}^{\frac{s+1}{2}}G_{ii}^{\frac{s+1}{2}}G_{ii}^{t-1} \bigg(\prod_{j=1}^{k} G_{v_{j}v_{j}}^{u_{j}}\bigg) \bigg(\partial_{ix}^{p-s}D(P)\bigg) } + \bigO_{\prec}\paren{ \Phi_{r} },
	\end{multline}
	where $c_{s}$ is a bounded coefficient for each $0\le s \le \ell$. In fact,
	\begin{equation}\label{eq: t3}
	c_{s} = \sum_{0\le \tilde{s}_{1},\cdots,\tilde{s}_{t}\le s } \indic[ \tilde{s}_{1}+\cdots+\tilde{s}_{t} = s]\times\indic[\tilde{s}_{1} \equiv1(\Mod2)]\times\indic[\tilde{s}_{2},\cdots,\tilde{s}_{t}\equiv0(\Mod2)],
	\end{equation}
	so $c_{1}=1$ in particular.
\begin{proof}
	The proof of this lemma is in Appendix \ref{appen: lemma proof}.
\end{proof}
\end{lemma}
Using Lemma \ref{lem: moment estimate general}, we start replacing a diagonal entry.
\begin{claim}[Replacement of a single diagonal entry]\label{claim: replacing single diagonal}
	Let $s_{1}$ be a positive odd integer. We have
	\begin{multline}\label{eq: single replacement}
	\sum_{p_{1}=s_{1}}^{\ell}\combi{p_{1}}{s_{1}}\frac{(s_{1}!)\mathcal{C}_{p_{1}+1}}{N^{2}q^{p_{1}-1}}
	\sum_{i\neq j}\expec{G_{ii}^{\frac{s_{1}+1}{2}}G_{jj}^{\frac{s_{1}+1}{2}}(\deri^{p_{1}-s_{1}}P^{2r-1})} \\
	= \sum_{p_{1}=s_{1}}^{\ell}\combi{p_{1}}{s_{1}}\frac{(s_{1}!)\mathcal{C}_{p_{1}+1}}{N^{2}q^{p_{1}-1}}
	\sum_{i\neq j}\expec{mG_{ii}^{\frac{s_{1}-1}{2}}G_{jj}^{\frac{s_{1}+1}{2}}(\deri^{p_{1}-s_{1}}P^{2r-1})} \\
	- \sum_{p_{1}=s_{1}}^{\ell}\sum_{p_{2}=2}^{\ell} \combi{p_{1}}{s_{1}}\combi{p_{2}}{1}\frac{(s_{1}!)\mathcal{C}_{p_{1}+1}(1!)\mathcal{C}_{p_{2}+1}}{N^{4}q^{p_{1}+p_{2}-2}} \sum_{i\neq j}\sum_{x\neq y} \expec{ G_{xx}G_{yy}G_{ii}^{\frac{s_{1}+1}{2}}G_{jj}^{\frac{s_{1}+1}{2}}(\partial_{xy}^{p_{2}-1}\deri^{p_{1}-s_{1}}P^{2r-1}) } \\
	- \sum_{\substack{1 < s_{2} \le \ell \\ s_{2} \equiv1(\Mod2)}} \sum_{p_{1}=s_{1}}^{\ell}\sum_{p_{2}=s_{2}}^{\ell} \combi{p_{1}}{s_{1}}\combi{p_{2}}{s_{2}}\frac{(s_{1}!)\mathcal{C}_{p_{1}+1}(s_{2}!)\mathcal{C}_{p_{2}+1}}{N^{4}q^{p_{1}+p_{2}-2}} \sum_{i\neq j}\sum_{x\neq y} \expec{ G_{xx}^{\frac{s_{2}+1}{2}}G_{yy}^{\frac{s_{2}+1}{2}}G_{ii}^{\frac{s_{1}+1}{2}}G_{jj}^{\frac{s_{1}+1}{2}}(\partial_{xy}^{p_{2}-s_{2}}\deri^{p_{1}-s_{1}}P^{2r-1}) } \\
	+ \sum_{p_{1}=s_{1}}^{\ell}\sum_{p_{2}=2}^{\ell} \combi{p_{1}}{s_{1}}\combi{p_{2}}{1}\frac{(s_{1}!)\mathcal{C}_{p_{1}+1}(1!)\mathcal{C}_{p_{2}+1}}{N^{3}q^{p_{1}+p_{2}-2}} \sum_{i\neq j}\sum_{x} \expec{ G_{xx}G_{ii}G_{ii}^{\frac{s_{1}-1}{2}}
		G_{jj}^{\frac{s_{1}+1}{2}}m(\partial_{ix}^{p_{2}-s_{2}}\deri^{p_{1}-s_{1}}P^{2r-1}) } \\
	+ \sum_{\substack{1 < s_{2} \le \ell \\ s_{2} \equiv1(\Mod2)}} c_{s_{2}} \sum_{p_{1}=s_{1}}^{\ell}\sum_{p_{2}=s_{2}}^{\ell} \combi{p_{1}}{s_{1}}\combi{p_{2}}{s_{2}}\frac{(s_{1}!)\mathcal{C}_{p_{1}+1}(s_{2}!)\mathcal{C}_{p_{2}+1}}{N^{3}q^{p_{1}+p_{2}-2}} \sum_{i\neq j}\sum_{x} \expec{ G_{xx}^{\frac{s_{2}+1}{2}}G_{ii}^{\frac{s_{2}+1}{2}}G_{ii}^{\frac{s_{1}-1}{2}}
		G_{jj}^{\frac{s_{1}+1}{2}}m(\partial_{ix}^{p_{2}-s_{2}}\deri^{p_{1}-s_{1}}P^{2r-1}) }
	\\ + \bigO_{\prec}\paren{ \Phi_{r} },
	\end{multline}
	where $c_{s_{2}}$ is a bounded coefficient for each $0\le s_{2} \le \ell$.
\begin{proof}
	According to \cite[Lemma 10.1]{HK21}, we have
	\begin{align*}
	G_{ij} = \delta_{ij}m + \frac{G_{ij}}{N}\sum_{x,y}h_{xy}G_{yx} - m\sum_{x}h_{ix}G_{xj}.
	\end{align*}
	If $i=j$, it follows that
	\begin{align}\label{eq: resolvent entry identity 3}
	G_{ii} = m + \frac{G_{ii}}{N}\sum_{x,y}h_{xy}G_{yx} - m\sum_{x}h_{ix}G_{xi}.
	\end{align}
	Using the identity \eqref{eq: resolvent entry identity 3}, we get
	\begin{multline}\label{eq: case s_1 > 1, auxiliary eq 1}
	\frac{1}{N^{2}}\sum_{i\neq j}\expec{G_{ii}^{\frac{s_{1}+1}{2}}G_{jj}^{\frac{s_{1}+1}{2}}(\deri^{p_{1}-s_{1}}P^{2r-1})} \\
	= \frac{1}{N^{2}}\sum_{i\neq j}\expec{mG_{ii}^{\frac{s_{1}-1}{2}}G_{jj}^{\frac{s_{1}+1}{2}}(\deri^{p_{1}-s_{1}}P^{2r-1})}
	+ \frac{1}{N^{3}}\sum_{i\neq j}\sum_{x,y}\expec{h_{xy}G_{yx}G_{ii}^{\frac{s_{1}+1}{2}}G_{jj}^{\frac{s_{1}+1}{2}}(\deri^{p_{1}-s_{1}}P^{2r-1})} \\
	- \frac{1}{N^{2}}\sum_{i\neq j}\sum_{x}\expec{h_{ix}G_{xi}G_{ii}^{\frac{s_{1}-1}{2}}G_{jj}^{\frac{s_{1}+1}{2}}m(\deri^{p_{1}-s_{1}}P^{2r-1})}.
	\end{multline}
	Applying the cumulant expansion to the last two terms on the right-hand side of \eqref{eq: case s_1 > 1, auxiliary eq 1}, we obtain
	\begin{multline}\label{eq: case s_1 > 1, auxiliary eq 2}
	\frac{1}{N^{3}}\sum_{i\neq j}\sum_{x,y}\expec{h_{xy}G_{yx}G_{ii}^{\frac{s_{1}+1}{2}}G_{jj}^{\frac{s_{1}+1}{2}}(\deri^{p_{1}-s_{1}}P^{2r-1})} 
	- \frac{1}{N^{2}}\sum_{i\neq j}\sum_{x}\expec{h_{ix}G_{xi}G_{ii}^{\frac{s_{1}-1}{2}}G_{jj}^{\frac{s_{1}+1}{2}}m(\deri^{p_{1}-s_{1}}P^{2r-1})} \\
	= 
	\frac{1}{N^{3}}\sum_{i\neq j}\sum_{x,y}\sum_{p_{2}=1}^{\ell}\frac{\mathcal{C}_{p_{2}+1}}{Nq^{p_{2}-1}}\expec{\partial_{xy}^{p_{2}}\paren{G_{yx}G_{ii}^{\frac{s_{1}+1}{2}}G_{jj}^{\frac{s_{1}+1}{2}}(\deri^{p_{1}-s_{1}}P^{2r-1})}}
	\\
	-\frac{1}{N^{2}}\sum_{i\neq j}\sum_{x}\sum_{p_{2}=1}^{\ell}\frac{\mathcal{C}_{p_{2}+1}}{Nq^{p_{2}-1}}\expec{\partial_{ix}^{p_{2}}\paren{G_{xi}G_{ii}^{\frac{s_{1}-1}{2}}G_{jj}^{\frac{s_{1}+1}{2}}m(\deri^{p_{1}-s_{1}}P^{2r-1})}} + \bigO_{\prec}\paren{ \Phi_{r} },
	\end{multline}
	Using Lemma \ref{lem: moment estimate general} to the right-hand side of \eqref{eq: case s_1 > 1, auxiliary eq 2} (with setting $d=0$, $t=\frac{s_{1}+1}{2}$, $k=1$, $v_{1}=j$, $u_{1}=\frac{s_{1}+1}{2}$, $D(P)=\deri^{p_{1}-s_{1}}P^{2r-1}$), we have
	\begin{multline}\label{eq: a}
	\frac{1}{N^{3}}\sum_{i\neq j}\sum_{x,y}\sum_{p_{2}=1}^{\ell}\frac{\mathcal{C}_{p_{2}+1}}{Nq^{p_{2}-1}}\expec{\partial_{xy}^{p_{2}}\paren{G_{yx}G_{ii}^{\frac{s_{1}+1}{2}}G_{jj}^{\frac{s_{1}+1}{2}}(\deri^{p_{1}-s_{1}}P^{2r-1})}} \\
	= - \sum_{\substack{0\le s_{2} \le \ell \\ s_{2} \equiv1(\Mod2)}} \sum_{p_{2}=s_{2}}^{\ell} \combi{p_{2}}{s_{2}}\frac{(s_{2}!)\mathcal{C}_{p_{2}+1}}{N^{4}q^{p_{2}-1}} \sum_{i\neq j}\sum_{x\neq y} \expec{ G_{xx}^{\frac{s_{2}+1}{2}}G_{yy}^{\frac{s_{2}+1}{2}}G_{ii}^{\frac{s_{1}+1}{2}}G_{jj}^{\frac{s_{1}+1}{2}}(\partial_{xy}^{p_{2}-s_{2}}\deri^{p_{1}-s_{1}}P^{2r-1}) } + \bigO_{\prec}\paren{ \Phi_{r} },
	\end{multline}
	and
	\begin{multline}\label{eq: b}
	\frac{1}{N^{2}}\sum_{i\neq j}\sum_{x}\sum_{p_{2}=1}^{\ell}\frac{\mathcal{C}_{p_{2}+1}}{Nq^{p_{2}-1}}\expec{\partial_{ix}^{p_{2}}\paren{G_{xi}G_{ii}^{\frac{s_{1}-1}{2}}G_{jj}^{\frac{s_{1}+1}{2}}m(\deri^{p_{1}-s_{1}}P^{2r-1})}}  \\
	= - \sum_{\substack{0\le s_{2} \le \ell \\ s_{2} \equiv1(\Mod2)}} c_{s_{2}} \sum_{p_{2}=s_{2}}^{\ell} \combi{p_{2}}{s_{2}}\frac{(s_{2}!)\mathcal{C}_{p_{2}+1}}{N^{3}q^{p_{2}-1}} \sum_{i\neq j}\sum_{x} \expec{ G_{xx}^{\frac{s_{2}+1}{2}}G_{ii}^{\frac{s_{2}+1}{2}}G_{ii}^{\frac{s_{1}-1}{2}}
		G_{jj}^{\frac{s_{1}+1}{2}}m(\partial_{ix}^{p_{2}-s_{2}}\deri^{p_{1}-s_{1}}P^{2r-1}) } + \bigO_{\prec}\paren{ \Phi_{r} }
	\end{multline}
	Combining \eqref{eq: case s_1 > 1, auxiliary eq 2}, \eqref{eq: a} and \eqref{eq: b}, it follows that
	\begin{multline}\label{eq: c}
	\frac{1}{N^{3}}\sum_{i\neq j}\sum_{x,y}\sum_{p_{2}=1}^{\ell}\frac{\mathcal{C}_{p_{2}+1}}{Nq^{p_{2}-1}}\expec{\partial_{xy}^{p_{2}}\paren{G_{yx}G_{ii}^{\frac{s_{1}+1}{2}}G_{jj}^{\frac{s_{1}+1}{2}}(\deri^{p_{1}-s_{1}}P^{2r-1})}}
	\\
	-\frac{1}{N^{2}}\sum_{i\neq j}\sum_{x}\sum_{p_{2}=1}^{\ell}\frac{\mathcal{C}_{p_{2}+1}}{Nq^{p_{2}-1}}\expec{\partial_{ix}^{p_{2}}\paren{G_{xi}G_{ii}^{\frac{s_{1}-1}{2}}G_{jj}^{\frac{s_{1}+1}{2}}m(\deri^{p_{1}-s_{1}}P^{2r-1})}} \\
	= - \sum_{\substack{0\le s_{2} \le \ell \\ s_{2} \equiv1(\Mod2)}} \sum_{p_{2}=s_{2}}^{\ell} \combi{p_{2}}{s_{2}}\frac{(s_{2}!)\mathcal{C}_{p_{2}+1}}{N^{4}q^{p_{2}-1}} \sum_{i\neq j}\sum_{x\neq y} \expec{ G_{xx}^{\frac{s_{2}+1}{2}}G_{yy}^{\frac{s_{2}+1}{2}}G_{ii}^{\frac{s_{1}+1}{2}}G_{jj}^{\frac{s_{1}+1}{2}}(\partial_{xy}^{p_{2}-s_{2}}\deri^{p_{1}-s_{1}}P^{2r-1}) } \\
	+ \sum_{\substack{0\le s_{2} \le \ell \\ s_{2} \equiv1(\Mod2)}} c_{s_{2}} \sum_{p_{2}=s_{2}}^{\ell} \combi{p_{2}}{s_{2}}\frac{(s_{2}!)\mathcal{C}_{p_{2}+1}}{N^{3}q^{p_{2}-1}} \sum_{i\neq j}\sum_{x} \expec{ G_{xx}^{\frac{s_{2}+1}{2}}G_{ii}^{\frac{s_{2}+1}{2}}G_{ii}^{\frac{s_{1}-1}{2}}
		G_{jj}^{\frac{s_{1}+1}{2}}m(\partial_{ix}^{p_{2}-s_{2}}\deri^{p_{1}-s_{1}}P^{2r-1}) } + \bigO_{\prec}\paren{ \Phi_{r} }.
	\end{multline}
    When $p_{2}=s_{2}=1$, there is a cancellation in the right-hand sides of \eqref{eq: c} as follows: using the fact that $c_{s_{2}}=1$ for $s_{2}=1$, we have
    \begin{multline*}
    \frac{1}{N^{4}} \sum_{i\neq j}\sum_{x\neq y} \expec{ G_{xx}G_{yy}G_{ii}^{\frac{s_{1}+1}{2}}G_{jj}^{\frac{s_{1}+1}{2}}(\deri^{p_{1}-s_{1}}P^{2r-1}) }
    - \frac{1}{N^{3}} \sum_{i\neq j}\sum_{x} \expec{ G_{xx}G_{ii}^{\frac{s_{1}+1}{2}}
    	G_{jj}^{\frac{s_{1}+1}{2}}m(\deri^{p_{1}-s_{1}}P^{2r-1}) } \\
    = \frac{1}{N^{4}} \sum_{i\neq j}\sum_{x,y} \expec{ G_{xx}G_{yy}G_{ii}^{\frac{s_{1}+1}{2}}G_{jj}^{\frac{s_{1}+1}{2}}(\deri^{p_{1}-s_{1}}P^{2r-1}) }
    - \frac{1}{N^{2}} \sum_{i\neq j} \expec{ G_{ii}^{\frac{s_{1}+1}{2}}
    	G_{jj}^{\frac{s_{1}+1}{2}}m^{2}(\deri^{p_{1}-s_{1}}P^{2r-1}) } + \bigO_{\prec}\paren{ \Phi_{r} } \\
    = \frac{1}{N^{2}} \sum_{i\neq j} \expec{ m^{2}G_{ii}^{\frac{s_{1}+1}{2}}
    	G_{jj}^{\frac{s_{1}+1}{2}}(\deri^{p_{1}-s_{1}}P^{2r-1}) }
    - \frac{1}{N^{2}} \sum_{i\neq j} \expec{ G_{ii}^{\frac{s_{1}+1}{2}}
    	G_{jj}^{\frac{s_{1}+1}{2}}m^{2}(\deri^{p_{1}-s_{1}}P^{2r-1}) } + \bigO_{\prec}\paren{ \Phi_{r} } \\
    = \bigO_{\prec}\paren{ \Phi_{r} },
    \end{multline*}
    where we use Proposition \ref{prop: bounds for derivatives} and \cite[Proposition A.1]{HLY20} to switch over from $\sum_{x\neq y}$ to $\sum_{x,y}$. Then, we get from \eqref{eq: c}
    \begin{multline}\label{eq: d}
    \frac{1}{N^{3}}\sum_{i\neq j}\sum_{x,y}\sum_{p_{2}=1}^{\ell}\frac{\mathcal{C}_{p_{2}+1}}{Nq^{p_{2}-1}}\expec{\partial_{xy}^{p_{2}}\paren{G_{yx}G_{ii}^{\frac{s_{1}+1}{2}}G_{jj}^{\frac{s_{1}+1}{2}}(\deri^{p_{1}-s_{1}}P^{2r-1})}}
    \\
    -\frac{1}{N^{2}}\sum_{i\neq j}\sum_{x}\sum_{p_{2}=1}^{\ell}\frac{\mathcal{C}_{p_{2}+1}}{Nq^{p_{2}-1}}\expec{\partial_{ix}^{p_{2}}\paren{G_{xi}G_{ii}^{\frac{s_{1}-1}{2}}G_{jj}^{\frac{s_{1}+1}{2}}m(\deri^{p_{1}-s_{1}}P^{2r-1})}} \\
    =  - \sum_{p_{2}=2}^{\ell} \combi{p_{2}}{1}\frac{(1!)\mathcal{C}_{p_{2}+1}}{N^{4}q^{p_{2}-1}} \sum_{i\neq j}\sum_{x\neq y} \expec{ G_{xx}G_{yy}G_{ii}^{\frac{s_{1}+1}{2}}G_{jj}^{\frac{s_{1}+1}{2}}(\partial_{xy}^{p_{2}-1}\deri^{p_{1}-s_{1}}P^{2r-1}) } \\
    - \sum_{\substack{1 < s_{2} \le \ell \\ s_{2} \equiv1(\Mod2)}} \sum_{p_{2}=s_{2}}^{\ell} \combi{p_{2}}{s_{2}}\frac{(s_{2}!)\mathcal{C}_{p_{2}+1}}{N^{4}q^{p_{2}-1}} \sum_{i\neq j}\sum_{x\neq y} \expec{ G_{xx}^{\frac{s_{2}+1}{2}}G_{yy}^{\frac{s_{2}+1}{2}}G_{ii}^{\frac{s_{1}+1}{2}}G_{jj}^{\frac{s_{1}+1}{2}}(\partial_{xy}^{p_{2}-s_{2}}\deri^{p_{1}-s_{1}}P^{2r-1}) } \\
    + \sum_{p_{2}=2}^{\ell} \combi{p_{2}}{1}\frac{(1!)\mathcal{C}_{p_{2}+1}}{N^{3}q^{p_{2}-1}} \sum_{i\neq j}\sum_{x} \expec{ G_{xx}G_{ii}G_{ii}^{\frac{s_{1}-1}{2}}
    	G_{jj}^{\frac{s_{1}+1}{2}}m(\partial_{ix}^{p_{2}-s_{2}}\deri^{p_{1}-s_{1}}P^{2r-1}) } \\
    + \sum_{\substack{1< s_{2} \le \ell \\ s_{2} \equiv1(\Mod2)}} c_{s_{2}} \sum_{p_{2}=s_{2}}^{\ell} \combi{p_{2}}{s_{2}}\frac{(s_{2}!)\mathcal{C}_{p_{2}+1}}{N^{3}q^{p_{2}-1}} \sum_{i\neq j}\sum_{x} \expec{ G_{xx}^{\frac{s_{2}+1}{2}}G_{ii}^{\frac{s_{2}+1}{2}}G_{ii}^{\frac{s_{1}-1}{2}}
    	G_{jj}^{\frac{s_{1}+1}{2}}m(\partial_{ix}^{p_{2}-s_{2}}\deri^{p_{1}-s_{1}}P^{2r-1}) } + \bigO_{\prec}\paren{ \Phi_{r} }.
    \end{multline}
    The claim follows from \eqref{eq: case s_1 > 1, auxiliary eq 1} and \eqref{eq: d}.
\end{proof}
\end{claim}

The right-hand side of \eqref{eq: single replacement} is too complicated so we introduce some new notations for brevity.
\begin{description}
	\item[Notation 1] For a positive integer $k$, we consider a set of index pairs
	\begin{equation*}
	\mathcal{I}_{k}=\{(i_{1},j_{1}),\cdots,(i_{k},j_{k})\}.
	\end{equation*}
	All indexes in the set $\mathcal{I}_{k}$ will be used to denote formal indexes in a summation. We allow the case that two formally distinct indexes such as $i_{l}$ and $i_{l'}$ (or $j_{l}$ and $j_{l'}$) are \emph{equivalent}, i.e.~$i_{l}\equiv i_{l'}$ (or $j_{l}\equiv j_{l'}$) where $l\neq l'$. If two formally distinct indexes are not equivalent, these are said to be \emph{non-equivalent}. If $i_{l}$ and $i_{l'}$ are non-equivalent, we denote this by $i_{l}\not\equiv i_{l'}$. Similarly, when $j_{l}$ and $j_{l'}$ are non-equivalent, we write $j_{l}\not\equiv j_{l'}$.
	
	We denote the maximal number of (pairwise) non-equivalent indexes in $\mathcal{I}_{k}$ by $\mathfrak{k}=\mathfrak{k}(\mathcal{I}_{k})$, and write
	\begin{align*}
	\{i_{1},j_{1},\cdots,i_{k},j_{k}\}=\{\mathfrak{i}_{1},\cdots,\mathfrak{i}_{\mathfrak{k}}\},
	\end{align*}
	where each $\mathfrak{i}_{l}$ $(1\le l\le\mathfrak{k})$ is an element in a set of equivalent indexes, and the set $\{\mathfrak{i}_{1},\cdots,\mathfrak{i}_{\mathfrak{k}}\}$ consists of non-equivalent indexes. For example, we consider the case that $k=2$. We have $\mathcal{I}_{2}=\{(i_{1},j_{1}),(i_{2},j_{2})\}$. If $i_{1},j_{1},j_{2}$ are non-equivalent but $i_{2}\equiv i_{1}$, then we have $\mathfrak{k}=3$ and also write $\{\mathfrak{i}_{1},\mathfrak{i}_{2},\mathfrak{i}_{3}\}=\{i_{1},j_{1},j_{2}\}$.
	
	We denote the summation over all distinct non-equivalent indexes $\{\mathfrak{i}_{1},\cdots,\mathfrak{i}_{\mathfrak{k}}\}$ by
	\begin{align}\label{eq: sum notation}
	\sum_{\mathcal{I}_{k}}\coloneqq\sum_{1\le \mathfrak{i}_{1},\cdots,\mathfrak{i}_{\mathfrak{k}} \le N }\indic[ \text{ $\mathfrak{i}_{1},\cdots,\mathfrak{i}_{\mathfrak{k}}$ are all distinct} ]
	\end{align}
	Let $\theta=\theta(\mathcal{I}_{k})$ be the cardinality of the set
	\begin{align}\label{eq: def of theta}
	\{ 1\le l\le k : \text{$i_{l}\not\equiv i_{l'}$ and $j_{l}\not\equiv j_{l'}$ for all $1\le l'<l$}\}.
	\end{align}
	
	For every positive integer $k\ge1$, we denote by $\mathcal{J}_{k}$ the collection of all $\mathcal{I}_{k}$ satisfying the following properties:
	\begin{itemize}
		\item[(P1)] For every $1\le l,l'\le k$, the index $i_{l}$ and $j_{l'}$ are non-equivalent.
		\item[(P2)] An index $i_{l}$ can be equivalent to $i_{l'}$ for some $l'<l$ (but not necessarily). Similarly an index $j_{l}$ can be equivalent to $j_{l'}$ for some $l'<l$ (but not necessarily).
		\item[(P3)] Suppose $l'<l$. If $i_{l}\equiv i_{l'}$, then $j_{l}$ is not equivalent to every index in $\{j_{r}:i_{r}=i_{l'}, 1\le r\le l-1\}$. Similarly, if $j_{l}\equiv j_{l'}$, then $i_{l}$ is not equivalent to every index in $\{i_{r}:j_{r}=j_{l'}, 1\le r\le l-1\}$.
	\end{itemize}
	In other words,
	\begin{equation}\label{eq: def of J_k}
	\mathcal{J}_{k} \coloneqq \{ \mathcal{I}_{k} : \text{$\mathcal{I}_{k}$ satisfies (P1), P(2) and (P3)} \}.
	\end{equation}
	\item[Notation 2] For an positive odd integer $s$ such that $s\le p$, we use the notation
	\begin{align}\label{eq: def of [p;s]}
	\sum_{[p;s]_{k}} = \begin{cases}
	\sum_{p=2}^{\ell}\combi{p}{1}\frac{(1!)\mathcal{C}_{p+1}}{q^{p-1}} & \text{$k>1$ and $s=1$},\\
	& \\
	\sum_{p=s}^{\ell}\combi{p}{s}\frac{(s!)\mathcal{C}_{p+1}}{q^{p-1}} & \text{otherwise}.
	\end{cases}
	\end{align}
%	\begin{align*}
%	\sum_{[p_{k};s_{k}]_{k}} = \begin{cases}
%	\sum_{p_{k}>1}\combi{p_{k}}{1}\frac{(1!)\mathcal{C}_{p_{k}+1}}{q^{p_{k}-1}} & \text{$k>1$ and $s_{k}=1$},\\
%	& \\
%	\sum_{p_{k}\ge s_{k}}\combi{p_{k}}{s_{k}}\frac{(s_{k}!)\mathcal{C}_{p_{k}+1}}{q^{p_{k}-1}} & \text{otherwise}.
%	\end{cases}
%	\end{align*}
\end{description}

\begin{remark}
	Throughout this paper, we always consider $\mathcal{I}_{k}\in\mathcal{J}_{k}$.
\end{remark}

Using the above new notation, we can rewrite the estimate \eqref{eq: single replacement}. First we consider
\begin{equation*}
\sum_{p_{1}=s_{1}}^{\ell}\combi{p_{1}}{s_{1}}\frac{(s_{1}!)\mathcal{C}_{p_{1}+1}}{N^{2}q^{p_{1}-1}}
\sum_{i\neq j}\expec{G_{ii}^{\frac{s_{1}+1}{2}}G_{jj}^{\frac{s_{1}+1}{2}}(\deri^{p_{1}-s_{1}}P^{2r-1})},
\end{equation*}
and
\begin{equation*}
\sum_{p_{1}=s_{1}}^{\ell}\combi{p_{1}}{s_{1}}\frac{(s_{1}!)\mathcal{C}_{p_{1}+1}}{N^{2}q^{p_{1}-1}}
\sum_{i\neq j}\expec{mG_{ii}^{\frac{s_{1}-1}{2}}G_{jj}^{\frac{s_{1}+1}{2}}(\deri^{p_{1}-s_{1}}P^{2r-1})}.
\end{equation*}
Simply setting $i_{1}\equiv i$, $j_{1}\equiv j$ and $k=1$, we have $\mathcal{I}_{1}\in\mathcal{J}_{1}$ and rewrite the above as follows:
\begin{equation}\label{eq: simple form 1}
\sum_{[p_{1};s_{1}]_{1}}\frac{1}{N^{1+\theta(\mathcal{I}_{1})}}\sum_{\mathcal{I}_{1}}
\expec{G_{i_{1}i_{1}}^{\frac{s_{1}+1}{2}}G_{j_{1}j_{1}}^{\frac{s_{1}+1}{2}}(\partial_{i_{1}j_{1}}^{p_{1}-s_{1}}P^{2r-1})}
\end{equation}
and
\begin{equation}\label{eq: simple form 2}
\sum_{[p_{1};s_{1}]_{1}}\frac{1}{N^{1+\theta(\mathcal{I}_{1})}}\sum_{\mathcal{I}_{1}}\expec{mG_{i_{1}i_{1}}^{\frac{s_{1}-1}{2}}G_{j_{1}j_{1}}^{\frac{s_{1}+1}{2}}(\partial_{i_{1}j_{1}}^{p_{1}-s_{1}}P^{2r-1})}.
\end{equation}
where we refer to \eqref{eq: def of theta} for the definition of $\theta(\mathcal{I}_{k})$ and note that $\theta(\mathcal{I}_{1})=1$.
Next we consider
\begin{multline*}
\sum_{p_{1}=s_{1}}^{\ell}\sum_{p_{2}=2}^{\ell} \combi{p_{1}}{s_{1}}\combi{p_{2}}{1}\frac{(s_{1}!)\mathcal{C}_{p_{1}+1}(1!)\mathcal{C}_{p_{2}+1}}{N^{4}q^{p_{1}+p_{2}-2}} \sum_{i\neq j}\sum_{x\neq y} \expec{ G_{xx}G_{yy}G_{ii}^{\frac{s_{1}+1}{2}}G_{jj}^{\frac{s_{1}+1}{2}}(\partial_{xy}^{p_{2}-1}\deri^{p_{1}-s_{1}}P^{2r-1}) } \\
+ \sum_{\substack{1 < s_{2} \le \ell \\ s_{2} \equiv1(\Mod2)}} \sum_{p_{1}=s_{1}}^{\ell}\sum_{p_{2}=s_{2}}^{\ell} \combi{p_{1}}{s_{1}}\combi{p_{2}}{s_{2}}\frac{(s_{1}!)\mathcal{C}_{p_{1}+1}(s_{2}!)\mathcal{C}_{p_{2}+1}}{N^{4}q^{p_{1}+p_{2}-2}} \sum_{i\neq j}\sum_{x\neq y} \expec{ G_{xx}^{\frac{s_{2}+1}{2}}G_{yy}^{\frac{s_{2}+1}{2}}G_{ii}^{\frac{s_{1}+1}{2}}G_{jj}^{\frac{s_{1}+1}{2}}(\partial_{xy}^{p_{2}-s_{2}}\deri^{p_{1}-s_{1}}P^{2r-1}) }
\end{multline*}
By setting $i_{1}\equiv i$, $j_{1}\equiv j$, $i_{2}\equiv x$, $j_{2}\equiv y$ and $k=2$, we have $\mathcal{I}_{2}\in\mathcal{J}_{2}$ and rewrite the above as follows:
\begin{equation}\label{eq: simple form 3}
\sum_{\substack{0 \le s_{2} \le \ell \\ s_{2} \equiv1(\Mod2)}} \sum_{[p_{2};s_{2}]_{2}}\sum_{[p_{1};s_{1}]_{1}}\frac{1}{N^{2+\theta(\mathcal{I}_{2})}}\sum_{\mathcal{I}_{2}} \expec{G_{i_{2}i_{2}}^{\frac{s_{2}+1}{2}}G_{j_{2}j_{2}}^{\frac{s_{2}+1}{2}}G_{i_{1}i_{1}}^{\frac{s_{1}+1}{2}}G_{j_{1}j_{1}}^{\frac{s_{1}+1}{2}}(\partial_{i_{2}j_{2}}^{p_{2}-s_{2}}\partial_{i_{1}j_{1}}^{p_{1}-s_{1}}P^{2r-1})} + \bigO_{\prec}\paren{ \Phi_{r} },
\end{equation}
where we use Proposition \ref{prop: bounds for derivatives} and \cite[Proposition A.1]{HLY20} to switch over from $\sum_{i\neq j}\sum_{x\neq y}$ to $\sum_{\mathcal{I}_{2}}$.
%$\theta$ is the cardinality of the set as in \eqref{eq: def of theta}
We also consider the summations below,
\begin{multline}\label{eq: target}
\sum_{p_{1}=s_{1}}^{\ell}\sum_{p_{2}=2}^{\ell} \combi{p_{1}}{s_{1}}\combi{p_{2}}{1}\frac{(s_{1}!)\mathcal{C}_{p_{1}+1}(1!)\mathcal{C}_{p_{2}+1}}{N^{3}q^{p_{1}+p_{2}-2}} \sum_{i\neq j}\sum_{x} \expec{ G_{xx}G_{ii}G_{ii}^{\frac{s_{1}-1}{2}}
	G_{jj}^{\frac{s_{1}+1}{2}}m(\partial_{ix}^{p_{2}-s_{2}}\deri^{p_{1}-s_{1}}P^{2r-1}) } \\
+ \sum_{\substack{0\le s_{2} \le \ell \\ s_{2} \equiv1(\Mod2)}} c_{s_{2}} \sum_{p_{1}=s_{1}}^{\ell}\sum_{p_{2}=s_{2}}^{\ell} \combi{p_{1}}{s_{1}}\combi{p_{2}}{s_{2}}\frac{(s_{1}!)\mathcal{C}_{p_{1}+1}(s_{2}!)\mathcal{C}_{p_{2}+1}}{N^{3}q^{p_{1}+p_{2}-2}} \sum_{i\neq j}\sum_{x} \expec{ G_{xx}^{\frac{s_{2}+1}{2}}G_{ii}^{\frac{s_{2}+1}{2}}G_{ii}^{\frac{s_{1}-1}{2}}
	G_{jj}^{\frac{s_{1}+1}{2}}m(\partial_{ix}^{p_{2}-s_{2}}\deri^{p_{1}-s_{1}}P^{2r-1}) }.
\end{multline}
We set $i_{1}\equiv i_{2}\equiv i$, $j_{1}\equiv j$, $j_{2}\equiv x$ and $k=2$. Then, we have $\mathcal{I}_{2}\in\mathcal{J}_{2}$. Note that $\theta(\mathcal{I}_{2})=1$ in this setting. We can replace \eqref{eq: target} with the following:
\begin{equation}\label{eq: simple form 4}
\sum_{\substack{0\le s_{2} \le \ell \\ s_{2} \equiv1(\Mod2)}} c_{s_{2}}\sum_{[p_{2};s_{2}]_{2}}\sum_{[p_{1};s_{1}]_{1}}\frac{1}{N^{2+\theta(\mathcal{I}_{2})}}\sum_{\mathcal{I}_{2}} \expec{G_{i_{2}i_{2}}^{\frac{s_{2}+1}{2}}G_{j_{2}j_{2}}^{\frac{s_{2}+1}{2}}G_{i_{1}i_{1}}^{\frac{s_{1}-1}{2}}G_{j_{1}j_{1}}^{\frac{s_{1}+1}{2}}m(\partial_{i_{2}j_{2}}^{p_{2}-s_{2}}\partial_{i_{1}j_{1}}^{p_{1}-s_{1}}P^{2r-1})} + \bigO_{\prec}\paren{ \Phi_{r} },
\end{equation}
where we use that $c_{s_{2}}=1$ when $s_{2}=1$, and switch over from $\sum_{i\neq j}\sum_{x}$ to $\sum_{\mathcal{I}_{2}}$. Now we can restate Claim \ref{claim: replacing single diagonal} as below.

\begin{claim}[Simple form of Claim \ref{claim: replacing single diagonal}]\label{claim: replacing single diagonal 2}
	Let $s_{1}$ be a positive odd integer. Let $\mathcal{J}_{2}$ be as in \eqref{eq: def of J_k} (with $k=2$). Then, we have
	\begin{multline}\label{eq: replacing diagonal first time}
	\sum_{[p_{1};s_{1}]_{1}}\frac{1}{N^{1+\theta(\mathcal{I}_{1})}}\sum_{\mathcal{I}_{1}}
	\expec{G_{i_{1}i_{1}}^{\frac{s_{1}+1}{2}}G_{j_{1}j_{1}}^{\frac{s_{1}+1}{2}}(\partial_{i_{1}j_{1}}^{p_{1}-s_{1}}P^{2r-1})}\\
	= \sum_{[p_{1};s_{1}]_{1}}\frac{1}{N^{1+\theta(\mathcal{I}_{1})}}\sum_{\mathcal{I}_{1}}\expec{mG_{i_{1}i_{1}}^{\frac{s_{1}-1}{2}}G_{j_{1}j_{1}}^{\frac{s_{1}+1}{2}}(\partial_{i_{1}j_{1}}^{p_{1}-s_{1}}P^{2r-1})} \\
	+ \sum_{\mathcal{I}_{2}\in\mathcal{J}_{2}} \sum_{\substack{0\le s_{2} \le \ell \\ s_{2} \equiv1(\Mod2)}} \sum_{[p_{2};s_{2}]_{2}}\sum_{[p_{1};s_{1}]_{1}} \frac{c_{\mathcal{I}_{2},s_{2}}}{N^{2+\theta(\mathcal{I}_{2})}}\sum_{\mathcal{I}_{2}} \expec{G_{i_{2}i_{2}}^{\frac{s_{2}+1}{2}}G_{j_{2}j_{2}}^{\frac{s_{2}+1}{2}}G_{i_{1}i_{1}}^{\frac{s_{1}+1}{2}-\indic(i_{2}\equiv i_{1})}G_{j_{1}j_{1}}^{\frac{s_{1}+1}{2}}m^{\indic(i_{2}\equiv i_{1})}(\partial_{i_{2}j_{2}}^{p_{2}-s_{2}}\partial_{i_{1}j_{1}}^{p_{1}-s_{1}}P^{2r-1})} \\
	+ \bigO_{\prec}\paren{ \Phi_{r} },
	\end{multline}
	where each $c_{\mathcal{I}_{2},s_{2}}$ is a constant only depending on $\mathcal{I}_{2}\in\mathcal{J}_{2}$ and $s_{2} \equiv1(\Mod2)$. (Refer to \eqref{eq: def of [p;s]} for the definition of $\sum_{[p;s]_{k}}$.)
\begin{proof}
	We compare \eqref{eq: single replacement} and \eqref{eq: replacing diagonal first time}. It immediately follows from \eqref{eq: simple form 1}, \eqref{eq: simple form 2}, \eqref{eq: simple form 3} and \eqref{eq: simple form 4}.
\end{proof}
\end{claim}
\begin{remark}
	In fact, $c_{\mathcal{I}_{2},s_{2}}=0$ if $\mathcal{I}_{2}$ does not belong to one of the following two cases: (case 1) $i_{1},j_{1},i_{2},j_{2}$ are all non-equivalent; (case 2) $i_{1},j_{1},j_{2}$ are all non-equivalent but $i_{2}\equiv i_{1}$.
\end{remark}

Now we just replace one diagonal entry. We want to replace all diagonal entries with the normalized trace $m$. We need the following lemma for that.
\begin{lemma}\label{lem: replacing single diagoanl general version}
	Let $d\ge 0$ be a non-negative integer. Let $t\ge 1$ be a positive integer. Fix an integer $k\ge 1$. Let $\{u_{j}\}_{j=1}^{k}$ be a finite sequence of non-negative integers. Let $\{v_{j}\}_{j=1}^{k}$ be a finite sequence of positive integers. Let $D(P)$ be a $l$-th order derivative of $P^{2r-1}$ with $l\ge0$, and we use the convention that $D(P)=P^{2r-1}$ for $l=0$. Then, we have
	\begin{multline*}
	\expec{ m^{d} G_{ii}^{t} \bigg(\prod_{j=1}^{k} G_{v_{j}v_{j}}^{u_{j}}\bigg) D(P) }
	= \expec{ m^{d+1} G_{ii}^{t-1} \bigg(\prod_{j=1}^{k} G_{v_{j}v_{j}}^{u_{j}}\bigg) D(P) } \\
	- \sum_{\substack{0\le s \le \ell \\ s \equiv1(\Mod2)}}\sum_{[p;s]} \frac{1}{N^{2}} \sum_{x\neq y} \expec{ m^{d} G_{xx}^{\frac{s+1}{2}}G_{yy}^{\frac{s+1}{2}}G_{ii}^{t} \bigg(\prod_{j=1}^{k} G_{v_{j}v_{j}}^{u_{j}}\bigg) \bigg(\partial_{xy}^{p-s}D(P)\bigg) } \\
	+ \sum_{\substack{0\le s \le \ell \\ s \equiv1(\Mod2)}} c_{s} \sum_{[p;s]} \frac{1}{N} \sum_{x} \expec{ m^{d+1} G_{xx}^{\frac{s+1}{2}}G_{ii}^{\frac{s+1}{2}}G_{ii}^{t-1} \bigg(\prod_{j=1}^{k} G_{v_{j}v_{j}}^{u_{j}}\bigg) \bigg(\partial_{ix}^{p-s}D(P)\bigg) } + \bigO_{\prec}\paren{ \Phi_{r} },
	\end{multline*}
	where each coefficient $c_{s}$ is as in Lemma \ref{lem: moment estimate general} and,
	\begin{equation*}
	\sum_{[p;s]} = \begin{cases}
	\sum_{p=2}^{\ell}\combi{p}{1}\frac{(1!)\mathcal{C}_{p+1}}{q^{p-1}} & \text{$s=1$},\\
	& \\
	\sum_{p=s}^{\ell}\combi{p}{s}\frac{(s!)\mathcal{C}_{p+1}}{q^{p-1}} & \text{$s>1$}.
	\end{cases}
	\end{equation*}
\begin{proof}
	We can prove this following the proof of Claim \ref{claim: replacing single diagonal}.
	\begin{align*} %\label{eq: resolvent entry identity 3}
	G_{ii} = m + \frac{G_{ii}}{N}\sum_{x,y}h_{xy}G_{yx} - m\sum_{x}h_{ix}G_{xi}
	\end{align*}
	Using the identity \eqref{eq: resolvent entry identity 3}, we have
	\begin{multline}\label{eq: e1}
	m^{d} G_{ii}^{t} \bigg(\prod_{j=1}^{k} G_{v_{j}v_{j}}^{u_{j}}\bigg) D(P) 
	= m^{d+1} G_{ii}^{t-1} \bigg(\prod_{j=1}^{k} G_{v_{j}v_{j}}^{u_{j}}\bigg) D(P) \\
	+ \frac{1}{N}\sum_{x,y} h_{xy}m^{d}G_{yx}G_{ii}^{t} \bigg(\prod_{j=1}^{k} G_{v_{j}v_{j}}^{u_{j}}\bigg) D(P)
	- \sum_{x} h_{ix}m^{d+1}G_{xi}G_{ii}^{t-1} \bigg(\prod_{j=1}^{k} G_{v_{j}v_{j}}^{u_{j}}\bigg) D(P).
	\end{multline}
    Applying the cumulant expansion to the last two terms, we get
	\begin{multline*}
	\frac{1}{N}\sum_{x,y}\expec{ h_{xy}m^{d}G_{yx}G_{ii}^{t} \bigg(\prod_{j=1}^{k} G_{v_{j}v_{j}}^{u_{j}}\bigg) D(P) } \\
	= \sum_{p=1}^{\ell}\frac{\mathcal{C}_{p+1}}{N^{2}q^{p-1}}\sum_{x,y}\expec{ \partial_{xy}^{p}\paren{ m^{d}G_{yx}G_{ii}^{t} \bigg(\prod_{j=1}^{k} G_{v_{j}v_{j}}^{u_{j}}\bigg) D(P) } } + \bigO_{\prec}\paren{ \Phi_{r} },
	\end{multline*}
	and
	\begin{multline*}
	\sum_{x}\expec{ h_{ix}m^{d+1}G_{xi}G_{ii}^{t-1} \bigg(\prod_{j=1}^{k} G_{v_{j}v_{j}}^{u_{j}}\bigg) D(P) } \\
	= \sum_{p=1}^{\ell}\frac{\mathcal{C}_{p+1}}{Nq^{p-1}}\sum_{x}\expec{ \partial_{ix}^{p}\paren{ m^{d+1}G_{xi}G_{ii}^{t-1} \bigg(\prod_{j=1}^{k} G_{v_{j}v_{j}}^{u_{j}}\bigg) D(P) } } + \bigO_{\prec}\paren{ \Phi_{r} }.
	\end{multline*}
	Using Lemma \ref{lem: moment estimate general} directly, we obtain
	\begin{multline}\label{eq: e}
	\sum_{p=1}^{\ell}\frac{\mathcal{C}_{p+1}}{N^{2}q^{p-1}}\sum_{x,y}\expec{ \partial_{xy}^{p}\paren{ m^{d}G_{yx}G_{ii}^{t} \bigg(\prod_{j=1}^{k} G_{v_{j}v_{j}}^{u_{j}}\bigg) D(P) } } \\ 
	- \sum_{p=1}^{\ell}\frac{\mathcal{C}_{p+1}}{Nq^{p-1}}\sum_{x}\expec{ \partial_{ix}^{p}\paren{ m^{d+1}G_{xi}G_{ii}^{t-1} \bigg(\prod_{j=1}^{k} G_{v_{j}v_{j}}^{u_{j}}\bigg) D(P) } } \\
	= - \sum_{\substack{0\le s \le \ell \\ s \equiv1(\Mod2)}}\sum_{p=s}^{\ell} \combi{p}{s}\frac{(s!)\mathcal{C}_{p+1}}{N^{2}q^{p-1}} \sum_{x\neq y} \expec{ m^{d} G_{xx}^{\frac{s+1}{2}}G_{yy}^{\frac{s+1}{2}}G_{ii}^{t} \bigg(\prod_{j=1}^{k} G_{v_{j}v_{j}}^{u_{j}}\bigg) \bigg(\partial_{xy}^{p-s}D(P)\bigg) } \\
	 + \sum_{\substack{0\le s \le \ell \\ s \equiv1(\Mod2)}}c_{s}\sum_{p=s}^{\ell} \combi{p}{s}\frac{(s!)\mathcal{C}_{p+1}}{Nq^{p-1}} \sum_{x} \expec{ m^{d+1} G_{xx}^{\frac{s+1}{2}}G_{ii}^{\frac{s+1}{2}}G_{ii}^{t-1} \bigg(\prod_{j=1}^{k} G_{v_{j}v_{j}}^{u_{j}}\bigg) \bigg(\partial_{ix}^{p-s}D(P)\bigg) } + \bigO_{\prec}\paren{ \Phi_{r} },
	\end{multline}
	where each $c_{s}$ is as in Lemma \ref{lem: moment estimate general}. For the case that $p=1$ and $s=1$, we find a cancellation in the right-hand sides of \eqref{eq: e} as follows: using the fact that $c_{s}=1$ for $s=1$, we have
	\begin{multline*}
	\frac{1}{N^{2}} \sum_{x\neq y} \expec{ m^{d} G_{xx}G_{yy}G_{ii}^{t} \bigg(\prod_{j=1}^{k} G_{v_{j}v_{j}}^{u_{j}}\bigg) D(P) } 
	- \frac{1}{N} \sum_{x} \expec{ m^{d+1} G_{xx}G_{ii}^{t} \bigg(\prod_{j=1}^{k} G_{v_{j}v_{j}}^{u_{j}}\bigg) D(P) } \\
	= \frac{1}{N^{2}} \sum_{x,y} \expec{ m^{d} G_{xx}G_{yy}G_{ii}^{t} \bigg(\prod_{j=1}^{k} G_{v_{j}v_{j}}^{u_{j}}\bigg) D(P) } 
	- \expec{ m^{d+2} G_{ii}^{t} \bigg(\prod_{j=1}^{k} G_{v_{j}v_{j}}^{u_{j}}\bigg) D(P) } + \bigO_{\prec}\paren{ \Phi_{r} } \\
	= \expec{ m^{d+2}G_{ii}^{t} \bigg(\prod_{j=1}^{k} G_{v_{j}v_{j}}^{u_{j}}\bigg) D(P) } 
	- \expec{ m^{d+2} G_{ii}^{t} \bigg(\prod_{j=1}^{k} G_{v_{j}v_{j}}^{u_{j}}\bigg) D(P) } + \bigO_{\prec}\paren{ \Phi_{r} } \\
	= \bigO_{\prec}\paren{ \Phi_{r} },
	\end{multline*}
	where we use Proposition \ref{prop: bounds for derivatives} and \cite[Proposition A.1]{HLY20} to switch over from $\sum_{x\neq y}$ to $\sum_{x,y}$. Due to the above cancellation, we get
	\begin{multline}\label{eq: e2}
	- \sum_{\substack{0\le s \le \ell \\ s \equiv1(\Mod2)}}\sum_{p=s}^{\ell} \combi{p}{s}\frac{(s!)\mathcal{C}_{p+1}}{N^{2}q^{p-1}} \sum_{x\neq y} \expec{ m^{d} G_{xx}^{\frac{s+1}{2}}G_{yy}^{\frac{s+1}{2}}G_{ii}^{t} \bigg(\prod_{j=1}^{k} G_{v_{j}v_{j}}^{u_{j}}\bigg) \bigg(\partial_{xy}^{p-s}D(P)\bigg) } \\
	+ \sum_{\substack{0\le s \le \ell \\ s \equiv1(\Mod2)}}c_{s}\sum_{p=s}^{\ell} \combi{p}{s}\frac{(s!)\mathcal{C}_{p+1}}{Nq^{p-1}} \sum_{x} \expec{ m^{d+1} G_{xx}^{\frac{s+1}{2}}G_{ii}^{\frac{s+1}{2}}G_{ii}^{t-1} \bigg(\prod_{j=1}^{k} G_{v_{j}v_{j}}^{u_{j}}\bigg) \bigg(\partial_{ix}^{p-s}D(P)\bigg) } \\
	= - \sum_{\substack{0\le s \le \ell \\ s \equiv1(\Mod2)}}\sum_{[p;s]} \frac{1}{N^{2}} \sum_{x\neq y} \expec{ m^{d} G_{xx}^{\frac{s+1}{2}}G_{yy}^{\frac{s+1}{2}}G_{ii}^{t} \bigg(\prod_{j=1}^{k} G_{v_{j}v_{j}}^{u_{j}}\bigg) \bigg(\partial_{xy}^{p-s}D(P)\bigg) } \\
	+ \sum_{\substack{0\le s \le \ell \\ s \equiv1(\Mod2)}}\sum_{[p;s]} \frac{c_{s}}{N} \sum_{x} \expec{ m^{d+1} G_{xx}^{\frac{s+1}{2}}G_{ii}^{\frac{s+1}{2}}G_{ii}^{t-1} \bigg(\prod_{j=1}^{k} G_{v_{j}v_{j}}^{u_{j}}\bigg) \bigg(\partial_{ix}^{p-s}D(P)\bigg) } + \bigO_{\prec}\paren{ \Phi_{r} }.
	\end{multline}
	The lemma follows from \eqref{eq: e1}, \eqref{eq: e} and \eqref{eq: e2}.
\end{proof}
\end{lemma}

Using Lemma \ref{lem: replacing single diagoanl general version}, we can replace all diagonal entries in the left-hand side of \eqref{eq: replacing diagonal first time} one by one.
\begin{proposition}\label{prop: replacing all diagonals}
	Suppose Assumption \ref{assump: random correction term 1} and Assumption \ref{assump: random correction term 2} hold. Let $s_{1}$ be a positive odd integer such that $s_{1}\le\ell$. The difference,
	\begin{equation*}
	\sum_{[p_{1};s_{1}]_{1}}\frac{1}{N^{1+\theta(\mathcal{I}_{1})}}\sum_{\mathcal{I}_{1}}
	\expec{G_{i_{1}i_{1}}^{\frac{s_{1}+1}{2}}G_{j_{1}j_{1}}^{\frac{s_{1}+1}{2}}(\partial_{i_{1}j_{1}}^{p_{1}-s_{1}}P^{2r-1})} - \sum_{[p_{1};s_{1}]_{1}}\frac{1}{N^{1+\theta(\mathcal{I}_{1})}}\sum_{\mathcal{I}_{1}}
	\expec{m^{s_{1}+1}(\partial_{i_{1}j_{1}}^{p_{1}-s_{1}}P^{2r-1})},
	\end{equation*}
	is a linear combination of the terms of the following form (with bounded coefficients):
	\begin{equation}\label{eq: form}
	\sum_{[p_{2};s_{2}]_{2}}\sum_{[p_{1};s_{1}]_{1}}\sum_{\mathcal{I}_{2}} \expec{T} + \bigO_{\prec}(\Phi_{r}),
	\end{equation}
	where $s_{2}$ is a positive odd integer
	such that $s_{2}\le\ell$, we have $\mathcal{I}_{2}\in\mathcal{J}_{2}$ (refer to \eqref{eq: def of J_k} for the definition of $\mathcal{J}_{k}$), and for some $0\le \alpha_{1} \le\frac{s_{1}-1}{2}$,
	\begin{align*}
	T = \frac{1}{N^{2+\theta(\mathcal{I}_{2})}}G_{i_{2}i_{2}}^{\frac{s_{2}+1}{2}}G_{j_{2}j_{2}}^{\frac{s_{2}+1}{2}}G_{i_{1}i_{1}}^{\frac{s_{1}+1}{2}}G_{j_{1}j_{1}}^{\frac{s_{1}+1}{2}-\alpha_{1}-\indic(j_{2}\equiv j_{1})}m^{\alpha_{1}+\indic(j_{2}\equiv j_{1})}
	\paren{\partial_{i_{2}j_{2}}^{p_{2}-s_{2}}\partial_{i_{1}j_{1}}^{p_{1}-s_{1}}P^{2r-1}},
	\end{align*}
	or
	\begin{align*}
	T = \frac{1}{N^{2+\theta(\mathcal{I}_{2})}}G_{i_{2}i_{2}}^{\frac{s_{2}+1}{2}}G_{j_{2}j_{2}}^{\frac{s_{2}+1}{2}}G_{i_{1}i_{1}}^{\frac{s_{1}+1}{2}-\alpha_{1}-\indic(i_{2}\equiv i_{1})}m^{\frac{s_{1}+1}{2}+\alpha_{1}+\indic(i_{2}\equiv i_{1})}
	\paren{\partial_{i_{2}j_{2}}^{p_{2}-s_{2}}\partial_{i_{1}j_{1}}^{p_{1}-s_{1}}P^{2r-1}}.
	\end{align*}
	(Refer to \eqref{eq: def of [p;s]} for the definition of $\sum_{[p;s]_{k}}$.)\\
	In fact, the length of the linear combination can be bounded by a large constant only depending on $\ell$ but not $N$.
\begin{proof}
	This time we replace $G_{j_{1}j_{1}}$ first. By symmetry, it follows from Claim \ref{claim: replacing single diagonal 2} that the difference,
	\begin{equation*}
	\sum_{[p_{1};s_{1}]_{1}}\frac{1}{N^{1+\theta(\mathcal{I}_{1})}}\sum_{\mathcal{I}_{1}}
	\expec{G_{i_{1}i_{1}}^{\frac{s_{1}+1}{2}}G_{j_{1}j_{1}}^{\frac{s_{1}+1}{2}}(\partial_{i_{1}j_{1}}^{p_{1}-s_{1}}P^{2r-1})}
	- \sum_{[p_{1};s_{1}]_{1}}\frac{1}{N^{1+\theta(\mathcal{I}_{1})}}\sum_{\mathcal{I}_{1}}
	\expec{G_{i_{1}i_{1}}^{\frac{s_{1}+1}{2}}G_{j_{1}j_{1}}^{\frac{s_{1}-1}{2}}m(\partial_{i_{1}j_{1}}^{p_{1}-s_{1}}P^{2r-1})},
	\end{equation*}
	is a linear combination of the terms of the form \eqref{eq: form} (with bounded coefficients) where $s_{2}\equiv 1 (\Mod 2)$, $0\le s_{2}\le \ell$, $\mathcal{I}_{2}\in\mathcal{J}_{2}$ and the term $T$ has the form
	\begin{equation*}
	\frac{1}{N^{2+\theta(\mathcal{I}_{2})}}G_{i_{2}i_{2}}^{\frac{s_{2}+1}{2}}G_{j_{2}j_{2}}^{\frac{s_{2}+1}{2}}G_{i_{1}i_{1}}^{\frac{s_{1}+1}{2}}G_{j_{1}j_{1}}^{\frac{s_{1}+1}{2}-\indic(j_{2}\equiv j_{1})}m^{\indic(j_{2}\equiv j_{1})}
	\paren{\partial_{i_{2}j_{2}}^{p_{2}-s_{2}}\partial_{i_{1}j_{1}}^{p_{1}-s_{1}}P^{2r-1}}.
	\end{equation*}
	According to Claim \ref{claim: replacing single diagonal 2}, we can find that the length of the linear combination can be bounded by a large constant only depending on $\ell$.
	
	Next, in order to replace one more diagonal entry, we apply Lemma \ref{lem: replacing single diagoanl general version} to
	\begin{equation*}
	\sum_{[p_{1};s_{1}]_{1}}\frac{1}{N^{1+\theta(\mathcal{I}_{1})}}\sum_{\mathcal{I}_{1}}
	\expec{G_{i_{1}i_{1}}^{\frac{s_{1}+1}{2}}G_{j_{1}j_{1}}^{\frac{s_{1}-1}{2}}m(\partial_{i_{1}j_{1}}^{p_{1}-s_{1}}P^{2r-1})}.
	\end{equation*}
	We assume $\frac{s_{1}-1}{2}\ge 1$. (Otherwise we will use Lemma \ref{lem: replacing single diagoanl general version} to replace $G_{i_{1}i_{1}}$.) Setting $d=1$, $t=\frac{s_{1}-1}{2}$ and $G_{ii}=G_{j_{1}j_{1}}$ in Lemma \ref{lem: replacing single diagoanl general version}, we find that the difference,
	\begin{equation*}
	\sum_{[p_{1};s_{1}]_{1}}\frac{1}{N^{1+\theta(\mathcal{I}_{1})}}\sum_{\mathcal{I}_{1}}
	\expec{G_{i_{1}i_{1}}^{\frac{s_{1}+1}{2}}G_{j_{1}j_{1}}^{\frac{s_{1}-1}{2}}m(\partial_{i_{1}j_{1}}^{p_{1}-s_{1}}P^{2r-1})}
	- \sum_{[p_{1};s_{1}]_{1}}\frac{1}{N^{1+\theta(\mathcal{I}_{1})}}\sum_{\mathcal{I}_{1}}
	\expec{G_{i_{1}i_{1}}^{\frac{s_{1}+1}{2}}G_{j_{1}j_{1}}^{\frac{s_{1}-3}{2}}m(\partial_{i_{1}j_{1}}^{p_{1}-s_{1}}P^{2r-1})},
	\end{equation*}
	is, again, a linear combination of the terms of the form \eqref{eq: form} (with bounded coefficients) where $s_{2}\equiv 1 (\Mod 2)$, $0\le s_{2}\le \ell$, $\mathcal{I}_{2}\in\mathcal{J}_{2}$ and the term $T$ has the form
	\begin{equation*}
	\frac{1}{N^{2+\theta(\mathcal{I}_{2})}}G_{i_{2}i_{2}}^{\frac{s_{2}+1}{2}}G_{j_{2}j_{2}}^{\frac{s_{2}+1}{2}}G_{i_{1}i_{1}}^{\frac{s_{1}+1}{2}}G_{j_{1}j_{1}}^{\frac{s_{1}-1}{2}-\indic(j_{2}\equiv j_{1})}m^{1+\indic(j_{2}\equiv j_{1})}
	\paren{\partial_{i_{2}j_{2}}^{p_{2}-s_{2}}\partial_{i_{1}j_{1}}^{p_{1}-s_{1}}P^{2r-1}}.
	\end{equation*}
    The length of the linear combination also can be bounded by a large constant only depending on $\ell$.
    
    We can apply Lemma \ref{lem: replacing single diagoanl general version} this way repetitively until there is no diagonal entry to replace. Then, the desired result follows.
\end{proof}
\end{proposition}

\subsection{Step 4.~Iterations}
We shall use the approach of \cite{HK21}. In this subsection, we will show that the main term \eqref{eq: main terms alt} can be approximated by a linear combination of the terms which have a certain form. Let us set
\begin{align*}
T_{1}=T(s_{1},p_{1})\coloneqq \frac{1}{N^{1+\theta(\mathcal{I}_{1})}}G_{i_{1}i_{1}}^{\frac{s_{1}+1}{2}}G_{j_{1}j_{1}}^{\frac{s_{1}+1}{2}}(\partial_{i_{1}j_{1}}^{p_{1}-s_{1}}P^{2r-1}).
\end{align*}
According to Proposition \ref{prop: replacing all diagonals}, we obtain
\begin{equation*} %\label{eq: general case, auxiliary 1}
\sum_{[p_{1};s_{1}]_{1}}\sum_{\mathcal{I}_{1}} \expec{ T_{1} }
= \sum_{[p_{1};s_{1}]_{1}}\frac{1}{N^{1+\theta(\mathcal{I}_{1})}}\sum_{\mathcal{I}_{1}} \expec{  m^{s_{1}+1}(\partial_{i_{1}j_{1}}^{p_{1}-s_{1}}P^{2r-1}) } + \sum_{t}c_{t}\sum_{[p_{2};s_{2}]_{2}}\sum_{[p_{1};s_{1}]_{1}}\sum_{\mathcal{I}_{2}} \expec{T_{2}^{(t)}} + \bigO_{\prec}(\Phi_{r}),
\end{equation*}
where $s_{2}\equiv 1 (\Mod 2)$, $0\le s_{2}\le \ell$, $\mathcal{I}_{2}\in\mathcal{J}_{2}$ (refer to \eqref{eq: def of J_k} for the definition of $\mathcal{J}_{k}$), the summation $\sum_{t}$ has a finite length which depends on $\ell$ but not $N$, the coefficient $c_{t}$ is bounded for every $t$, and $T_{2}^{(t)}$ in the summation has the following forms, for some $0\le \alpha_{1} \le\frac{s_{1}-1}{2}$,
\begin{align*}
\frac{1}{N^{2+\theta(\mathcal{I}_{2})}}G_{i_{2}i_{2}}^{\frac{s_{2}+1}{2}}G_{j_{2}j_{2}}^{\frac{s_{2}+1}{2}}G_{i_{1}i_{1}}^{\frac{s_{1}+1}{2}}G_{j_{1}j_{1}}^{\frac{s_{1}+1}{2}-\alpha_{1}-\indic(j_{2}\equiv j_{1})}m^{\alpha_{1}+\indic(j_{2}\equiv j_{1})}
\paren{\partial_{i_{2}j_{2}}^{p_{2}-s_{2}}\partial_{i_{1}j_{1}}^{p_{1}-s_{1}}P^{2r-1}},
\end{align*}
or
\begin{align*}
\frac{1}{N^{2+\theta(\mathcal{I}_{2})}}G_{i_{2}i_{2}}^{\frac{s_{2}+1}{2}}G_{j_{2}j_{2}}^{\frac{s_{2}+1}{2}}G_{i_{1}i_{1}}^{\frac{s_{1}+1}{2}-\alpha_{1}-\indic(i_{2}\equiv i_{1})}m^{\frac{s_{1}+1}{2}+\alpha_{1}+\indic(i_{2}\equiv i_{1})}
\paren{\partial_{i_{2}j_{2}}^{p_{2}-s_{2}}\partial_{i_{1}j_{1}}^{p_{1}-s_{1}}P^{2r-1}}.
\end{align*}
Note that
\begin{equation*}
\size{ \sum_{t}\sum_{[p_{2};s_{2}]_{2}}\sum_{[p_{1};s_{1}]_{1}}\sum_{\mathcal{I}_{2}} \expec{T_{2}^{(t)}} } \prec \frac{1}{q}\max_{s\ge 0}\expec{\paren{\frac{|\partial_{2}P|\Im m(\tz)}{N\eta}+\paren{\frac{\Im m(\tz)}{N\eta}}^{2}+\frac{1}{N}}^{s}|P|^{2r-s-1} },
\end{equation*}
because in the summation $\sum_{[p_{2};s_{2}]_{2}}$, we have the factor $q^{p_{2}-1}$ with $p_{2} > 1$, and also we can use \eqref{eq: derivative P}.

\begin{remark}
	We will denote by $c_{t}$ a coefficient of a linear combination. The coefficient $c_{t}$ may differ whenever it occurs. It is an abuse of notation for convenience.
\end{remark}

Pick a single $T_{2}^{(t)}$ and drop the the superscript $^{(t)}$ for brevity. Following the proof of Proposition \ref{prop: replacing all diagonals}, we replace all diagonal entries in $T_{2}$ using Lemma \ref{lem: replacing single diagoanl general version} and keep track of remaining terms. It follows that
\begin{multline*}
\sum_{[p_{2};s_{2}]_{2}}\sum_{[p_{1};s_{1}]_{1}}\sum_{\mathcal{I}_{2}} \expec{T_{2}} \\
= \sum_{[p_{2};s_{2}]_{2}}\sum_{[p_{1};s_{1}]_{1}}\frac{1}{N^{2+\theta(\mathcal{I}_{2})}}\sum_{\mathcal{I}_{2}} \expec{ m^{s_{1}+s_{2}+2}(\partial_{i_{2}j_{2}}^{p_{2}-s_{2}}\partial_{i_{1}j_{1}}^{p_{1}-s_{1}}P^{2r-1}) } \\
+ \sum_{t}c_{t}\sum_{[p_{3};s_{3}]_{3}}\sum_{[p_{2};s_{2}]_{2}}\sum_{[p_{1};s_{1}]_{1}}\sum_{\mathcal{I}_{3}}\expec{T_{3}^{(t)}} + \bigO_{\prec}(\Phi_{r}),
\end{multline*}
where $s_{3}\equiv 1 (\Mod 2)$, $0\le s_{3}\le \ell$, $\mathcal{I}_{3}\in\mathcal{J}_{3}$, the summation $\sum_{t}$ has a finite length which depends on $\ell$ but not $N$, the coefficient $c_{t}$ is bounded for every $t$, and $T_{3}^{(t)}$ in the summation has the following forms, for some $0\le \alpha_{1} \le\frac{s_{1}-1}{2}$ and $0\le \alpha_{2} \le\frac{s_{2}-1}{2}$, 
\begin{multline*}
\frac{1}{N^{3+\theta(\mathcal{I}_{3})}}G_{i_{3}i_{3}}^{\frac{s_{3}+1}{2}}G_{j_{3}j_{3}}^{\frac{s_{3}+1}{2}}
G_{i_{2}i_{2}}^{\frac{s_{2}+1}{2}}G_{j_{2}j_{2}}^{\frac{s_{2}+1}{2}}
G_{i_{1}i_{1}}^{\frac{s_{1}+1}{2}}G_{j_{1}j_{1}}^{\frac{s_{1}+1}{2}-\alpha_{1}-\indic(j_{3}\equiv j_{1})} m^{\alpha_{1}+\indic(j_{3}\equiv j_{1})}\paren{\partial_{i_{3}j_{3}}^{p_{3}-s_{3}}\partial_{i_{2}j_{2}}^{p_{2}-s_{2}}\partial_{i_{1}j_{1}}^{p_{1}-s_{1}}P^{2r-1}},
\end{multline*}
\begin{multline*}
\frac{1}{N^{3+\theta(\mathcal{I}_{3})}}G_{i_{3}i_{3}}^{\frac{s_{3}+1}{2}}G_{j_{3}j_{3}}^{\frac{s_{3}+1}{2}}
G_{i_{2}i_{2}}^{\frac{s_{2}+1}{2}}G_{j_{2}j_{2}}^{\frac{s_{2}+1}{2}}
G_{i_{1}i_{1}}^{\frac{s_{1}+1}{2}-\alpha_{1}-\indic(i_{3}\equiv i_{1})} m^{\frac{s_{1}+1}{2}+\alpha_{1}+\indic(i_{3}\equiv i_{1})}\paren{\partial_{i_{3}j_{3}}^{p_{3}-s_{3}}\partial_{i_{2}j_{2}}^{p_{2}-s_{2}}\partial_{i_{1}j_{1}}^{p_{1}-s_{1}}P^{2r-1}},
\end{multline*}
\begin{multline*}
\frac{1}{N^{3+\theta(\mathcal{I}_{3})}}G_{i_{3}i_{3}}^{\frac{s_{3}+1}{2}}G_{j_{3}j_{3}}^{\frac{s_{3}+1}{2}}
G_{i_{2}i_{2}}^{\frac{s_{2}+1}{2}}G_{j_{2}j_{2}}^{\frac{s_{2}+1}{2}-\alpha_{2}-\indic(j_{3}\equiv j_{2})} m^{(s_{1}+1)+\alpha_{2}+\indic(j_{3}\equiv j_{2})}\paren{\partial_{i_{3}j_{3}}^{p_{3}-s_{3}}\partial_{i_{2}j_{2}}^{p_{2}-s_{2}}\partial_{i_{1}j_{1}}^{p_{1}-s_{1}}P^{2r-1}},
\end{multline*}
or
\begin{multline*}
\frac{1}{N^{3+\theta(\mathcal{I}_{3})}}G_{i_{3}i_{3}}^{\frac{s_{3}+1}{2}}G_{j_{3}j_{3}}^{\frac{s_{3}+1}{2}}
G_{i_{2}i_{2}}^{\frac{s_{2}+1}{2}-\alpha_{2}-\indic(i_{3}\equiv i_{2})} m^{\frac{s_{2}+1}{2}+(s_{1}+1)+\alpha_{2}+\indic(i_{3}\equiv i_{2})}\paren{\partial_{i_{3}j_{3}}^{p_{3}-s_{3}}\partial_{i_{2}j_{2}}^{p_{2}-s_{2}}\partial_{i_{1}j_{1}}^{p_{1}-s_{1}}P^{2r-1}}.
\end{multline*}
Note that
\begin{equation*}
\size{ \sum_{t}\sum_{[p_{3};s_{3}]_{2}}\sum_{[p_{2};s_{2}]_{2}}\sum_{[p_{1};s_{1}]_{1}}\sum_{\mathcal{I}_{2}}\sum_{i_{3},j_{3}}\expec{T_{3}^{(t)}} } \prec \frac{1}{q^{2}}\max_{s\ge 0}\expec{\paren{\frac{|\partial_{2}P|\Im m(\tz)}{N\eta}+\paren{\frac{\Im m(\tz)}{N\eta}}^{2}+\frac{1}{N}}^{s}|P|^{2r-s-1} },
\end{equation*}
where we use \eqref{eq: derivative P} and the fact that in the summation $\sum_{[p_{l};s_{l}]_{l}}$, if $l\ge 2$, we have the factor $q^{p_{l}-1}$ with $p_{l} > 1$.

We shall iterate this procedure until the remaining terms are absorbed into $\bigO_{\prec}(\Phi_{r})$ by getting many factors of $q^{-1}$. We assume that $\mathcal{I}_{k}\in\mathcal{J}_{k}$ and $T_{k}$ has the following forms, for $1\le l\le k-1$ and $0\le \alpha_{l} \le\frac{s_{l}-1}{2}$,
\begin{multline}\label{eq: g1}
\frac{1}{N^{k+\theta(\mathcal{I}_{k})}}G_{i_{k}i_{k}}^{\frac{s_{k}+1}{2}}G_{j_{k}j_{k}}^{\frac{s_{k}+1}{2}}\ldots
G_{i_{l}i_{l}}^{\frac{s_{l}+1}{2}}G_{j_{l}j_{l}}^{\frac{s_{l}+1}{2}-\alpha_{l}-\indic(j_{k}\equiv j_{l})}
m^{\sum_{l'=1}^{l-1}(s_{l'+1})+\alpha_{l}+\indic(j_{k}\equiv j_{l})}\paren{\partial_{i_{k}j_{k}}^{p_{k}-s_{k}}\ldots\partial_{i_{1}j_{1}}^{p_{1}-s_{1}}P^{2r-1}},
\end{multline}
or
\begin{multline}\label{eq: g2}
\frac{1}{N^{k+\theta(\mathcal{I}_{k})}}G_{i_{k}i_{k}}^{\frac{s_{k}+1}{2}}G_{j_{k}j_{k}}^{\frac{s_{k}+1}{2}}\ldots
G_{i_{l}i_{l}}^{\frac{s_{l}+1}{2}-\alpha_{l}-\indic(i_{k}\equiv i_{l})}
m^{\sum_{l'=1}^{l-1}(s_{l'+1})+\frac{s_{l}+1}{2}+\alpha_{l}+\indic(i_{k}\equiv i_{l})}\paren{\partial_{i_{k}j_{k}}^{p_{k}-s_{k}}\ldots\partial_{i_{1}j_{1}}^{p_{1}-s_{1}}P^{2r-1}},
\end{multline}
where $s_{1},\cdots,s_{k}$ are positive odd integers such that $s_{l}\le \ell$ for every $1\le l\le k$.
Them, we have the following claim.
% The proof of Proposition \ref{prop: replacing all diagonals} implies the following claim.
\begin{claim}\label{claim: iteration}
	Let $k$ be a positive integer. Let $s_{1},\cdots,s_{k}$ be positive odd integers such that $s_{l}\le \ell$ for every $1\le l\le k$. Assume $T_{k}$ has the form \eqref{eq: g1} or \eqref{eq: g2} for $1\le l\le k-1$ and $0\le \alpha_{l} \le\frac{s_{l}-1}{2}$. Consider $\mathcal{I}_{k}\in\mathcal{J}_{k}$ (refer to \eqref{eq: def of J_k} for the definition of $\mathcal{J}_{k}$). Then, it follows that
	\begin{multline}\label{eq: est}
	\sum_{[p_{k};s_{k}]_{k}}\ldots \sum_{[p_{1};s_{1}]_{1}} \sum_{\mathcal{I}_{k}} \expec{ T_{k} } \\
	= \sum_{[p_{k};s_{k}]_{k}}\ldots \sum_{[p_{1};s_{1}]_{1}} \frac{1}{N^{k+\theta(\mathcal{I}_{k})}} \sum_{\mathcal{I}_{k}} \expec{ m^{s_{1}+\ldots+s_{k}+k}\paren{\partial_{i_{k}j_{k}}^{p_{k}-s_{k}}\ldots\partial_{i_{1}j_{1}}^{p_{1}-s_{1}}P^{2r-1}} } \\
	+ \sum_{t} \sum_{[p_{k+1};s_{k+1}]_{k+1}}\sum_{[p_{k};s_{k}]_{k}}\ldots\sum_{[p_{1};s_{1}]_{1}}\sum_{\mathcal{I}_{k+1}}\expec{T_{k+1}^{(t)}} + \bigO_{\prec}(\Phi_{r}),
	\end{multline}
	where $s_{k+1}\equiv 1 (\Mod 2)$, $0\le s_{k+1}\le \ell$, $\mathcal{I}_{k+1}\in\mathcal{J}_{k+1}$, the summation $\sum_{t}$ has a finite length which depends on $\ell$ but not $N$, the coefficient $c_{t}$ is bounded for every $t$, and $T_{k+1}^{(t)}$ has the form \eqref{eq: g1} or \eqref{eq: g2} (replacing $k$ with $k+1$) for $1\le l\le k$ and $0\le \alpha_{l} \le\frac{s_{l}-1}{2}$. (Refer to \eqref{eq: def of [p;s]} for the definition of $\sum_{[p;s]_{k}}$.)
	
	In addition, we have
	\begin{multline}\label{eq: est2}
	\size{ \sum_{t} \sum_{[p_{k+1};s_{k+1}]_{k+1}}\sum_{[p_{k};s_{k}]_{k}}\ldots\sum_{[p_{1};s_{1}]_{1}}\sum_{\mathcal{I}_{k+1}}\expec{T_{k+1}^{(t)}} } \\
	\prec \frac{1}{q^{k}} \max_{s\ge 0}\expec{\paren{\frac{|\partial_{2}P|\Im m(\tz)}{N\eta}+\paren{\frac{\Im m(\tz)}{N\eta}}^{2}+\frac{1}{N}}^{s}|P|^{2r-s-1} }.
	\end{multline}
\begin{proof}
	As in Proposition \ref{prop: replacing all diagonals}, the estimate \eqref{eq: est} follows from Lemma \ref{lem: replacing single diagoanl general version}. If $l\ge 2$, we can get the factor $q^{p_{l}-1}$ with $p_{l} > 1$ in the summation $\sum_{[p_{l};s_{l}]_{l}}$, which implies the other estimate \eqref{eq: est2} due to \eqref{eq: derivative P}.
\end{proof}
\end{claim}

Now we are ready to show that the main term \eqref{eq: main terms alt} can be written as a linear combination.

\begin{proposition}\label{prop: moment final}
	Suppose Assumption \ref{assump: random correction term 1} and Assumption \ref{assump: random correction term 2} hold. Consider the term \eqref{eq: main terms alt}:
	\begin{align*}
	\sum_{p=1}^{\ell}\frac{\mathcal{C}_{p+1}}{N^{2}q^{p-1}}\sum_{i,j}\expec{\deri^{p}(G_{ij}P^{2r-1})}.
	\end{align*}
	If $\ell$ is large enough, then the term \eqref{eq: main terms alt} is a linear combination of the terms of the following form (with bounded coefficients):
	\begin{equation*}
	\sum_{[p_{k};s_{k}]_{k}}\ldots \sum_{[p_{1};s_{1}]_{1}} \sum_{\mathcal{I}_{k}} \expec{M} + \bigO_{\prec}(\Phi_{r}), \quad 1\le k\le \ell,
	\end{equation*}
	where for $1\le l\le k$, each $s_{l}$ is a positive odd integer such that $s_{l}\le\ell$, we have $\mathcal{I}_{k}\in\mathcal{J}_{k}$ (refer to \eqref{eq: def of J_k} for the definition of $\mathcal{J}_{k}$), and $M$ has the form
	\begin{equation*}
	M = \frac{1}{N^{k+\theta(\mathcal{I}_{k})}} m^{s_{1}+\cdots+s_{k}+k}\paren{\partial_{i_{k}j_{k}}^{p_{k}-s_{k}}\ldots\partial_{i_{1}j_{1}}^{p_{1}-s_{1}}P^{2r-1}}.
	\end{equation*}
	(Refer to \eqref{eq: def of [p;s]} for the definition of $\sum_{[p;s]_{k}}$.)\\
	In fact, the length of the linear combination can be bounded by a large constant only depending on $\ell$ but not $N$.
\begin{proof}
	By Proposition \ref{prop: moment estimates}, the term \eqref{eq: main terms alt} is a linear combination of the terms of the following form:
	\begin{equation*}
	\sum_{[p_{1};s_{1}]_{1}}\sum_{\mathcal{I}_{1}} \expec{ T_{1} } + \bigO_{\prec}(\Phi_{r}),
	\end{equation*}
	where $s_{1}\equiv 1 (\Mod 2)$, $0\le s_{1}\le \ell$, $\mathcal{I}_{1}\in\mathcal{J}_{1}$ and $T_{1}$ has the form
	\begin{equation*}
	T_{1} = \frac{1}{N^{1+\theta(\mathcal{I}_{1})}}G_{i_{1}i_{1}}^{\frac{s_{1}+1}{2}}G_{j_{1}j_{1}}^{\frac{s_{1}+1}{2}}(\partial_{i_{1}j_{1}}^{p_{1}-s_{1}}P^{2r-1}).
	\end{equation*}
	The length of the linear combination can be bounded by $\ell$.
	
	Applying Claim \ref{claim: iteration}, we replace all diagonal entries in $T_{1}$ and observe that the difference,
	\begin{equation*}
	\sum_{[p_{1};s_{1}]_{1}}\sum_{\mathcal{I}_{1}} \expec{ T_{1} } - \sum_{[p_{1};s_{1}]_{1}}\frac{1}{N^{1+\theta(\mathcal{I}_{1})}}\sum_{\mathcal{I}_{1}}
	\expec{m^{s_{1}+1}(\partial_{i_{1}j_{1}}^{p_{1}-s_{1}}P^{2r-1})},
	\end{equation*}
	is a linear combination of the terms of the form
	\begin{equation}
	\sum_{[p_{2};s_{2}]_{2}}\sum_{[p_{1};s_{1}]_{1}}\sum_{\mathcal{I}_{2}} \expec{T_{2}} + \bigO_{\prec}(\Phi_{r}),
	\end{equation}
	where $s_{2}\equiv 1 (\Mod 2)$, $0\le s_{2}\le \ell$, $\mathcal{I}_{2}\in\mathcal{J}_{2}$, and $T_{2}$ has the form \eqref{eq: g1} or \eqref{eq: g2} with $k=2$. One important thing is that
	\begin{equation*}
	\size{ \sum_{[p_{2};s_{2}]_{2}}\sum_{[p_{1};s_{1}]_{1}}\sum_{\mathcal{I}_{2}} \expec{T_{2}}  } \prec \frac{1}{q} \max_{s\ge 0}\expec{\paren{\frac{|\partial_{2}P|\Im m(\tz)}{N\eta}+\paren{\frac{\Im m(\tz)}{N\eta}}^{2}+\frac{1}{N}}^{s}|P|^{2r-s-1} }.
	\end{equation*}
	% because we have the factor $q^{p_{2}-1}$ with $p_{2} > 1$ in the summation $\sum_{[p_{2};s_{2}]_{2}}$.
	
	Using Claim \ref{claim: iteration} again, we replace all diagonal entries in $T_{2}$ and observe that the difference,
	\begin{equation*}
	\sum_{[p_{2};s_{2}]_{2}}\sum_{[p_{1};s_{1}]_{1}}\sum_{\mathcal{I}_{2}} \expec{T_{2}} - \sum_{[p_{2};s_{2}]_{2}}\sum_{[p_{1};s_{1}]_{1}}\frac{1}{N^{2+\theta(\mathcal{I}_{2})}}\sum_{\mathcal{I}_{1}}
	\expec{m^{s_{1}+s_{2}+2}(\partial_{i_{2}j_{2}}^{p_{2}-s_{2}}\partial_{i_{1}j_{1}}^{p_{1}-s_{1}}P^{2r-1})},
	\end{equation*}
	is a linear combination of the terms of the form
	\begin{equation}
	\sum_{[p_{3};s_{3}]_{3}}\sum_{[p_{2};s_{2}]_{2}}\sum_{[p_{1};s_{1}]_{1}}\sum_{\mathcal{I}_{3}} \expec{T_{3}} + \bigO_{\prec}(\Phi_{r}),
	\end{equation}
	where $s_{3}\equiv 1 (\Mod 2)$, $0\le s_{3}\le \ell$, $\mathcal{I}_{3}\in\mathcal{J}_{3}$, and $T_{3}$ has the form \eqref{eq: g1} or \eqref{eq: g2} with $k=3$. Note that
	\begin{equation*}
	\size{ \sum_{[p_{3};s_{3}]_{3}}\sum_{[p_{2};s_{2}]_{2}}\sum_{[p_{1};s_{1}]_{1}}\sum_{\mathcal{I}_{3}} \expec{T_{3}} } \prec \frac{1}{q^{2}} \max_{s\ge 0}\expec{\paren{\frac{|\partial_{2}P|\Im m(\tz)}{N\eta}+\paren{\frac{\Im m(\tz)}{N\eta}}^{2}+\frac{1}{N}}^{s}|P|^{2r-s-1} }.
	\end{equation*} 
	
	In this way, we use Claim \ref{claim: iteration} to replace all diagonal entries in $T_{k}$ and show that the difference,
	\begin{equation*}
	\sum_{[p_{k};s_{k}]_{k}}\ldots \sum_{[p_{1};s_{1}]_{1}} \sum_{\mathcal{I}_{k}} \expec{ T_{k} }
	- \sum_{[p_{k};s_{k}]_{k}}\ldots \sum_{[p_{1};s_{1}]_{1}} \frac{1}{N^{k+\theta(\mathcal{I}_{k})}} \sum_{\mathcal{I}_{k}} \expec{ m^{s_{1}+\ldots+s_{k}+k}\paren{\partial_{i_{k}j_{k}}^{p_{k}-s_{k}}\ldots\partial_{i_{1}j_{1}}^{p_{1}-s_{1}}P^{2r-1}} },
	\end{equation*}
	is a linear combination of the terms of the form
	\begin{equation}
	\sum_{[p_{k+1};s_{k+1}]_{k+1}}\ldots \sum_{[p_{1};s_{1}]_{1}}\sum_{\mathcal{I}_{k+1}} \expec{T_{k+1}} + \bigO_{\prec}(\Phi_{r}),
	\end{equation}
	where $s_{k+1}\equiv 1 (\Mod 2)$, $0\le s_{k+1}\le \ell$, $\mathcal{I}_{k+1}\in\mathcal{J}_{k+1}$, and $T_{k+1}$ has the form \eqref{eq: g1} or \eqref{eq: g2} (replacing $k$ with $k+1$).
	
	After $\ell$-iterations of this argument using Claim \ref{claim: iteration}, the term \eqref{eq: main terms alt} can be approximated by a linear combination of terms of the form
	\begin{equation*}
	\sum_{[p_{k};s_{k}]_{k}}\ldots \sum_{[p_{1};s_{1}]_{1}} \frac{1}{N^{k+\theta(\mathcal{I}_{k})}} \sum_{\mathcal{I}_{k}} \expec{ m^{s_{1}+\ldots+s_{k}+k}\paren{\partial_{i_{k}j_{k}}^{p_{k}-s_{k}}\ldots\partial_{i_{1}j_{1}}^{p_{1}-s_{1}}P^{2r-1}} },\quad 1\le k\le \ell,
	\end{equation*}
	and the error is also a linear combination of terms of the form
	\begin{equation*}
	\sum_{[p_{\ell+1};s_{\ell+1}]_{\ell+1}}\ldots \sum_{[p_{1};s_{1}]_{1}}\sum_{\mathcal{I}_{\ell+1}} \expec{T_{\ell+1}} + \bigO_{\prec}(\Phi_{r}),
	\end{equation*}
	where $s_{\ell+1}\equiv 1 (\Mod 2)$, $0\le s_{\ell+1}\le \ell$, $\mathcal{I}_{\ell+1}\in\mathcal{J}_{\ell+1}$, and $T_{\ell+1}$ has the form \eqref{eq: g1} or \eqref{eq: g2} with $k=\ell+1$.
	Since $\ell$ is large enough so that we have
	\begin{equation*}
	\size{\sum_{[p_{\ell+1};s_{\ell+1}]_{\ell+1}}\ldots \sum_{[p_{1};s_{1}]_{1}}\sum_{\mathcal{I}_{\ell+1}}} \prec \frac{1}{q^{\ell}} \max_{s\ge 0}\expec{\paren{\frac{|\partial_{2}P|\Im m(\tz)}{N\eta}+\paren{\frac{\Im m(\tz)}{N\eta}}^{2}+\frac{1}{N}}^{s}|P|^{2r-s-1} } = \bigO_{\prec}(\Phi_{r}),
	\end{equation*}
	the desired result follows.
\end{proof}
\end{proposition}

\subsection{Step 5.~Random correction terms}\label{subsec: step5}

In this subsection, our first goal is to show
\begin{multline}\label{eq: correction term mathing}
\sum_{[p_{k};s_{k}]_{k}}\ldots \sum_{[p_{1};s_{1}]_{1}} \frac{1}{N^{k+\theta(\mathcal{I}_{k})}} \sum_{\mathcal{I}_{k}} \expec{ m^{s_{1}+\cdots+s_{k}+k}\paren{\partial_{i_{k}j_{k}}^{p_{k}-s_{k}}\ldots\partial_{i_{1}j_{1}}^{p_{1}-s_{1}}P^{2r-1}} } \\
= \expec{\frac{1}{N^{\theta(\mathcal{I}_{k})}}
	\paren{\sum_{\mathcal{I}_{k}}\prod_{l=1}^{k}A_{l}(s_{l})} m^{s_{1}+\cdots+s_{k}+k} P^{2r-1} } + \bigO_{\prec}(\Phi_{r})
\end{multline}
where $s_{1},s_{2},\cdots, s_{k}$ are positive odd integers, %such that $s_{1}+\cdots+s_{k}+k= 2n$
% which implies $k\le m$.
$\mathcal{I}_{k}\in\mathcal{J}_{k}$, and the term $A_{l}(s_{l})$ is defined through
\begin{align}\label{eq: building block}
A_{l}(s_{l}) \coloneqq \begin{cases}
h_{i_{l}j_{l}}^{2} - \expec{h_{i_{l}j_{l}}^{2}} & \text{$l>1$ and $s_{l}=1$}, \\
h_{i_{l}j_{l}}^{s_{l}+1} & \text{otherwise}.
\end{cases}
\end{align}
Then, the term
\begin{equation*}
\frac{1}{N^{\theta(\mathcal{I}_{k})}}
\paren{\sum_{\mathcal{I}_{k}}\prod_{l=1}^{k}A_{l}(s_{l})},
\end{equation*}
will be a building block of the random correction term $\mathcal{Z}_{n}$ (as in Theorem \ref{thm: local law}) when $s_{1}+\cdots+s_{k}+k= 2n$.

Firstly, we consider the case $k=1$.

\begin{claim}\label{claim: correction term mathing 1}
	Let $s_{1}$ be a positive odd integer. Consider $\mathcal{I}_{1}\in\mathcal{J}_{1}$ (refer to \eqref{eq: def of J_k} for the definition of $\mathcal{J}_{k}$). Then, we have
	\begin{align*}
	\sum_{[p_{1};s_{1}]_{1}} \frac{1}{N^{1+\theta(\mathcal{I}_{1})}} \sum_{\mathcal{I}_{1}}\expec{m^{s_{1}+1}(\partial_{i_{1}j_{1}}^{p_{1}-s_{1}}P^{2r-1})} 
	= \expec{ \frac{1}{N^{\theta(\mathcal{I}_{1})}}\paren{\sum_{\mathcal{I}_{1}}h_{i_{1}j_{1}}^{s_{1}+1}}m^{s_{1}+1}P^{2r-1} } + \bigO_{\prec}\paren{ \Phi_{r} }.
	\end{align*}
	(Refer to \eqref{eq: def of [p;s]} for the definition of $\sum_{[p;s]_{k}}$.)
\begin{proof}
	Consider the first term of the right-hand side
	\begin{equation*}
	\frac{1}{N^{\theta(\mathcal{I}_{1})}}\sum_{\mathcal{I}_{1}}\expec{ h_{i_{1}j_{1}}^{s_{1}+1}m^{s_{1}+1}P^{2r-1} }.
	\end{equation*}
	Referring to \eqref{eq: def of theta} (for the definition of $\theta(\mathcal{I}_{k})$), we note that $\theta(\mathcal{I}_{1})=1$. Using the cumulant expansion, we have
	\begin{equation*}
	\frac{1}{N}\sum_{i_{1}\neq j_{1}}\sum_{p_{1} = 1}^{\ell}\frac{\mathcal{C}_{p_{1}+1}}{Nq^{p_{1}-1}}\expec{ \partial_{i_{1}j_{1}}^{p_{1}}(h_{i_{1}j_{1}}^{s_{1}}m^{s_{1}+1}P^{2r-1}) } + \bigO_{\prec}\paren{ \Phi_{r} }.
	\end{equation*}
	One remark here is that, if $\partial_{i_{1}j_{1}}$ does not hit $h_{i_{1}j_{1}}$ exactly $s_{1}$ times, then we can apply the cumulant expansion again with a remaining $h_{i_{1}j_{1}}$ so that we get an additional factor $1/N$ and the resulting terms are bounded by $\bigO_{\prec}\paren{ \Phi_{r} }$. For example, we can think of the following case:
	\begin{multline*}
	\frac{1}{N}\sum_{i_{1}\neq j_{1}}\sum_{p_{1} = 1}^{\ell}\frac{\mathcal{C}_{p_{1}+1}}{Nq^{p_{1}-1}}\expec{ h_{i_{1}j_{1}}^{s_{1}}\partial_{i_{1}j_{1}}^{p_{1}}(m^{s_{1}+1}P^{2r-1}) } \\
	= \frac{1}{N}\sum_{i_{1}\neq j_{1}}\sum_{p_{1} = 1}^{\ell}\frac{\mathcal{C}_{p_{1}+1}}{Nq^{p_{1}-1}}\sum_{p'=1}^{\ell}\frac{\mathcal{C}_{p'+1}}{Nq^{p'-1}}\expec{ \partial_{i_{1}j_{1}}^{p'}\big(h_{i_{1}j_{1}}^{s_{1}-1}\partial_{i_{1}j_{1}}^{p_{1}}(m^{s_{1}+1}P^{2r-1})\big) } + \bigO_{\prec}\paren{ \Phi_{r} } = \bigO_{\prec}\paren{ \Phi_{r} }.
	\end{multline*}
	Thus, in order to get a non-negligible contribution, the derivative $\partial_{i_{1}j_{1}}$ must hit $h_{i_{1}j_{1}}$ exactly $s_{1}$ times so we have
	\begin{equation*}
	\frac{1}{N}\sum_{i_{1}\neq j_{1}}\expec{ h_{i_{1}j_{1}}^{s_{1}+1}m^{s_{1}+1}P^{2r-1} } = \sum_{p_{1} = s_{1}}^{\ell}{p_{1} \choose s_{1}}\frac{(s_{1}!)\mathcal{C}_{p_{1}+1}}{N^{2}q^{p_{1}-1}}\sum_{i_{1}\neq j_{1}}\expec{ \partial_{i_{1}j_{1}}^{p_{1}-s_{1}}(m^{s_{1}+1}P^{2r-1}) } + \bigO_{\prec}\paren{ \Phi_{r} }.
	\end{equation*}
	If $\deri$ hits the normalized trace $m$ at least once, then the resulting terms are absorbed into $\bigO_{\prec}\paren{ \Phi_{r} }$ because of \eqref{eq: derivative m}. As a result, we get
	\begin{multline*}
	\sum_{p_{1} = s_{1}}^{\ell}{p_{1} \choose s_{1}}\frac{(s_{1}!)\mathcal{C}_{p_{1}+1}}{N^{2}q^{p_{1}-1}}\sum_{i_{1}\neq j_{1}}\expec{ \partial_{i_{1}j_{1}}^{p_{1}-s_{1}}(m^{s_{1}+1}P^{2r-1}) } \\
	= \sum_{p_{1} = s_{1}}^{\ell}{p_{1} \choose s_{1}}\frac{(s_{1}!)\mathcal{C}_{p_{1}+1}}{N^{2}q^{p_{1}-1}}\sum_{i_{1}\neq j_{1}}\expec{ m^{s_{1}+1}\partial_{i_{1}j_{1}}^{p_{1}-s_{1}}(P^{2r-1}) } + \bigO_{\prec}\paren{ \Phi_{r} } \\
	= \sum_{[p_{1};s_{1}]_{1}} \frac{1}{N^{1+\theta(\mathcal{I}_{1})}} \sum_{\mathcal{I}_{1}}\expec{m^{s_{1}+1}(\partial_{i_{1}j_{1}}^{p_{1}-s_{1}}P^{2r-1})} + \bigO_{\prec}\paren{ \Phi_{r} },
	\end{multline*}
	and obtain the claim. (Refer to \eqref{eq: def of [p;s]} for the definition of $\sum_{[p;s]_{k}}$.)
\end{proof}
\end{claim}

Next, we consider the case $k=2$.
\begin{claim}\label{claim: correction term mathing 2}
	Let $s_{1}$ and $s_{2}$ be positive odd integers. Consider $\mathcal{I}_{2}\in\mathcal{J}_{2}$ (refer to \eqref{eq: def of J_k} for the definition of $\mathcal{J}_{k}$). Let $A_{l}$ be as in \eqref{eq: building block} for each $l=1,2$. Then, we have
	\begin{multline*}
	\sum_{[p_{2};s_{2}]_{2}}\sum_{[p_{1};s_{1}]_{1}} \frac{1}{N^{2+\theta(\mathcal{I}_{2})}} \sum_{\mathcal{I}_{2}}\expec{m^{s_{1}+s_{2}+2}(\partial_{i_{2}j_{2}}^{p_{2}-s_{2}}\partial_{i_{1}j_{1}}^{p_{1}-s_{1}}P^{2r-1})} \\
	= \expec{\frac{1}{N^{\theta(\mathcal{I}_{2})}}
		\paren{\sum_{\mathcal{I}_{2}}A_{1}(s_{1})A_{2}(s_{2})} m^{s_{1}+s_{2}+2}P^{2r-1} } + \bigO_{\prec}\paren{ \Phi_{r} }.
	\end{multline*}
	(Refer to \eqref{eq: def of [p;s]} for the definition of $\sum_{[p;s]_{k}}$.)
\begin{proof}
	If $s_{2}=1$, the first term of the right-hand side is,
	\begin{equation*}
        \frac{1}{N^{\theta(\mathcal{I}_{2})}}
		\sum_{\mathcal{I}_{2}} \expec{ h_{i_{1}j_{1}}^{s_{1}+1} (h_{i_{2}j_{2}}^{2}-\expec{h_{i_{2}j_{2}}^{2}}) m^{s_{1}+s_{2}+2}P^{2r-1} }.
	\end{equation*}
	Using the cumulant expansion with $h_{i_{1}j_{1}}$ and following the proof of Claim \ref{claim: correction term mathing 1}, we have
	\begin{multline*}
	\frac{1}{N^{\theta(\mathcal{I}_{2})}}
	\sum_{\mathcal{I}_{2}} \expec{ h_{i_{1}j_{1}}^{s_{1}+1} (h_{i_{2}j_{2}}^{2}-\expec{h_{i_{2}j_{2}}^{2}}) m^{s_{1}+s_{2}+2}P^{2r-1} } \\
	= \sum_{[p_{1};s_{1}]_{1}}\frac{1}{N^{1+\theta(\mathcal{I}_{2})}}\sum_{\mathcal{I}_{2}}\expec{   (h_{i_{2}j_{2}}^{2}-\expec{h_{i_{2}j_{2}}^{2}}) m^{s_{1}+s_{2}+2} ( \partial_{i_{1}j_{1}}^{p_{1}-s_{1}} P^{2r-1} ) } + \bigO_{\prec}\paren{ \Phi_{r} },
	\end{multline*}
	where we use that $\partial_{i_{1}j_{1}}(h_{i_{2}j_{2}})=0$ on the summation over distinct indexes $\sum_{\mathcal{I}_{2}}$ (refer to \eqref{eq: sum notation} for the definition of $\sum_{\mathcal{I}_{k}}$).
	
	Applying the cumulant expansion with $h_{i_{2}j_{2}}$, we get
	\begin{multline*}
	\sum_{[p_{1};s_{1}]_{1}}\frac{1}{N^{1+\theta(\mathcal{I}_{2})}}\sum_{\mathcal{I}_{2}}\expec{   (h_{i_{2}j_{2}}^{2}-\expec{h_{i_{2}j_{2}}^{2}}) m^{s_{1}+s_{2}+2} ( \partial_{i_{1}j_{1}}^{p_{1}-s_{1}} P^{2r-1} ) } \\
	= \sum_{[p_{1};s_{1}]_{1}}\frac{1}{N^{1+\theta(\mathcal{I}_{2})}}\sum_{\mathcal{I}_{2}}\sum_{p_{2}=1}^{\ell}\frac{\mathcal{C}_{p_{2}+1}}{Nq^{p_{2}-1}} \expec{ \partial_{i_{2}j_{2}}^{p_{2}} \big( h_{i_{2}j_{2}} m^{s_{1}+s_{2}+2} ( \partial_{i_{1}j_{1}}^{p_{1}-s_{1}} P^{2r-1} ) \big) } \\
	- \sum_{[p_{1};s_{1}]_{1}}\frac{1}{N^{1+\theta(\mathcal{I}_{2})}}\sum_{\mathcal{I}_{2}}\expec{   \expec{h_{i_{2}j_{2}}^{2}} m^{s_{1}+s_{2}+2} ( \partial_{i_{1}j_{1}}^{p_{1}-s_{1}} P^{2r-1} ) } + \bigO_{\prec}\paren{ \Phi_{r} }.
	\end{multline*}
	The derivative $\partial_{i_{2}j_{2}}$ have to hit $h_{i_{2}j_{2}}$ otherwise the resulting terms are absorbed into $\bigO_{\prec}\paren{ \Phi_{r} }$, which implies
	\begin{multline*}
	\sum_{[p_{1};s_{1}]_{1}}\frac{1}{N^{1+\theta(\mathcal{I}_{2})}}\sum_{\mathcal{I}_{2}}\sum_{p_{2}=1}^{\ell}\frac{\mathcal{C}_{p_{2}+1}}{Nq^{p_{2}-1}} \expec{ \partial_{i_{2}j_{2}}^{p_{2}} \big( h_{i_{2}j_{2}} m^{s_{1}+s_{2}+2} ( \partial_{i_{1}j_{1}}^{p_{1}-s_{1}} P^{2r-1} ) \big) } \\
	= \sum_{[p_{1};s_{1}]_{1}}\frac{1}{N^{1+\theta(\mathcal{I}_{2})}}\sum_{\mathcal{I}_{2}}\sum_{p_{2}=1}^{\ell}{p_{2} \choose 1}\frac{\mathcal{C}_{p_{2}+1}}{Nq^{p_{2}-1}} \expec{ \partial_{i_{2}j_{2}}^{p_{2}-1} \big( m^{s_{1}+s_{2}+2} ( \partial_{i_{1}j_{1}}^{p_{1}-s_{1}} P^{2r-1} ) \big) } + \bigO_{\prec}\paren{ \Phi_{r} }.
	\end{multline*}
	We find the cancellation for $p_{2}=1$ as follows:
	\begin{multline*}
	\sum_{[p_{1};s_{1}]_{1}}\frac{1}{N^{1+\theta(\mathcal{I}_{2})}}\sum_{\mathcal{I}_{2}}\sum_{p_{2}=1}^{\ell}{p_{2} \choose 1}\frac{\mathcal{C}_{p_{2}+1}}{Nq^{p_{2}-1}} \expec{ \partial_{i_{2}j_{2}}^{p_{2}-1} \big( m^{s_{1}+s_{2}+2} ( \partial_{i_{1}j_{1}}^{p_{1}-s_{1}} P^{2r-1} ) \big) } \\
	- \sum_{[p_{1};s_{1}]_{1}}\frac{1}{N^{1+\theta(\mathcal{I}_{2})}}\sum_{\mathcal{I}_{2}}\expec{   \expec{h_{i_{2}j_{2}}^{2}} m^{s_{1}+s_{2}+2} ( \partial_{i_{1}j_{1}}^{p_{1}-s_{1}} P^{2r-1} ) } \\
	= \sum_{[p_{1};s_{1}]_{1}}\frac{1}{N^{1+\theta(\mathcal{I}_{2})}}\sum_{\mathcal{I}_{2}}\sum_{p_{2}=2}^{\ell}{p_{2} \choose 1}\frac{\mathcal{C}_{p_{2}+1}}{Nq^{p_{2}-1}} \expec{ \partial_{i_{2}j_{2}}^{p_{2}-1} \big( m^{s_{1}+s_{2}+2} ( \partial_{i_{1}j_{1}}^{p_{1}-s_{1}} P^{2r-1} ) \big) },
	\end{multline*}
	where we use that $\mathcal{C}_{2}=1$ and $\expec{h_{i_{2}j_{2}}^{2}}=N^{-1}$ by Definition \ref{def: sparse RM}.
    
    If $\partial_{i_{2}j_{2}}$ hits the normalized trace $m$ at least once, then the resulting terms are absorbed into $\bigO_{\prec}\paren{ \Phi_{r} }$ due to \eqref{eq: derivative m}, which means
	\begin{multline*}
	\sum_{[p_{1};s_{1}]_{1}}\frac{1}{N^{1+\theta(\mathcal{I}_{2})}}\sum_{\mathcal{I}_{2}}\sum_{p_{2}=2}^{\ell}{p_{2} \choose 1}\frac{\mathcal{C}_{p_{2}+1}}{Nq^{p_{2}-1}} \expec{ \partial_{i_{2}j_{2}}^{p_{2}-1} \big( m^{s_{1}+s_{2}+2} ( \partial_{i_{1}j_{1}}^{p_{1}-s_{1}} P^{2r-1} ) \big) } \\
	= \sum_{[p_{1};s_{1}]_{1}}\frac{1}{N^{1+\theta(\mathcal{I}_{2})}}\sum_{\mathcal{I}_{2}}\sum_{p_{2}=2}^{\ell}{p_{2} \choose 1}\frac{\mathcal{C}_{p_{2}+1}}{Nq^{p_{2}-1}} \expec{ m^{s_{1}+s_{2}+2} (\partial_{i_{2}j_{2}}^{p_{2}-1}\partial_{i_{1}j_{1}}^{p_{1}-s_{1}} P^{2r-1} ) } + \bigO_{\prec}\paren{ \Phi_{r} } \\
	= \sum_{[p_{2};s_{2}]_{2}}\sum_{[p_{1};s_{1}]_{1}}\frac{1}{N^{2+\theta(\mathcal{I}_{2})}}\sum_{\mathcal{I}_{2}}\expec{ m^{s_{1}+s_{2}+2} (\partial_{i_{2}j_{2}}^{p_{2}-s_{2}}\partial_{i_{1}j_{1}}^{p_{1}-s_{1}} P^{2r-1} ) } + \bigO_{\prec}\paren{ \Phi_{r} }, \quad s_{2}=1,
	\end{multline*}
	which concludes the proof for the case $s_{2}=1$. (Refer to \eqref{eq: def of [p;s]} for the definition of $\sum_{[p;s]_{k}}$.)
	
	If $s_{2}>1$, we consider
	\begin{equation*}
	\frac{1}{N^{\theta(\mathcal{I}_{2})}}
	\sum_{\mathcal{I}_{2}} \expec{ h_{i_{1}j_{1}}^{s_{1}+1} h_{i_{2}j_{2}}^{s_{2}+1}  m^{s_{1}+s_{2}+2}P^{2r-1} }.
	\end{equation*}
	The proof is almost the same with the previous case. Using the cumulant expansion with $h_{i_{1}j_{1}}$ and following the proof of Claim \ref{claim: correction term mathing 1}, we get
	\begin{multline*}
	\frac{1}{N^{\theta(\mathcal{I}_{2})}}
	\sum_{\mathcal{I}_{2}} \expec{ h_{i_{1}j_{1}}^{s_{1}+1} h_{i_{2}j_{2}}^{s_{2}+1}  m^{s_{1}+s_{2}+2}P^{2r-1} } \\
	= \sum_{[p_{1};s_{1}]_{1}}\frac{1}{N^{1+\theta(\mathcal{I}_{2})}}\sum_{\mathcal{I}_{2}}\expec{   h_{i_{2}j_{2}}^{s_{2}+1} m^{s_{1}+s_{2}+2} ( \partial_{i_{1}j_{1}}^{p_{1}-s_{1}} P^{2r-1} ) } + \bigO_{\prec}\paren{ \Phi_{r} },
	\end{multline*}
	where we use that $\partial_{i_{1}j_{1}}(h_{i_{2}j_{2}})=0$ on the summation over distinct indexes $\sum_{\mathcal{I}_{2}}$.

	By the cumulant expansion with $h_{i_{2}j_{2}}$, we have
	\begin{multline*}
	\sum_{[p_{1};s_{1}]_{1}}\frac{1}{N^{1+\theta(\mathcal{I}_{2})}}\sum_{\mathcal{I}_{2}}\expec{   h_{i_{2}j_{2}}^{s_{2}+1} m^{s_{1}+s_{2}+2} ( \partial_{i_{1}j_{1}}^{p_{1}-s_{1}} P^{2r-1} ) } \\
	= \sum_{[p_{1};s_{1}]_{1}}\frac{1}{N^{1+\theta(\mathcal{I}_{2})}}\sum_{\mathcal{I}_{2}} \sum_{p_{2}=1}^{\ell}\frac{\mathcal{C}_{p_{2}+1}}{Nq^{p_{2}-1}} \expec{ \partial_{i_{2}j_{2}}^{p_{2}} \big( h_{i_{2}j_{2}}^{s_{2}} m^{s_{1}+s_{2}+2} ( \partial_{i_{1}j_{1}}^{p_{1}-s_{1}} P^{2r-1} ) \big) } + \bigO_{\prec}\paren{ \Phi_{r} }.
	\end{multline*}
	The derivative $\partial_{i_{2}j_{2}}$ have to hit $h_{i_{2}j_{2}}$ exactly $s_{2}$ times so it follows that
	\begin{multline*}
	\sum_{[p_{1};s_{1}]_{1}}\frac{1}{N^{1+\theta(\mathcal{I}_{2})}}\sum_{\mathcal{I}_{2}} \sum_{p_{2}=1}^{\ell}\frac{\mathcal{C}_{p_{2}+1}}{Nq^{p_{2}-1}} \expec{ \partial_{i_{2}j_{2}}^{p_{2}} \big( h_{i_{2}j_{2}}^{s_{2}} m^{s_{1}+s_{2}+2} ( \partial_{i_{1}j_{1}}^{p_{1}-s_{1}} P^{2r-1} ) \big) } \\
	= \sum_{[p_{1};s_{1}]_{1}}\frac{1}{N^{1+\theta(\mathcal{I}_{2})}}\sum_{\mathcal{I}_{2}} \sum_{p_{2}=s_{2}}^{\ell}{p_{2} \choose s_{2}}\frac{(s_{2}!)\mathcal{C}_{p_{2}+1}}{Nq^{p_{2}-1}} \expec{ \partial_{i_{2}j_{2}}^{p_{2}-s_{2}} \big( m^{s_{1}+s_{2}+2} ( \partial_{i_{1}j_{1}}^{p_{1}-s_{1}} P^{2r-1} ) \big) } + \bigO_{\prec}\paren{ \Phi_{r} }.
	\end{multline*}
	If $\partial_{i_{2}j_{2}}$ hits the normalized trace $m$ at least once, then the resulting terms are absorbed into $\bigO_{\prec}\paren{ \Phi_{r} }$ due to \eqref{eq: derivative m}. Therefore we obtain the desired result:
	\begin{multline*}
	\sum_{[p_{1};s_{1}]_{1}}\frac{1}{N^{1+\theta(\mathcal{I}_{2})}}\sum_{\mathcal{I}_{2}} \sum_{p_{2}=s_{2}}^{\ell} {p_{2} \choose s_{2}}\frac{(s_{2}!)\mathcal{C}_{p_{2}+1}}{Nq^{p_{2}-1}} \expec{ \partial_{i_{2}j_{2}}^{p_{2}-s_{2}} \big( m^{s_{1}+s_{2}+2} ( \partial_{i_{1}j_{1}}^{p_{1}-s_{1}} P^{2r-1} ) \big) } \\
	= \sum_{[p_{1};s_{1}]_{1}}\frac{1}{N^{1+\theta(\mathcal{I}_{2})}}\sum_{\mathcal{I}_{2}} \sum_{p_{2}=s_{2}}^{\ell}{p_{2} \choose s_{2}}\frac{(s_{2}!)\mathcal{C}_{p_{2}+1}}{Nq^{p_{2}-1}} \expec{ m^{s_{1}+s_{2}+2} ( \partial_{i_{2}j_{2}}^{p_{2}-s_{2}}\partial_{i_{1}j_{1}}^{p_{1}-s_{1}} P^{2r-1} ) } + \bigO_{\prec}\paren{ \Phi_{r} } \\
	= \sum_{[p_{2};s_{2}]_{2}}\sum_{[p_{1};s_{1}]_{1}}\frac{1}{N^{2+\theta(\mathcal{I}_{2})}}\sum_{\mathcal{I}_{2}}\expec{ m^{s_{1}+s_{2}+2} (\partial_{i_{2}j_{2}}^{p_{2}-s_{2}}\partial_{i_{1}j_{1}}^{p_{1}-s_{1}} P^{2r-1} ) } + \bigO_{\prec}\paren{ \Phi_{r} }, \quad s_{2}>1.
	\end{multline*}
\end{proof}
\end{claim}

Let us consider a general case.

\begin{proposition}\label{prop: correction term mathing}
	Suppose Assumption \ref{assump: random correction term 1} and Assumption \ref{assump: random correction term 2} hold. Let $s_{1},s_{2},\cdots, s_{k}$ are positive odd integers. Consider $\mathcal{I}_{k}\in\mathcal{J}_{k}$ (refer to \eqref{eq: def of J_k} for the definition of $\mathcal{J}_{k}$). Let $A_{l}$ be as in \eqref{eq: building block} for each $l=1,2,\cdots,k$. Then, \eqref{eq: correction term mathing} holds. (Refer to \eqref{eq: def of [p;s]} for the definition of $\sum_{[p;s]_{k}}$.)
\begin{proof}
	Since the case $k=1$ or $k=2$ is already proved in Claim \ref{claim: correction term mathing 1} and Claim \ref{claim: correction term mathing 2}, it is enough to consider the case $k\ge 3$. We first do the cumulant expansion with $h_{i_{1}j_{1}}$. Following the proof of Claim \ref{claim: correction term mathing 1}, we get 
	\begin{multline*}
	\frac{1}{N^{\theta(\mathcal{I}_{k})}}
	\sum_{\mathcal{I}_{k}} \expec{ A_{1}(s_{1})A_{2}(s_{2})\cdots A_{k}(s_{k}) m^{s_{1}+\cdots+s_{k}+k} P^{2r-1} } \\
	= \sum_{[p_{1};s_{1}]_{1}}\frac{1}{N^{1+\theta(\mathcal{I}_{k})}}\sum_{\mathcal{I}_{k}}\expec{   A_{2}(s_{2})A_{3}(s_{3})\cdots A_{k}(s_{k}) m^{s_{1}+\cdots+s_{k}+k} ( \partial_{i_{1}j_{1}}^{p_{1}-s_{1}} P^{2r-1} ) } + \bigO_{\prec}\paren{ \Phi_{r} },
	\end{multline*}
	where we use that $\partial_{i_{1}j_{1}}(h_{i_{l}j_{l}})=0$ for any $1<l\le k$ on the summation over distinct indexes $\sum_{\mathcal{I}_{k}}$ (refer to \eqref{eq: sum notation} for the definition of $\sum_{\mathcal{I}_{k}}$).
	
	Next we use the cumulant expansion with $h_{i_{2}j_{2}}$ and follow the proof of Claim \ref{claim: correction term mathing 2}. Since $\partial_{i_{2}j_{2}}(h_{i_{l}j_{l}})=0$ for any $2<l\le k$ on the summation over distinct indexes $\sum_{\mathcal{I}_{k}}$, we observe that
	\begin{multline*}
	\sum_{[p_{1};s_{1}]_{1}}\frac{1}{N^{1+\theta(\mathcal{I}_{k})}}\sum_{\mathcal{I}_{k}}\expec{   A_{2}(s_{2})\cdots A_{k}(s_{k}) m^{s_{1}+\cdots+s_{k}+k} ( \partial_{i_{1}j_{1}}^{p_{1}-s_{1}} P^{2r-1} ) } \\
	= \sum_{[p_{2};s_{2}]_{2}}\sum_{[p_{1};s_{1}]_{1}}\frac{1}{N^{2+\theta(\mathcal{I}_{k})}}\sum_{\mathcal{I}_{k}}
	\expec{   A_{3}(s_{3})\cdots A_{k}(s_{k}) m^{s_{1}+\cdots+s_{k}+k} ( \partial_{i_{1}j_{1}}^{p_{1}-s_{1}} P^{2r-1} ) } + \bigO_{\prec}\paren{ \Phi_{r} }.
	\end{multline*}
	
	Repeating the cumulant expansion and following the proof of Claim \ref{claim: correction term mathing 2}, we have
	\begin{multline*}
	\sum_{[p_{2};s_{2}]_{2}}\sum_{[p_{1};s_{1}]_{1}}\frac{1}{N^{2+\theta(\mathcal{I}_{k})}}\sum_{\mathcal{I}_{k}}
	\expec{   A_{3}(s_{3})\cdots A_{k}(s_{k}) m^{s_{1}+\cdots+s_{k}+k} ( \partial_{i_{1}j_{1}}^{p_{1}-s_{1}} P^{2r-1} ) } \\
	= \sum_{[p_{k};s_{k}]_{k}}\ldots \sum_{[p_{1};s_{1}]_{1}} \frac{1}{N^{k+\theta(\mathcal{I}_{k})}} \sum_{\mathcal{I}_{k}} \expec{ m^{s_{1}+\cdots+s_{k}+k}\paren{\partial_{i_{k}j_{k}}^{p_{k}-s_{k}}\ldots\partial_{i_{1}j_{1}}^{p_{1}-s_{1}}P^{2r-1}} } + \bigO_{\prec}\paren{ \Phi_{r} },
	\end{multline*}
	where we use that $\partial_{i_{l}j_{l}}(h_{i_{l'}j_{l'}})=0$ for any $l\neq l'$ on the summation over distinct indexes $\sum_{\mathcal{I}_{k}}$. The desired result follows.
\end{proof}
\end{proposition}

Therefore, the random correction term $\mathcal{Z}_{n}$ should be a linear combination of terms of the form
\begin{align*}
\frac{1}{N^{\theta}}\sum_{\mathcal{I}_{k}}\prod_{l=1}^{k}A_{l}(s_{l}),
\end{align*}
where $s_{1},s_{2},\cdots, s_{k}$ are positive odd integers such that $s_{1}+\cdots+s_{k}+k= 2n$.

\subsection{Proof of Proposition \ref{prop: RME}}\label{subsec: RME proof}

\begin{proof}[Proof of Proposition \ref{prop: RME}]
	Recall \eqref{eq: for RME 1} and \eqref{eq: for RME 2}. Then, we have to construct $Q(m)=\sum_{n=1}^{\ell}\mathcal{Z}_{n}m^{2n}$ so that
	\begin{equation}\label{eq: r1}
	\sum_{p=1}^{\ell}\frac{\mathcal{C}_{p+1}}{N^{2}q^{p-1}}\sum_{i,j}\expec{\deri^{p}(G_{ij}P^{r-1}\bar{P}^{r})} + \E\left[ Q(m)P^{r-1}\bar{P}^{r} \right] = \bigO_{\prec}\paren{ \Phi_{r} }.
	\end{equation}
	For brevity, instead of \eqref{eq: r1}, we shall show
	\begin{equation}\label{eq: r2}
	\sum_{p=1}^{\ell}\frac{\mathcal{C}_{p+1}}{N^{2}q^{p-1}}\sum_{i,j}\expec{\deri^{p}(G_{ij}P^{2r-1})} + \E\left[ Q(m)P^{2r-1} \right] = \bigO_{\prec}\paren{ \Phi_{r} }.
	\end{equation}
	
	Assume that a set of random coefficients $\{\mathcal{Z}_{n}\}_{n=1}^{\ell}$ satisfies Assumption \ref{assump: random correction term 1} and Assumption \ref{assump: random correction term 2}. Consider the first term of the left-hand side of \eqref{eq: r2}:
	\begin{equation}\label{eq: r3}
	\sum_{p=1}^{\ell}\frac{\mathcal{C}_{p+1}}{N^{2}q^{p-1}}\sum_{i,j}\expec{\deri^{p}(G_{ij}P^{2r-1})}.
	\end{equation}
	
	By Proposition \ref{prop: moment final}, the term \eqref{eq: r3} is a linear combination of the terms of the following form (with bounded coefficients):
	\begin{equation*}
	\sum_{[p_{k};s_{k}]_{k}}\ldots \sum_{[p_{1};s_{1}]_{1}} \frac{1}{N^{k+\theta(\mathcal{I}_{k})}} \sum_{\mathcal{I}_{k}} \expec{  m^{s_{1}+\ldots+s_{k}+k}\paren{\partial_{i_{k}j_{k}}^{p_{k}-s_{k}}\ldots\partial_{i_{1}j_{1}}^{p_{1}-s_{1}}P^{2r-1}} } + \bigO_{\prec}(\Phi_{r}), \quad 1\le k\le \ell,
	\end{equation*}
	where for $1\le l\le k$, each $s_{l}$ is a positive odd integer such that $s_{l}\le\ell$, and $\mathcal{I}_{k}\in\mathcal{J}_{k}$. Note that the length of the linear combination is bounded by a large constant only depending on $\ell$ but not $N$.
	
	In addition, due to Proposition \ref{prop: correction term mathing}, the term \eqref{eq: r3} is also a linear combination of the terms of the following form (with bounded coefficients):
	\begin{equation*}
	\expec{\frac{1}{N^{\theta(\mathcal{I}_{k})}}
		\paren{\sum_{\mathcal{I}_{k}}\prod_{l=1}^{k}A_{l}(s_{l})} m^{s_{1}+\cdots+s_{k}+k} P^{2r-1} } + \bigO_{\prec}(\Phi_{r}), \quad 1\le k\le \ell,
	\end{equation*}
	where for $1\le l\le k$, each $s_{l}$ is a positive odd integer such that $s_{l}\le\ell$, $\mathcal{I}_{k}\in\mathcal{J}_{k}$, and $A_{l}$ has the form \eqref{eq: building block} for each $l=1,2,\cdots,k$.
	
	Let us consider the form
	\begin{equation}\label{eq: r4}
	\expec{\frac{1}{N^{\theta(\mathcal{I}_{k})}}
		\paren{\sum_{\mathcal{I}_{k}}\prod_{l=1}^{k}A_{l}(s_{l})} m^{s_{1}+\cdots+s_{k}+k} P^{2r-1} }, \quad 1\le k\le \ell,
	\end{equation}
	where for $1\le l\le k$, each $s_{l}$ is a positive odd integer such that $s_{l}\le\ell$, $\mathcal{I}_{k}\in\mathcal{J}_{k}$. Formally we can write
	\begin{equation*}
	\sum_{p=1}^{\ell}\frac{\mathcal{C}_{p+1}}{N^{2}q^{p-1}}\sum_{i,j}\expec{\deri^{p}(G_{ij}P^{2r-1})}
	= \big(\text{linear combination of terms of the form \eqref{eq: r4}}\big) +  \bigO_{\prec}(\Phi_{r}).
	\end{equation*}
	Then, we can construct $\{\mathcal{Z}_{n}\}_{n=1}^{\ell}$ such that
	\begin{equation*}
	\big(\text{linear combination of terms of the form \eqref{eq: r4}}\big) + \E\left[ Q(m)P^{2r-1} \right] = 0,
	\end{equation*}
	by setting $\mathcal{Z}_{n}$ as a linear combination of terms of the form
	\begin{align}\label{eq: form of correction terms}
	\frac{1}{N^{\theta(\mathcal{I}_{k})}}\sum_{\mathcal{I}_{k}}\prod_{l=1}^{k}A_{l}(s_{l}),
	\end{align}
	where $s_{1},s_{2},\cdots, s_{k}$ are positive odd integers such that $s_{1}+\cdots+s_{k}+k= 2n$. By the construction, the estimate \eqref{eq: r2} holds. Moreover, Assumption \ref{assump: random correction term 1} and Assumption \ref{assump: random correction term 2} hold for the constructed $\{\mathcal{Z}_{n}\}_{n=1}^{\ell}$ according to Appendix \ref{sec: correction term calculation}. Therefore we complete the proof.
\end{proof}

\begin{remark}
	\begin{figure}[h]\label{fig: bipartite subgraph}
		\includegraphics[scale=1]{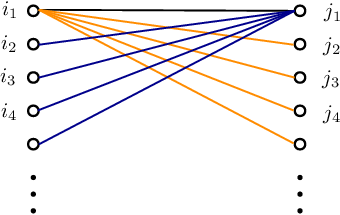}
		\caption{Interpretation of random correction terms on the graph}
	\end{figure}
    The term \eqref{eq: form of correction terms} can be associated with the number of bipartite subgraphs of a certain form; see Figure \ref{fig: bipartite subgraph}.
    In \cite{HY22+}, Huang and Yau describe the term \eqref{eq: form of correction terms} using a weighted forest. We refer to \cite[Definition 2.2, Eq (2.7)]{HY22+} for more detail.
\end{remark}

\section{Local law and edge rigidity}\label{sec: proof of local law and rigidity} 
\begin{proof}[Proof of Theorem \ref{thm: local law}]
	We start with a standard argument using stability analysis (e.g.~\cite[Lemma 6.1]{BL21+}). Let us set
	\begin{align*}
		g(\tz) = m(\tz) - \widetilde{m}(\tz), \quad \Lambda(\tz)=|g(\tz)|.
	\end{align*}
	From Taylor expansion, it follows that
	\begin{align*}
		P(\tz,m(\tz)) = \partial_{2}P(\tz,\widetilde{m}(\tz))g(\tz) + \frac{1}{2}\partial_{2}^{2}P(\tz,\widetilde{m}(\tz))g^{2}(\tz) + R(g(\tz)),
	\end{align*}
	where $R(y)=b_{1}y^{3}+b_{2}y^{4}+\cdots$ is a polynomial whose coefficients are random and stochastically dominated by $q^{-1}$. We write
	\begin{align*}
		f(\tz) = P(\tz,m(\tz)) - R(g(\tz)), \quad b(\tz) = \partial_2 P(\tz,\widetilde{m}(\tz)), \quad a(\tz) = \partial_{2}^{2}P(\tz,\widetilde{m}(\tz)).
	\end{align*}
	Then, we have the
	%\begin{align*}
	%	f(\tz) = b(\tz)g(\tz) + \frac{1}{2}a(\tz)g^{2}(\tz).
	%\end{align*}
	\begin{align*}
		a(\tz)g^{2}(\tz) + 2b(\tz)g(\tz) - 2f(\tz) = 0.
	\end{align*}
	Solving the quadratic equation with respect to $g(\tz)$, we obtain 
	\begin{align}\label{eq: sol of quadratic eq}
		a(\tz)g(\tz) = -b(\tz) \pm \sqrt{ b^{2}(\tz) + 2f(\tz)a(\tz) }.
	\end{align}
	Since $m$ and $\widetilde{m}$ has a trivial bound $\eta^{-1}$, we find $g(\tz)\to0$ as $\eta\to\infty$. Note that
	\begin{align*}
		b(\tz) &= \tz + \sum_{n=1}^{\ell} 2n\mathcal{Z}_{n}\widetilde{m}^{2n-1}(\tz),\\
		a(\tz) &= 2 + 2(\mathcal{Z}_{n}-1) + \sum_{n=2}^{\ell} 2n(2n-1)\mathcal{Z}_{n}\widetilde{m}^{2n-2}(\tz).
	\end{align*}
	If $E$ is bounded and $\eta>0$ is large, we get $b(\tz)\asymp \eta$ for large $N$. We also have $a(\tz)\asymp 2$ for any $\eta>0$ if $N$ is large. Combining these estimates, we conclude that $g(\tz)\to0$ and $b^{2}(\tz)\gg f(\tz)a(\tz)$ as $\eta\to\infty$. Thus we should take $+$ sign in the equation (\ref{eq: sol of quadratic eq}):
	\begin{align*}
		a(\tz)g(\tz) = -b(\tz) + \sqrt{ b^{2}(\tz) + 2f(\tz)a(\tz) }.
	\end{align*}
	We claim that 
	\begin{align*}
		\Lambda \leq C\sqrt{|f(\tz)|}.
	\end{align*}
	If $b^{2}(\tz)\le |2f(\tz)a(\tz)|$, then $|g(\tz)|\le C(|b(\tz)|+\sqrt{4|f(\tz)a(\tz)|}) \le C\sqrt{|f(\tz)|}$. Otherwise, we can observe
	\begin{align*}
		\sqrt{ b^{2}(\tz) + 2f(\tz)a(\tz) } = b(\tz)\sqrt{1+\frac{2f(\tz)a(\tz)}{b^{2}(\tz)}} = b(\tz) \paren{ 1 + \bigO\paren{\frac{|f(\tz)a(\tz)|}{b^{2}(\tz)}} } \le b(\tz) + C\sqrt{|f(\tz)|},
	\end{align*}
	so the above claim follows. Furthermore, since $|R(g(\tz))|\prec \Lambda^{3}(\tz)/q < \Lambda^{2}(\tz)/q $ by the a priori bound $\Lambda \ll N^{-c}$ for large $N$, we can show $\Lambda \le C\sqrt{|P(\tz,m(\tz))|} + C\Lambda(\tz)/\sqrt{q}$ so it follows
	\begin{align}\label{eq: bound for local law}
		\Lambda \le C\sqrt{|P(\tz,m(\tz))|}.
	\end{align}
	We are now prepared to derive the local law from the recursive moment estimate, Proposition \ref{prop: RME}. Using Young's inequality, we have
	\begin{multline*}
		\expec{|P(\tz,m(\tz))|^{2r}} \prec \E\bigg[\paren{\frac{\Im m(\tz)}{N\eta}}^{2r} + \frac{1}{N^{2r}}  \\
		+ |\partial_{2}P|^{r}\paren{\frac{\Im m(\tz)}{N\eta}}^{3r/2} + \frac{1}{N^{r}}\paren{\frac{\Im m(\tz)}{N\eta}}^{r/2} + \paren{\frac{|\partial_{2}P|\Im m(\tz)}{N^{2}\eta^{2}}}^{r} + \frac{1}{N^{2r}\eta^{r}}\bigg].
	\end{multline*}
	Let $\boldsymbol{\phi}$ be as in \eqref{eq: bold phi}. It follows from Lemma \ref{lem: property of tilde m} that
	\begin{multline}\label{eq: RME and Young's ineq}
		\expec{|P(\tz,m(\tz))|^{2r}}
		\prec \E\bigg[ \paren{\frac{\boldsymbol{\phi}}{N\eta}}^{2r} + \frac{\Lambda^{2r}}{(N\eta)^{2r}} + \frac{1}{N^{2r}} + (\sqrt{\kappa+\eta})^{r}\paren{\frac{\boldsymbol{\phi}}{N\eta}}^{3r/2} + (\sqrt{\kappa+\eta})^{r}\frac{\Lambda^{3r/2}}{(N\eta)^{3r/2}}
		\\ 
		+ \frac{1}{N^{r}}\paren{\frac{\boldsymbol{\phi}}{N\eta}}^{r/2} + \frac{\Lambda^{r/2}}{N^{r}(N\eta)^{r/2}}
		+ (\sqrt{\kappa+\eta})^{r}\paren{\frac{\boldsymbol{\phi}}{N^{2}\eta^{2}}}^{r} + (\sqrt{\kappa+\eta})^{r}\frac{\Lambda^{r}}{(N\eta)^{2r}} + \frac{1}{N^{2r}\eta^{r}} \bigg].
	\end{multline}
	Since the estimate \eqref{eq: bound for local law} implies
	\begin{align*}
		\E\Lambda^{4r} \le C \expec{|P(\tz,m(\tz))|^{2r}},
	\end{align*}
	one of the following estimates should hold:
	\begin{align*}
		\E\Lambda^{4r} \prec \E\bigg[ \paren{\frac{\boldsymbol{\phi}}{N\eta}}^{2r} + \frac{1}{N^{2r}} + (\sqrt{\kappa+\eta})^{r}\paren{\frac{\boldsymbol{\phi}}{N\eta}}^{3r/2} + \frac{1}{N^{r}}\paren{\frac{\boldsymbol{\phi}}{N\eta}}^{r/2} + (\sqrt{\kappa+\eta})^{r}\paren{\frac{\boldsymbol{\phi}}{N^{2}\eta^{2}}}^{r} + \frac{1}{N^{2r}\eta^{r}} \bigg] ;
	\end{align*}
	\begin{align*}
		\E\Lambda^{2r} \prec \frac{1}{(N\eta)^{2r}} ; \quad \E\Lambda^{3r/2} \prec \frac{(\sqrt{\kappa+\eta})^{3r/5}}{(N\eta)^{9r/10}}; \quad \E\Lambda^{r/2} \prec \frac{1}{N^{r/7}(N\eta)^{r/14}}; \quad
		\E\Lambda^{r} \prec \frac{(\sqrt{\kappa+\eta})^{r/3}}{(N\eta)^{2r/3}}.
	\end{align*}
	This implies the local law:
	\begin{multline}\label{eq: local law bound}
		\Lambda \prec \paren{\frac{\boldsymbol{\phi}}{N\eta}}^{1/2} + (\sqrt{\kappa+\eta})^{1/4}\paren{\frac{\boldsymbol{\phi}}{N\eta}}^{3/8} + \frac{1}{N^{1/4}}\paren{\frac{\boldsymbol{\phi}}{N\eta}}^{1/8} + (\sqrt{\kappa+\eta})^{1/4}\paren{\frac{\boldsymbol{\phi}}{N^{2}\eta^{2}}}^{1/4} \\
		+ \frac{1}{N^{1/2}\eta^{1/4}} + \frac{1}{N\eta} + \frac{(\sqrt{\kappa+\eta})^{2/5}}{(N\eta)^{3/5}} + \frac{1}{N^{2/7}(N\eta)^{1/7}} + \frac{(\sqrt{\kappa+\eta})^{1/3}}{(N\eta)^{2/3}}.
	\end{multline}
\end{proof}

\begin{proof}[Proof of Theorem \ref{thm: rigidity}]                                                                                                                                                                                                                                                                                                                                                                                                                                                                                                                                                                                                                                                                                                                                                                                                                                                                                                                                                                                                                                                                                                                                                                                                                                                                                                                                                                                                                                                                                                                                                                                                                                                                                                                                                                                                                                                                                                                                                                                                                                                                                                                                                                                                                                                                                                                                                                                                                                                                                                                                                                                                                                                                                                                                                                                                                                                                                                                                                                                                                                                                                                                                                                                                                                                                                                                                                                                                                                                                                                                                                                                                                                                                                                                                                                                                                                                                                                                                                                                                                                                                                             
	We shall follow the proof of \cite[Theorem 1.4]{HLY20}. Let $c>0$ be a small positive real. Take $E \ge N^{-2/3+c}$ and $\eta=N^{-2/3}$. It follows from \eqref{eq: RME and Young's ineq} that
	\begin{multline*}
		\size{ P(\tz,m(\tz)) } \prec \frac{\boldsymbol{\phi}}{N\eta} + \frac{\Lambda}{N\eta} + (\sqrt{\kappa+\eta})^{1/2}\paren{\frac{\boldsymbol{\phi}}{N\eta}}^{3/4} + (\sqrt{\kappa+\eta})^{1/2}\frac{\Lambda^{3/4}}{(N\eta)^{3/4}} + \frac{1}{N^{1/2}}\paren{\frac{\phi}{N\eta}}^{1/4} \\
		+ \frac{\Lambda^{1/4}}{N^{1/2}(N\eta)^{1/4}} + (\sqrt{\kappa+\eta})^{1/2}\paren{\frac{\boldsymbol{\phi}}{N^{2}\eta^{2}}}^{1/2} + (\sqrt{\kappa+\eta})^{1/2}\frac{\Lambda^{1/2}}{N\eta} + \frac{1}{N\eta^{1/2}}.
	\end{multline*}
	% \tcr{Check a small modification of \cite[Lemma 2.13]{HLY20}.}\\
	Applying \cite[Lemma 2.13]{HLY20}, we get
	\begin{multline*}
		\Lambda \prec \frac{\boldsymbol{\phi}}{N\eta\sqrt{\kappa+\eta}} + \frac{\Lambda}{N\eta\sqrt{\kappa+\eta}} + (\sqrt{\kappa+\eta})^{-1/2}\paren{\frac{\boldsymbol{\phi}}{N\eta}}^{3/4} + (\sqrt{\kappa+\eta})^{-1/2}\frac{\Lambda^{3/4}}{(N\eta)^{3/4}} + \frac{1}{N^{1/2}\sqrt{\kappa+\eta}}\paren{\frac{\phi}{N\eta}}^{1/4} \\
		+ \frac{\Lambda^{1/4}}{N^{1/2}(N\eta)^{1/4}\sqrt{\kappa+\eta}} + (\sqrt{\kappa+\eta})^{-1/2}\paren{\frac{\boldsymbol{\phi}}{N^{2}\eta^{2}}}^{1/2} + (\sqrt{\kappa+\eta})^{-1/2}\frac{\Lambda^{1/2}}{N\eta} + \frac{1}{N\eta^{1/2}\sqrt{\kappa+\eta}},
	\end{multline*}
	and hence conclude with very high probability 
	\begin{align}\label{eq: imporved local law outside the spectrum}
		\size{m(\tz)} \ll \frac{1}{N\eta}.
	\end{align}
	Using the estimate \eqref{eq: imporved local law outside the spectrum}, we can show that there is no eigenvalue in the interval $[\widetilde{\mathcal{L}}+E-\eta,\widetilde{\mathcal{L}}+E+\eta]$ with very high probability. If there is an eigenvalue $\lambda_{i}$ in the small interval $[\widetilde{\mathcal{L}}+E-\eta,\widetilde{\mathcal{L}}+E+\eta]$, then it follows that
	\begin{align*}
		\Im m(\tz) \ge \frac{1}{N}\Im\paren{\frac{1}{\lambda_{i}-(\widetilde{\mathcal{L}}+E+i\eta)}} \ge \frac{1}{2N\eta},
	\end{align*}
	which is in contradiction to the bound \eqref{eq: imporved local law outside the spectrum}.
	Following the lattice argument and using a priori bound (e.g.~\cite[Lemma 4.4]{EKYY13} or \cite[Theorem 2.9]{LS18}),
	% $[\widetilde{\mathcal{L}}+E_{n}-\eta,\widetilde{\mathcal{L}}+E_{n}+\eta]$, $E_{n}=N^{2/3+c}+(n-1)N^{2/3}$.
	we have with very high probability
	\begin{align}\label{eq: upper bound for the largest eignevalue}
		\lambda_{1}-\widetilde{\mathcal{L}} \le N^{-2/3+c}.
	\end{align}
	In addition, using the standard argument of Helffer-Sj\H{o}strand calculus (e.g.~\cite[Lemma B.1]{ERSY10}), the following estimate holds with very high probability:
	\begin{align*}
		\mu\paren{[\widetilde{\mathcal{L}}-N^{-2/3+c},\infty)} \asymp N^{-1+2c/3},
	\end{align*} % \frac{1}{N}\size{\{i:\lambda_{i}\in[\widetilde{\mathcal{L}}-N^{-2/3+c},\infty)\}}
	where $\mu$ is the empirical eigenvalue distribution of $H$ given by \eqref{eq: ESD}.
	The desired result follows if $c>0$ is small enough.
\end{proof}

\appendix
\begin{appendices}

\section{Asymptotics of the random correction terms}\label{sec: correction term calculation}

We should check Assumption \ref{assump: random correction term 1} and Assumption \ref{assump: random correction term 2} for $\{\mathcal{Z}_{n}\}_{n=1}^{\ell}$ constructed in Section \ref{subsec: RME proof}. By the construction, it follows that
\begin{equation*}
\mathcal{Z}_{1} = \frac{1}{N}\sum_{i\neq j} h_{ij}^{2} .
\end{equation*}
We have
\begin{equation*}
\mathcal{Z}_{1} - 1 = \frac{1}{N}\sum_{i,j}\paren{h_{ij}^{2}-\expec{h_{ij}^{2}}} + \bigO_{\prec}\paren{\frac{1}{N}}.
\end{equation*}
According to \cite{HLY20}, we get
\begin{equation*}
\frac{1}{N}\sum_{i,j}\paren{h_{ij}^{2}-\expec{h_{ij}^{2}}} = \mathcal{X} \prec \frac{1}{\sqrt{N}q},
\end{equation*}
where recall \eqref{eq: mathcal X} for the definition of $\mathcal{X}$. Thus the estimate \eqref{eq: asymp 1} follows.

Now we want to check \eqref{eq: asymp 2} and \eqref{eq: asymp 3}. Since $\mathcal{Z}_{n}$ is a linear combination of terms of the form \eqref{eq: form of correction terms}, it is enough to check that for 
\begin{equation}\label{eq: s1}
\frac{1}{N^{\theta}}\sum_{\mathcal{I}_{k}}\prod_{l=1}^{k}A_{l}(s_{l}) \prec \frac{1}{q^{n-1}},
\end{equation}
and
\begin{align}\label{eq: s2}
\deri\paren{\frac{1}{N^{\theta}}\sum_{\mathcal{I}_{k}}\prod_{l=1}^{k}A_{l}(s_{l})} \prec \frac{1}{N},
\end{align}
where $s_{1},s_{2},\cdots, s_{k}$ are positive odd integers such that $s_{1}+\cdots+s_{k}+k= 2n$.

Let us consider
\begin{align*}
	\expec{\paren{\frac{1}{N}\sum_{i,j}h_{ij}^{2m}}^{r}}.  
\end{align*}
It follows that
\begin{align*}
	\frac{1}{N^{r}} \expec{\paren{\sum_{i,j}h_{ij}^{2m}}^{r}}
	& =
	\frac{1}{N^{r}}\sum_{l=1}^{r}\sum_{r_{1},\cdots,r_{l}}\sum_{(i_{1},j_{1})\neq\cdots\neq (i_{l},j_{l})}\expec{(h_{i_{1}j_{1}}^{2m})^{r_{1}}\cdots (h_{i_{l}j_{l}}^{2m})^{r_{l}}} \\
	& \le C \sum_{l=1}^{r} \combi{N^{2}}{l}\paren{\frac{1}{N^{r+l}q^{2mr-2l}}} \\
	& \le \frac{C}{q^{2r(m-1)}},
\end{align*}
where in the first line, the notation
$$\sum_{(i_{1},j_{1})\neq\cdots\neq (i_{l},j_{l})},$$
means the summation over all $l$ distinct index pairs $\{(i_{1},j_{1}),\cdots,(i_{l},j_{l})\}$. By Markov's inequality, we have for $m \ge 1$
\begin{align*}
	\frac{1}{N}\sum_{i,j}h_{ij}^{2m} \prec \frac{1}{q^{2(m-1)}}.
\end{align*}
Similarly, for $m\ge1$, we get
\begin{align}\label{eq: s3}
	\sum_{j}h_{ij}^{2m} \prec \frac{1}{q^{2(m-1)}}.
\end{align}
We also have
\begin{align*}
	\frac{1}{N}\sum_{i,j}\paren{h_{ij}^{2}-\expec{h_{ij}^{2}}} \prec \frac{1}{\sqrt{N}q},
\end{align*}
and
\begin{align*}
	\sum_{j}\paren{h_{ij}^{2}-\expec{h_{ij}^{2}}} \prec \frac{1}{q}.
\end{align*}
Thus we can see that
\begin{align*}
	\frac{1}{N}\sum_{i_{1},j_{1}}h_{i_{1}j_{1}}^{s_{1}+1}\sum_{j_{2}}h_{i_{1}j_{2}}^{s_{2}+1}\sum_{i_{3}}h_{i_{3}j_{1}}^{s_{3}+1} \prec \frac{1}{q^{s_{1}+s_{2}+s_{3}-3}},
\end{align*}
where $s_{1},s_{2},s_{3}$ are positive odd integer. Extending the above argument, we can obtain \eqref{eq: s1} using the condition that $s_{1},s_{2},\cdots, s_{k}$ are positive odd integers such that $s_{1}+\cdots+s_{k}+k= 2n$.

By the product rule, we have
\begin{multline*}
\deri\paren{\frac{1}{N^{\theta}}\sum_{\mathcal{I}_{k}}\prod_{l=1}^{k}A_{l}(s_{l})} \\
= \frac{1}{N^{\theta}}\sum_{\mathcal{I}_{k}}\big( \deri(A_{1}(s_{1}))A_{2}(s_{2})\cdots A_{k}(s_{k}) + A_{1}(s_{1})\deri(A_{2}(s_{2}))A_{3}(s_{3})\cdots A_{k}(s_{k}) \\
+ \cdots + A_{1}(s_{1})A_{2}(s_{2})\cdots A_{k-1}(s_{k-1})\deri(A_{k}(s_{k})) \big).
\end{multline*}
Thus, it is enough to show for $1\le l\le k$
\begin{equation*}
\frac{1}{N^{\theta}}\sum_{\mathcal{I}_{k}} A_{1}(s_{1})\cdots A_{l-1}(s_{l-1}) \deri(A_{l}(s_{l})) A_{l+1}(s_{l+1})\cdots A_{k}(s_{k}) \prec \frac{1}{N}
\end{equation*}
We claim that
\begin{equation*}
\frac{1}{N}\sum_{i_{l},j_{l}}\deri(A_{l}(s_{l})) \prec \frac{1}{N}.
\end{equation*}
We first observe that
\begin{equation*}
\deri(A_{l}(s_{l})) =
\begin{cases*}
(s_{l}+1)h_{i_{l}j_{l}}^{s_{l}} & $\{i_{l}, j_{l}\}=\{i,j\}$, \\
0 & otherwise.
\end{cases*}
\end{equation*}
Then, it follows that
\begin{equation*}
\frac{1}{N}\sum_{i_{l},j_{l}}\deri(A_{l}(s_{l})) = \frac{2(s_{l}+1)}{N}h_{ij}^{s_{l}}.
\end{equation*}
Since we have $h_{ij}\prec q^{-1}$ due to the given condition in Definition \ref{def: sparse RM}, the claim is proved. Consider an index pair $(i_{l'},j_{l'})$ such that $1\le l' \le k$ and $l'\neq l$. Due to \eqref{eq: s3}, if $i_{l'} \equiv i_{l}$, then we have
\begin{equation*}
\sum_{j_{l'}}h_{i_{l'}i_{l'}}^{s_{l'}+1} \prec \frac{1}{q^{s_{l'}-1}}.
\end{equation*}
Similarly, if $j_{l'} \equiv j_{l}$, we get
\begin{equation*}
\sum_{i_{l'}}h_{i_{l'}i_{l'}}^{s_{l'}+1} \prec \frac{1}{q^{s_{l'}-1}}.
\end{equation*}
Combining the above observations with \eqref{eq: s1}, we can deduce that
\begin{equation*}
\frac{1}{N^{\theta}}\sum_{\mathcal{I}_{k}} A_{1}(s_{1})\cdots A_{l-1}(s_{l-1}) \deri(A_{l}(s_{l})) A_{l+1}(s_{l+1})\cdots A_{k}(s_{k}) \prec \frac{1}{N}.
\end{equation*}
Therefore \eqref{eq: s2} is also proved.

\section{Properties of a solution of the self-consistent equation}\label{sec: properties of tilde m}
\begin{proof}[Proof of Lemma \ref{lem: property of tilde m}]
	We shall use the argument in \cite[Proposition 2.5, Proposition 2.6]{HLY20} and \cite[Lemma 4.1]{LS18}. Define
	\begin{align*}
		R(w) = -\frac{1}{w} - w - \frac{Q(w)-w^{2}}{w} = -\frac{1}{w} - w - (\mathcal{Z}_{1}-1)w - \sum_{n=2}^{\ell}\mathcal{Z}_{n}w^{2n-1}.
	\end{align*}
	Observe that $P(z,w)=0$ if and only if $R(w)=z$. Taking derivative, we have
	\begin{align*}
		R'(w) = \frac{1}{w^{2}} - 1 - (\mathcal{Z}_{1}-1) - \sum_{n=2}^{\ell}(2n-1)\mathcal{Z}_{n}w^{2n-2},
	\end{align*}
	and
	\begin{align*}
		R''(w) = -\frac{2}{w^{3}} - \sum_{n=2}^{\ell} (2n-1)(2n-2)\mathcal{Z}_{n}w^{2n-3}.
	\end{align*}
	Note that
	\begin{align*}
		\mathcal{Z}_{1}-1 \prec \frac{1}{\sqrt{N}q},
	\end{align*}
	and for $n\ge 2$
	\begin{align*}
		\mathcal{Z}_{n} \prec \frac{1}{q^{n-1}}.
	\end{align*}
	Following the argument in \cite[Appendix B]{HLY20}, we notice that $R'(w)$ has two solutions on $(-C,C)$ if $N$ is large enough. Denote these solutions by $\pm\tau$. From now on, we shall construct $\widetilde{\mathcal{L}}$. We write $w=1-\eps$. Then we can see that
	\begin{align*}
		R'(w) = (1+\eps+\eps^{2}+\cdots)^{2} - 1 - (\mathcal{Z}_{1}-1) -  \sum_{n=2}^{\ell}(2n-1)\mathcal{Z}_{n}(1-\eps)^{2n-2}.
	\end{align*}
	We set
	\begin{align*}
		\eps_{0} \coloneqq \frac{1}{2}\paren{(\mathcal{Z}_{1}-1)+\sum_{n=2}^{\ell}(2n-1)\mathcal{Z}_{n}}.
	\end{align*}
	Next we write $w=1-(\eps_{0}+\eps)$. It follows that
	\begin{align*}
		R'(w) & = \paren{1+(\eps_{0}+\eps)+(\eps_{0}+\eps)^{2}+\cdots}^{2} - 1 - (\mathcal{Z}_{1}-1) -  \sum_{n=2}^{\ell}(2n-1)\mathcal{Z}_{n}\paren{1-(\eps_{0}+\eps)}^{2n-2} \\
		& = \paren{1+(\eps_{0}+\eps)+(\eps_{0}+\eps)^{2}+\cdots}^{2} - 1 - 2\eps_{0} - E\paren{(\mathcal{Z}_{n}), \eps_{0}} - \text{(remainder)}.
	\end{align*}
	where in the second line $E\paren{(\mathcal{Z}_{n}), \eps_{0}}$ is given by
	%the sum of the terms, depending only $(\mathcal{Z}_{n})$ and $\eps_{0}$, in the following term of the first line,
	\begin{align*}
		E\paren{(\mathcal{Z}_{n}), \eps_{0}} \coloneqq \sum_{n=2}^{\ell}(2n-1)\mathcal{Z}_{n}\paren{\paren{1-\eps_{0}}^{2n-2}-1}.
	\end{align*}
	To make cancellation, we set
	\begin{align*}
		\eps_{1} \coloneqq \frac{1}{2}E\paren{(\mathcal{Z}_{n}), \eps_{0}} - (\eps_{0}^{2} + \eps_{0}^{3} + \cdots).
	\end{align*}
	%    Similarly,
	%    \begin{align*}
	%    	\eps_{2} \coloneqq \frac{1}{2}E\paren{(\mathcal{Z}_{n}), \eps_{0}, \eps_{1}} - \left\{ \paren{(\eps_{0}+\eps_{1})^{2}-\eps_{0}^{2}} + \paren{(\eps_{0}+\eps_{1})^{3}-\eps_{0}^{3}} + \cdots \right\},
	%    \end{align*}
	%    where
	%    \begin{align*}
	%    	E\paren{(\mathcal{Z}_{n}), \eps_{0}, \eps_{1}} \coloneqq \sum_{n=2}^{\ell}(2n-1)\mathcal{Z}_{n}\paren{ (1-\eps_{0}-\eps_{1})^{2n-2} - (1-\eps_{0})^{2n-2} }.
	%    \end{align*}
	Repeating the above argument similarly, we can make a sequence $\{\eps_{m}\}$ such that $\tau \asymp 1 - \sum_{m}\eps_{m} + \bigO_{\prec}(N^{-C})$, $\eps_{m}\gg \eps_{m+1}$ and each $\eps_{m}$ is a polynomial in variables $(\mathcal{Z}_{n})$. Let us define $\widetilde{L}$ by setting 
	\begin{align*}
		\widetilde{L} \coloneqq R(-\tau).
	\end{align*}
	Using $\tau \asymp 1 - \sum_{m}\eps_{m} + \bigO_{\prec}(N^{-C})$, we obtain
	\begin{align*}
		\widetilde{L} = \asedge + \bigO_{\prec}(N^{-C}).
	\end{align*}
	where $\asedge$ is a polynomial in variables $(\mathcal{Z}_{n})$. We define
	\begin{align*}
		\mathcal{D}_{w}\coloneqq\{w\in\C : |w| < 5 \},
	\end{align*}
	and
	\begin{align*}
		\mathcal{D}_{z}\coloneqq\{z=E+i\eta : |E|,|\eta|\le 2\sqrt{2} \}.
	\end{align*}
	According to \cite[Appendix B]{HLY20}, by Rouch\'{e}'s theorem, we find that $R(w)$ has exactly two critical points $\pm\tau$ on $\mathcal{D}_{w}$. Again, by Rouch\'{e}'s theorem, if $z\in\mathcal{D}_{z}$, the equation $P(z,w)=0$ has exactly two solutions on $\mathcal{D}_{w}$. Furthermore, if $z\in(-\widetilde{L}, \widetilde{L})$, one solution of $P(z,w)=0$ is on $\C_{+}$ and the other is on $\C_{+}$. Considering $z\in(-\widetilde{L}, \widetilde{L})$, the solution $w(z)\in\C_{+}$ of $P(z,w)=0$ forms a curve on $\C_{+}$, joining $\tau$ and $-\tau$. We denote this curve by $\Gamma$.
	
	Consider the region $\mathcal{D}_{\Gamma}$, bounded by the curve $\Gamma$ and the interval $[-\widetilde{L},\widetilde{L}]$. We can see that $R(w)$ is biholomorphic from the region $\mathcal{D}_{\Gamma}$ to the upper half-plane by using maximum principle and considering one-point compactification of the complex plane. This let us take $\widetilde{m}$ to be the inverse of $R(w)$, i.e.,
	\begin{align*}
		\widetilde{m}(z) \coloneqq R^{-1}(z), \quad z\in\C_{+}.
	\end{align*}
	Let $\widetilde{\rho}$ be the probability measure obtained by Stieltjes inversion of $\widetilde{m}(z)$. The probability measure $\widetilde{\rho}$ is supported on $[-\widetilde{L},\widetilde{L}]$ and has strictly positive density on $(-\widetilde{L},\widetilde{L})$. Considering Taylor expansion of $R(w)$ in a small neighborhood of $\pm\tau$, we can show $\widetilde{\rho}$ has square root behavior at the edges $\pm \widetilde{L}$. The asymptotics of derivatives of $P(z,\widetilde{m}(z))$ easily can be checked as in \cite[Appendix B]{HLY20}.
\end{proof}

\section{Proof of Lemma \ref{lem: moment estimate general}}\label{appen: lemma proof}

\begin{proof}[Proof of Lemma \ref{lem: moment estimate general}]
	Consider the left-hand side of \eqref{eq: t1}
	\begin{equation*}
	\sum_{p=1}^{\ell}\frac{\mathcal{C}_{p+1}}{N^{2}q^{p-1}}\sum_{x,y}\expec{ \partial_{xy}^{p}\paren{ m^{d}G_{yx}G_{ii}^{t} \bigg(\prod_{j=1}^{k} G_{v_{j}v_{j}}^{u_{j}}\bigg) D(P) } }.
	\end{equation*}
	Due to Proposition \ref{prop: bounds for derivatives} and \cite[Proposition A.1]{HLY20}, we can replace $\sum_{x,y}$ with $\sum_{x\neq y}$ as follows:
	\begin{multline*}
	\sum_{p=1}^{\ell}\frac{\mathcal{C}_{p+1}}{N^{2}q^{p-1}}\sum_{x,y}\expec{ \partial_{xy}^{p}\paren{ m^{d}G_{yx}G_{ii}^{t} \bigg(\prod_{j=1}^{k} G_{v_{j}v_{j}}^{u_{j}}\bigg) D(P) } } \\
	= \sum_{p=1}^{\ell}\frac{\mathcal{C}_{p+1}}{N^{2}q^{p-1}}\sum_{x\neq y}\expec{ \partial_{xy}^{p}\paren{ m^{d}G_{yx}G_{ii}^{t} \bigg(\prod_{j=1}^{k} G_{v_{j}v_{j}}^{u_{j}}\bigg) D(P) } } + \bigO_{\prec}\paren{ \Phi_{r} }.
	\end{multline*}
	Note that
	\begin{multline}\label{eq: tt1}
	\frac{\mathcal{C}_{p+1}}{N^{2}q^{p-1}}\sum_{x\neq y}\expec{ \partial_{xy}^{p}\paren{ m^{d}G_{yx}G_{ii}^{t} \bigg(\prod_{j=1}^{k} G_{v_{j}v_{j}}^{u_{j}}\bigg) D(P) } } \\
	= \sum_{s=0}^{p}\combi{p}{s}\frac{\mathcal{C}_{p+1}}{N^{2}q^{p-1}}\sum_{x\neq y}
	\expec{ \paren{ \partial_{xy}^{s} G_{yx} } \partial_{xy}^{p-s}\paren{ m^{d}G_{ii}^{t} \bigg(\prod_{j=1}^{k} G_{v_{j}v_{j}}^{u_{j}}\bigg) D(P) } } \\
	= \sum_{s=0}^{p}\combi{p}{s}\frac{\mathcal{C}_{p+1}}{N^{2}q^{p-1}}\sum_{x\neq y}
	\expec{ \paren{ D_{xy}^{s} G_{yx} } \partial_{xy}^{p-s}\paren{ m^{d}G_{ii}^{t} \bigg(\prod_{j=1}^{k} G_{v_{j}v_{j}}^{u_{j}}\bigg) D(P) } } + \bigO_{\prec}\paren{ \Phi_{r} },
	\end{multline}
	where we use \eqref{eq: derivative G_ij} for the second equality. If the derivative $\partial_{xy}$ hits $m$, $G_{ii}$ or $G_{v_{j}v_{j}}$, then the resulting terms are absorbed into $\bigO_{\prec}\paren{ \Phi_{r} }$ due to \eqref{eq: derivative m} and \eqref{eq: ward id}, which implies
	\begin{multline}\label{eq: tt2}
	\sum_{s=0}^{p}\combi{p}{s}\frac{\mathcal{C}_{p+1}}{N^{2}q^{p-1}}\sum_{x\neq y}
	\expec{ \paren{ D_{xy}^{s} G_{yx} } \partial_{xy}^{p-s}\paren{ m^{d}G_{ii}^{t} \bigg(\prod_{j=1}^{k} G_{v_{j}v_{j}}^{u_{j}}\bigg) D(P) } } \\
	= \sum_{s=0}^{p}\combi{p}{s}\frac{\mathcal{C}_{p+1}}{N^{2}q^{p-1}}\sum_{x\neq y}
	\expec{ \paren{ D_{xy}^{s} G_{yx} } m^{d}G_{ii}^{t} \bigg(\prod_{j=1}^{k} G_{v_{j}v_{j}}^{u_{j}}\bigg) \paren{ \partial_{xy}^{p-s} D(P) } }
	\end{multline}
	If $s$ is odd, we observe that
	\begin{equation*}
	D_{xy}^{s}G_{xy} = -(s!)G_{xx}^{\frac{s+1}{2}}G_{yy}^{\frac{s+1}{2}} + (\text{the terms having at least two off-diagonal entries}).
	\end{equation*}
	Thus, due to \eqref{eq: ward id}, we have
	\begin{multline*}
	\sum_{s=0}^{p}\combi{p}{s}\frac{\mathcal{C}_{p+1}}{N^{2}q^{p-1}}\sum_{x\neq y}
	\expec{ \paren{ D_{xy}^{s} G_{yx} } m^{d}G_{ii}^{t} \bigg(\prod_{j=1}^{k} G_{v_{j}v_{j}}^{u_{j}}\bigg) \paren{ \partial_{xy}^{p-s} D(P) } } \\
	= -\sum_{s=0}^{p}\combi{p}{s}\frac{(s!)\mathcal{C}_{p+1}}{N^{2}q^{p-1}}\sum_{x\neq y} \expec{ G_{xx}^{\frac{s+1}{2}}G_{yy}^{\frac{s+1}{2}} m^{d}G_{ii}^{t} \bigg(\prod_{j=1}^{k} G_{v_{j}v_{j}}^{u_{j}}\bigg) \paren{ \partial_{xy}^{p-s} D(P) } }, \quad s\equiv1(\Mod2).
	\end{multline*}
	
	If $s$ is even, we can see that
	\begin{equation*}
	D_{xy}^{s}G_{xy} = t_{s}G_{xy}G_{xx}^{s/2}G_{yy}^{s/2} + (\text{the terms having at least three off-diagonal entries}),
	\end{equation*}
	where $t_{s}$ is a constant depending on $s$. Since any term having at least two off-diagonal entries is absorbed into $\bigO_{\prec}\paren{ \Phi_{r} }$, it is enough to consider
	\begin{equation*}
	\sum_{s=0}^{p}\combi{p}{s}\frac{\mathcal{C}_{p+1}}{N^{2}q^{p-1}}\sum_{x\neq y}
	\expec{ G_{xy}G_{xx}^{s/2}G_{yy}^{s/2} m^{d}G_{ii}^{t} \bigg(\prod_{j=1}^{k} G_{v_{j}v_{j}}^{u_{j}}\bigg) \paren{ \partial_{xy}^{p-s} D(P) } }.
	\end{equation*}
	Following the argument from \eqref{eq: ss1} and \eqref{eq: ss2}, we get
	\begin{equation*}
	\sum_{s=0}^{p}\combi{p}{s}\frac{\mathcal{C}_{p+1}}{N^{2}q^{p-1}}\sum_{x\neq y}
	\expec{ G_{xy}G_{xx}^{s/2}G_{yy}^{s/2} m^{d}G_{ii}^{t} \bigg(\prod_{j=1}^{k} G_{v_{j}v_{j}}^{u_{j}}\bigg) \paren{ \partial_{xy}^{p-s} D(P) } } = \bigO_{\prec}\paren{ \Phi_{r} }.
	\end{equation*}
	Thus, we have \eqref{eq: t1}.
	
	Next we consider the left-hand side of \eqref{eq: t2}
	\begin{equation*}
	\sum_{p=1}^{\ell}\frac{\mathcal{C}_{p+1}}{Nq^{p-1}}\sum_{x}\expec{ \partial_{ix}^{p}\paren{ m^{d+1}G_{xi}G_{ii}^{t-1} \bigg(\prod_{j=1}^{k} G_{v_{j}v_{j}}^{u_{j}}\bigg) D(P) } }.
	\end{equation*}
	As in \eqref{eq: tt1} and \eqref{eq: tt2}, using \eqref{eq: derivative G_ij}, \eqref{eq: derivative m} and \eqref{eq: ward id}, we have
	\begin{multline*}
	\frac{\mathcal{C}_{p+1}}{Nq^{p-1}}\sum_{x}\expec{ \partial_{ix}^{p}\paren{ m^{d+1}G_{xi}G_{ii}^{t-1} \bigg(\prod_{j=1}^{k} G_{v_{j}v_{j}}^{u_{j}}\bigg) D(P) } } \\
	= \sum_{s=0}^{p}\combi{p}{s}\frac{\mathcal{C}_{p+1}}{N^{2}q^{p-1}}\sum_{x\neq y} \expec{ D_{ix}^{s}( G_{xi}G_{ii}^{t-1}) m^{d+1}\bigg(\prod_{j=1}^{k} G_{v_{j}v_{j}}^{u_{j}}\bigg) \paren{ \partial_{ix}^{p-s} D(P) } }.
	\end{multline*}
	As in the previous case, if $s$ is even, it is absorbed into $\bigO_{\prec}\paren{ \Phi_{r} }$ by the same argument from \eqref{eq: ss1} and \eqref{eq: ss2}. Thus, we focus on the case $s$ is odd. If $s$ is odd, we observe that
	\begin{equation*}
	D_{ix}^{s}( G_{xi}G_{ii}^{t-1}) = - c_{s}\times (s!) G_{xx}^{\frac{s+1}{2}}G_{ii}^{\frac{s+1}{2}}G_{ii}^{t-1} + (\text{the terms having at least two off-diagonal entries}).
	\end{equation*}
	where $c_{s}$ is a constant depending on $s$. This implies \eqref{eq: t2} and, in fact, $c_{s}$ satisfies \eqref{eq: t3}.
\end{proof}

\end{appendices}

\section*{Acknowledgments.}
This research was motivated by discussions with Charles Bordenave. The author thanks Yukun He, Paul Jung and Ji Oon Lee for their helpful comments and suggestions. The author is also grateful to the referees for their careful reading of the manuscript and many helpful suggestions.
This work was supported in part by the National Research Foundation of Korea (NRF-2017R1A2B2001952, NRF-2019R1A5A1028324) and the
Hong Kong Research Grants Council (GRF-16301519, GRF-16301520).

\bibliographystyle{plain}

\end{document}